\documentclass[a4paper,10pt,reqno]{amsart}

\usepackage[english]{babel}
\usepackage[utf8]{inputenc}
\usepackage{amsmath,amssymb,amsfonts,amsthm,amscd}
\usepackage[mathscr]{eucal}
\usepackage{mathtools}

\usepackage{enumitem}
\setlist[itemize]{topsep=0pt,leftmargin=2em}
\setlist[enumerate]{topsep=0pt,leftmargin=2em}

\usepackage{tikz-cd}
\usepackage{hyperref}
\usepackage{booktabs,subcaption,dcolumn,multirow}
\newcolumntype{d}[1]{D..{#1}}

\let\oldtocsection=\tocsection
\let\oldtocsubsection=\tocsubsection

\renewcommand{\tocsection}[2]{\hspace{0em}\oldtocsection{#1}{#2}}
\renewcommand{\tocsubsection}[2]{\hspace{1em}\oldtocsubsection{#1}{#2}}

\newcommand{\R}{\mathbb{R}}
\newcommand{\C}{\mathbb{C}}

\newcommand{\df}{\mathrm{d}}
\newcommand{\Nil}{\mathrm{Nil}_3}
\DeclareMathOperator{\Iso}{Iso}

\DeclareMathOperator{\arctanh}{arctanh}
\DeclareMathOperator{\SO}{SO}
\DeclareMathOperator{\OO}{O}
\DeclareMathOperator{\SU}{SU}
\DeclareMathOperator{\SL}{SL}
\DeclareMathOperator{\trace}{tr}
\DeclareMathOperator{\adj}{adj}
\DeclareMathOperator{\Ric}{Ric}
\DeclareMathOperator{\Rot}{Rot}
\DeclareMathOperator{\Stab}{Stab}

\newcommand{\E}{\mathbb{E}}

\newcommand{\X}{\mathfrak{X}}
\newcommand{\id}{\mathrm{id}}

\newtheorem{theorem}{Theorem}[section]
\newtheorem{proposition}[theorem]{Proposition}
\newtheorem{corollary}[theorem]{Corollary}
\newtheorem{lemma}[theorem]{Lemma}

\theoremstyle{definition}
\newtheorem{definition}{Definition}[section]

\theoremstyle{remark}
\newtheorem{remark}[theorem]{Remark}
\newtheorem{example}[theorem]{Example}

\numberwithin{equation}{section}

\title[Isometric immersions into unimodular metric Lie groups]{Isometric immersions into three-dimensional unimodular metric Lie groups}
\date{}

\author{Ildefonso Castro}
\address{Departamento de Matem\'aticas, Universidad de Ja\'en, 23071 Ja\'en SPAIN}
\email{icastro@ujaen.es}

\author{Jos\'e M. Manzano}
\address{Departamento de Matem\'aticas, Universidad de Ja\'en, 23071 Ja\'en SPAIN}
\email{jmprego@ujaen.es}

\author{Jos\'e S. Santiago}
\address{Departamento de Matem\'aticas, Universidad de Ja\'en, 23071 Ja\'en SPAIN}
\email{jssantia@ujaen.es}



\begin{document}
	
	\begin{abstract}
	We study isometric immersions of surfaces into simply connected $3$-dimensional unimodular Lie groups endowed with either Riemannian or Lorentzian left-invariant metrics, assuming that Milnor's operator is diagonalizable in the Lorentzian case. We provide global models in coordinates for all these \emph{metric Lie groups} that depend analytically on the structure constants and establish some fundamental theorems characterizing such immersions. In this sense, we study up to what extent we can recover the immersion from (a) the tangent projections of the natural left-invariant ambient frame, (b) the left-invariant Gauss map, and (c) the shape operator. In particular, we show that an isometric immersion is determined by its left-invariant Gauss map up to certain well controlled \emph{angular companions}. We also classify totally geodesic surfaces and introduce four Lorentzian analogues of the Daniel correspondence within two families of Lorentzian homogeneous $3$-manifolds with $4$-dimensional isometry group. We also classify isometric immersions in $\mathbb{R}^3$ or $\mathbb{S}^3$ whose left-invariant Gauss maps differ by a direct isometry of $\mathbb{S}^2$. Finally, we show that Daniel's is the furthest extension of the classical Lawson correspondence for constant mean curvature surfaces in Riemannian unimodular metric Lie groups. 
	\end{abstract}
	
	\maketitle

\tableofcontents

\section{Introduction}

It is a foundational problem in classical submanifold theory to decide if a given Riemannian manifold $M$ admits an isometric immersion $\phi:M\to\overline{M}$ into another Riemannian manifold $\overline{M}$ of higher dimension. The usual way of addressing this question is by finding geometric \emph{fundamental data} that $M$ inherits from $\overline{M}$ via $\phi$, with the requirement that the existence of the immersion only depends on whether or not these data satisfy some differential equations on $M$, called \emph{fundamental equations}. The earliest result of the kind dates back to Bonnet, who proved that the intrinsic metric and the shape operator (i.e., the first and second fundamental forms) of a surface immersed in Euclidean space $\R^3$ are enough to recover the immersion up to rigid motions, being Gauss and Codazzi equations the fundamental equations in this case. This was later generalized to immersions with arbitrary codimension in ambient spaces of constant curvature, namely $\R^n$, $\mathbb{S}^n$, and $\mathbb{H}^n$ (e.g., see~\cite[\S7.C]{Spivak4} and references therein). 

In this work, we will focus on the case of surfaces isometrically immersed in $3$-manifolds. Apart from the already mentioned space forms, the problem is naturally well posed in homogeneous ambient spaces, since in that case the fundamental equations are autonomous. Building upon the previous work of Tenenblat~\cite{Ten71} in $\mathbb{R}^n$, Daniel~\cite{Daniel07,Daniel09} found fundamental equations for isometric immersions in $\mathbb{E}(\kappa,\tau)$-spaces, the most symmetric ambient $3$-manifolds after the space forms, being their isometry group $4$-dimensional. This family $\mathbb{E}(\kappa,\tau)$ depends on two parameters $\kappa,\tau\in\mathbb{R}$ with $\kappa-4\tau^2\neq 0$ and includes the Thurston geometries $\mathbb{H}^2\times\R$, $\mathbb{S}^2\times\R$, $\Nil$ and $\widetilde{\mathrm{SL}}_2(\R)$, plus the Berger spheres. The key difference with the case of space forms is that the shape operator does not seem to be sufficient to recover the immersion into $\mathbb{E}(\kappa,\tau)$; it is also necessary to prescribe the tangent component of the distinguished unit Killing field, together with the associated angle function. 

Milnor~\cite{Milnor} showed that every simply connected homogeneous Riemannian space, except for the product $\mathbb{S}^2(\kappa)\times\mathbb{R}$, is isometric to a \emph{metric Lie group}, that is, a simply connected 3-dimensional Lie group equipped with a left-invariant metric. We also refer to the work of Meeks--Pérez~\cite{MeeksPerez12} for a thorough survey on the geometry of metric Lie groups with applications to constant mean curvature (\textsc{cmc} for short) surfaces. These Riemannian metric Lie groups are split into two $3$-parameter families depending on whether the underlying Lie group is unimodular or not. The family $\mathbb{E}(\kappa,\tau)$ with $\tau\neq 0$ lies naturally within the unimodular case, although $\mathbb{E}(\kappa,\tau)$ also admits a non-unimodular structure when $\kappa<0$. There have been relevant recent works in surface theory in metric Lie groups, such as the classification of spheres with \textsc{cmc} by Meeks--Mira--Pérez--Ros~\cite{MMPR21,MMPR22}, the classification of polar actions by Domínguez-Vázquez--Ferreira--Otero~\cite{DFT25} or the classification of totally umbilical surfaces by Souam and the second author~\cite{MS}.

Our aim is to consolidate some foundations for a theory of isometric immersions into unimodular metric Lie groups. Although this family of ambient spaces depends on three structure constants $c_1,c_2,c_3\in\R$, we will also allow the left-invariant metric to be semi-Riemannian in order to obtain parallel results for Riemannian and Lorentzian homogeneous 3-manifolds, considering both spacelike and timelike surfaces in the latter. However, the classification of Lorentzian metric Lie groups is more involved than in the Riemannian case, as shown by Rahmani~\cite{Rahmani} (this work was later used by Calvaruso~\cite{Calvaruso06} to classify all homogeneous Lorentzian $3$-manifolds). For each choice of signs $\epsilon_1,\epsilon_2,\epsilon_3\in\{-1,1\}$, we will have a $3$-parameter family of semi-Riemannian $3$-manifolds $G$ which are algebraically similar to those in the Riemannian case. They will be called of \emph{diagonalizable type}, for there is a global orthonormal left-invariant frame $\{E_1,E_2,E_3\}$ in $G$ such that
\[[E_1,E_2]=c_3E_3,\qquad [E_2,E_3]=c_1E_1,\qquad [E_3,E_1]=c_2E_2,\]
where $\epsilon_i=\langle E_i,E_i\rangle$. These vector fields diagonalize the Ricci operator, see the proof of Proposition~\ref{prop:dim-iso-4-6}.

This family includes the Riemannian space forms $\R^3$ and $\mathbb{S}^3$, and in parallel the Minkowski space $\mathbb{L}^3$ and the universal cover of the Anti-de Sitter space $\mathbb{H}^3_1$. It also contains the Lorentzian version $\mathbb{L}(\kappa,\tau)$ of the family $\mathbb{E}(\kappa,\tau)$, see~\cite{Lee}. It is well known that both $\mathbb{E}(\kappa,\tau)$ and $\mathbb{L}(\kappa,\tau)$ are unit Killing submersions with constant bundle curvature $\tau$ over the complete Riemannian surface $\mathbb{M}^2(\kappa)$ of constant curvature $\kappa$, see also~\cite{Man14,DLM}. Interestingly, among Lorentzian metric Lie groups of diagonalizable type, we can find another new family $\widehat{\mathbb{L}}(\kappa,\tau)$ of spaces that admit a semi-Riemannian submersion with constant bundle curvature $\tau$ over the complete Lorentzian surface $\mathbb{M}^2_1(\kappa)$ with constant curvature $\kappa$. Indeed, $\widehat{\mathbb{L}}(\kappa,\tau)$ also admits a unitary (spacelike) Killing vector field tangent to the fibers of the submersion. After the submission of this work, we found this new family of spaces was discovered independently by Calvaruso and Pellegrino~\cite{CalPel25}. Note also that there are other Lorentzian metric Lie groups with isometry group of dimension 4, see~\cite[Rmk.~2.5]{CalPel25}.

In Section~\ref{sec:lie-groups}, we will discuss the preliminaries on metric Lie groups of diagonalizable type and classify those with isometry group of dimension $4$ or $6$ (see Proposition~\ref{prop:dim-iso-4-6}). Also, we give explicit metrics in coordinates for all these semi-Riemannian $3$-manifolds depending analytically on the structure constants $c_1,c_2,c_3$. This extends Daniel's models in~\cite{Daniel07} for $\mathbb{E}(\kappa,\tau)$, as well as the standard metrics given by Milnor~\cite{Milnor} in case the Lie group is a semidirect product $\R^2\rtimes_A\R$. Other than these, the Lie group is either $\SU(2)$ or $\widetilde{\SL}_2(\R)$, in which case we provide explicit isometries or cover maps to the usual matrix representations of these classical groups. Our models fail to be global only if the underlying Lie group is $\SU(2)\cong\mathbb{S}^3$.

Section~\ref{sec:fundamental} is devoted to the main result. Given a non-degenerate surface $\Sigma$ and a metric Lie group $G$, we begin by defining the fundamental data of an isometric immersion $\phi:\Sigma\to G$. They consist of the shape operator $S$ and the orientation $J$ induced by a unit normal $N$, the tangent projections $T_i=E_i^\top$ over $\Sigma$ and the angle functions $\nu_i=\langle E_i,N\rangle$. As compatibility equations, we have Gauss and Codazzi equations, some algebraic restrictions, and equations for the derivatives of $T_i$ and $\nu_i$, see Proposition~\ref{prop:compatibility}. However, this set of equations turns out to be very redundant and our first result will show that, given $T_1,T_2,T_3$ and $J$, one can define univocally $S$ and $\nu_1,\nu_2,\nu_3$, and the only needed equations are those for the derivatives of $T_1,T_2,T_3$, see the Theorem~\ref{thm:fundamental}. In case the immersion exists, it is unique up to left-translations in $G$. The proof of this result is inspired by Daniel's approach in~\cite{Daniel07,Daniel09}, which uses Cartan's moving frames. However, in order to get the immersion from just $T_1,T_2,T_3$ we only need to use Frobënius theorem once, not twice as we shall explain later. Piccione--Tausk~\cite{PicTau08} also present a related approach to problems of the same kind using $G$-structures but their work is not specialized for our type of fundamental data. To be more precise, our work mainly deals with finding a minimal subset of the fundamental data one has to prescribe in order to recover the immersion. Furthermore, we present three applications of Theorem~\ref{thm:fundamental}.

First, it is natural to ask whether or not the left-invariant Gauss map $g=(\nu_1,\nu_2,\nu_3)$ determines the isometric immersion. Actually, we will be always prescribing the metric and orientation $J$ in $\Sigma$, so we are asking if we can also prescribe $g$. There have been other approaches to this problem in the literature, e.g., the Weierstrass-type representation for \textsc{cmc} surfaces by Meeks--Mira--Pérez--Ros~\cite{MMPR22} or another representation assuming positivity of the Gauss curvature by Folha--Peñafiel~\cite{FolPen16}. After submitting this manuscript, we found that (in the Riemannian case) our result can be deduced from the Kenmotsu-type formula given Gálvez--Mira~\cite[Thm.~3.1]{GM16}, see Remark~\ref{rmk:kenmotsu}; however, our arguments provide a different geometric insight and also apply in the Lorentzian setting. More specifically, in Theorem~\ref{thm:angles}, some clever algebraic tricks allow us to recover $T_1,T_2,T_3$ and $S$ from $\nu_1,\nu_2,\nu_3$ up to certain choices of signs at points where the trace of $S$ is non-zero (indeed, we shall see that we are only able to determine the square of the mean curvature, not its sign). The existence of the immersion just relies only on the equations for the derivatives of $T_1,T_2,T_3$ obtained in the process for each possible choice of the mean curvature function, so the immersion might be not unique. 

In this sense, we introduce the notion of \emph{angular companions} as isometric immersions $\phi,\widetilde\phi:\Sigma\to G$ that have pointwise the same three angle functions. Although this problem seems to be overdetermined, we find that an immersion has an angular companion if and only if it satisfies a rather non-evident \textsc{pde} locally at points with non-zero mean curvature, see Equation~\eqref{eqn:spm}, in which case such a companion is locally unique. The Bonnet associate family of zero mean curvature surfaces in $\mathbb{R}^3$ or $\mathbb{L}^3$ are classical examples of angular companions, as well as twin $H$-immersions in the round sphere $\mathbb{S}^3$, see Examples~\ref{ex:bonnet} and~\ref{ex:twin}. Note the lesser known fact that Lawson correspondence~\cite{Lawson70} (see also Gro\ss{}e-Brauckmann~\cite{GB93}) preserves the angle functions while running over different space forms.

Restricting to the Riemannian case, we use the characterization of angular companions in order to give a complete classification of isometric immersions into $\R^3$ or $\mathbb{S}^3$ whose left-invariant Gauss maps into $\mathbb{S}^2$ differ by an orientation-preserving isometry of $\mathbb{S}^2$. This problem was already solved by Abe--Erbacher~\cite{AbeErb75} in $\mathbb{R}^3$, though our approach provides a more direct proof. In~~\cite{AbeErb75}, the case of $\mathbb{S}^3$ is left as an open question that is now settled by Corollary~\ref{coro:uniqueness-angles}. Essentially, we show that the Bonnet family in $\mathbb{R}^3$ and the twin $H$-immersions in $\mathbb{S}^3$ are the only non-trivial examples of this behaviour up to ambient isometries. It would be very interesting to extend this result and the classification of angular companions to the rest of $\mathbb{E}(\kappa,\tau)$-spaces, but some difficulties arise as shown by Examples~\ref{ex:cylinders} and~\ref{ex:twin}.

Our approach to this problem of prescribing $g$ turns out to find an obstruction when the target surface has $H=0$ and $\sum_{\alpha=1}^3c_\alpha\nu_\alpha^2=0$, in which case most equations degenerate. We analyze these conditions in Section~\ref{sec:Hzeta0} to find that they are equivalent to the fact that the immersion has constant angle functions, which is in turn equivalent to be a left coset of a $2$-dimensional Lie subgroup of $G$, see Proposition~\ref{prop:Hzeta0}. Indeed, we find that $H=0$ and $\sum_{i=1}^3c_\alpha\nu_\alpha^2=0$ are necessary and sufficient conditions to have an immersion with constant angles. This also leads us to classify totally geodesic surfaces, see Theorem~\ref{thm:tot-geodesic}, which generalizes the work of Tsukada~\cite{Tsukada} in the Riemannian case. It is worth citing the recent work of Calvaruso--Castrillón--Pellegrino~\cite{CalCasPel25} where totally umbilical surfaces are classified in $\Nil$ even in the non-diagonalizable case. It would be natural to extend these results to totally umbilical surfaces in all Lorentzian homogeneous $3$-manifolds (the Riemannian case was completely solved in~\cite{MS}).

As a second application of Theorem~\ref{thm:fundamental}, in Section~\ref{sec:bcv} we deal with the extension of Daniel's fundamental theorem to the Lorentzian setting. The idea is to prescribe the shape operator $S$ plus a minimal subset of $T_1,T_2,T_3,\nu_1,\nu_2,\nu_3$. Proposition~\ref{prop:M-LM} actually shows that this only makes sense if the dimension of the isometry group is at least $4$, which turns our attention to the spaces $\mathbb{E}(\kappa,\tau)$, $\mathbb{L}(\kappa,\tau)$ and $\widehat{\mathbb{L}}(\kappa,\tau)$. In Theorem~\ref{thm:dim4}, we will give a unified fundamental theorem in the spirit of Daniel's that works in the three scenarios for spacelike or timelike immersions. This leads to four Lorentzian analogues of the Daniel correspondence for spacelike and timelike surfaces in $\mathbb{L}(\kappa,\tau)$ and $\widehat{\mathbb{L}}(\kappa,\tau)$, see Corollary~\ref{coro:fundamental-L}. In $\mathbb{L}(\kappa,\tau)$, our correspondence generalizes that of Palmer~\cite{Palmer90} in Lorentzian space forms, and we conjecture it must be dual to the Daniel correspondence via the Calabi--Lee duality~\cite{Lee}. This framework encompasses product spaces with $\tau=0$, although product spaces are not unimodular metric Lie groups. Rather luckily, the fundamental equations extend continuously to the case $\tau=0$ and we can rely on fundamental theorems previously proved by Roth~\cite{Roth11} in $\mathbb{M}^2(\kappa)\times\R_1$ and by Lawn--Ortega~\cite{LO15} in $\mathbb{M}^2_1(\kappa)\times\R$ (this last work actually applies to a broader family of warped products, not necessarily homogeneous). 

It is interesting to point out that all results in Section~\ref{sec:fundamental}, i.e., the prescription of either $T_1,T_2,T_3$ or $\nu_1,\nu_2,\nu_3$ can be formally solved as first-order \textsc{pde} system, in the sense that we employ Frobënius theorem just once. However, in order to prescribe the shape operator in Section~\ref{sec:bcv}, the equations becomes a second-order \textsc{pde} system, and Frobënius theorem has to be applied twice, being the change-of-frame matrix $M$ introduced in Section~\ref{sec:change-of-frame} the intermediate object in the integration process.

As our third and last application, in Section~\ref{sec:lawson} we will show that there is no further extension (among Riemannian unimodular metric Lie groups) of the classical Lawson correspondence~\cite{Lawson70} other than the Daniel correspondence. Note that Daniel correspondence has been a cornerstone in the theory of \textsc{cmc} surfaces in $\mathbb{E}(\kappa,\tau)$ and has fostered conjugate Plateau constructions in product spaces (see~\cite{CMT} and its references). Also, one can make use of a correspondence to get a simpler related problem in another ambient space, e.g., see~\cite{DomMan,CCC}. In Theorem~\ref{thm:lawson} we show that for fixed $H,\widetilde H\in\R$ and Riemannian unimodular metric Lie groups $G$ and $\widetilde G$, we cannot find any non-trivial bijection (up to ambient isometries and permutation of structure constants) from the family of \textsc{cmc} $H$ surfaces in $G$ to the family of \textsc{cmc} $\widetilde H$ surfaces in $\widetilde G$ that rotates the traceless shape operators (with some natural extra assumptions of compatibility with the orientation). We conjecture this non-existence result also holds in the Lorentzian case, but the computations and distinction of cases are too intricate and laborious for a result that would likely remain negative.

\medskip

\noindent\textbf{Acknowledgement.} This work was supported by grant no.~PID2022-142559NB-I00, funded by MCIN/AEI/10.13039/501100011033, and is part of the third author’s PhD thesis. The authors would like to thank B. Daniel and M. Domínguez-Vázquez for their comments, which improved the final version of the manuscript.

\section{Unimodular metric Lie groups of diagonalizable type}\label{sec:lie-groups}

A Lie group $G$ is called \emph{unimodular} if its left-invariant Haar measure is also right-invariant. It is well known that $G$ is unimodular if, and only if, for any $X\in\mathfrak g$, the adjoint endomorphism $\mathrm{ad}_X:\mathfrak g\to\mathfrak g$ given by $\mathrm{ad}_X(Y)=[X,Y]$ has trace equal to zero (as usual, $\mathfrak g$ denotes the Lie algebra of $G$). In dimension $3$, if $G$ is oriented and endowed with a (nondegenerate) semi-Riemannian metric $\langle\cdot,\cdot\rangle$, then both the cross product $\times$ and the Lie bracket $[\cdot,\cdot]$ are skew-symmetric, Milnor~\cite{Milnor} considered the unique linear operator $L:\mathfrak g\to\mathfrak g$ such that
\begin{equation}\label{eqn:unimodular}
[X,Y]=L(X\times Y),\quad\text{for all }X,Y\in\mathfrak g.
\end{equation}
Here, the cross product is defined so that $\langle u\times v,w\rangle=\det_{\mathbb{B}}(u,v,w)$ for all $u,v,w\in\mathfrak{g}$, where the determinant takes the coordinates of $u,v,w$ (as columns) in any positively oriented orthonormal basis $\{u_1,u_2,u_3\}$ with $\langle u_i,u_j\rangle=\epsilon_i\delta_j^i$ and $\epsilon_1,\epsilon_2,\epsilon_3\in\{-1,1\}$ indicate the signature of the metric. It turns out that $G$ is unimodular if and only if $L$ in~\eqref{eqn:unimodular} is self-adjoint, see~\cite[Lem.~2.1]{Rahmani}. 

\begin{definition}
We will say that $G$ is an unimodular metric Lie group of diagonalizable type if the Milnor operator $L$ is diagonalizable.
\end{definition}

This condition will be assume hereafter, and it is automatically true in the Riemannian case $\epsilon_1=\epsilon_2=\epsilon_3=1$. It implies the existence of a positively oriented left-invariant orthonormal frame $\{E_1,E_2,E_3\}$ in $G$ such that $\langle E_i,E_j\rangle=\epsilon_i\delta_j^i$ and constants $c_1,c_2,c_3\in\R$ such that 
\begin{align}
[E_1,E_2]&=c_3E_3, &[E_2,E_3]&=c_1E_1,&[E_3,E_1]&=c_2E_2.\label{eqn:lie-bracket-unimodular-frame}
\end{align}
Note also that being positively oriented yields
\begin{align}
E_1\times E_2&=\epsilon_3 E_3, &E_2\times E_3&=\epsilon_1E_1,&E_3\times E_1&=\epsilon_2E_2.\label{eqn:cross-unimodular-frame}
\end{align}
The constants $c_1,c_2,c_3$ and $\epsilon_1,\epsilon_2,\epsilon_3$ will be called the \emph{structure constants} and the \emph{signs} of $G$, respectively, for short.

The signs of $c_1,c_2,c_3\in\R$ in turn determine the underlying Lie group structure as shown in Table~\ref{fig:unimodular-mlg}. Note that it is not necessary to consider all possible signs because a change of sign of all the $c_i$ amounts to reversing the orientation of $G$ while preserving its geometry. We can also relabel $E_1$, $E_2$ and $E_3$ to discard some repeated cases, but in the Lorentzian case we have some duplicities because the vector field $E_3$ is assumed timelike, which breaks the cyclic symmetry of indexes. In order to preserve this symmetry in all the forthcoming computations, we will keep the arbitrary signature constants $\epsilon_1,\epsilon_2,\epsilon_3$ (this will also help us distinguish between timelike and spacelike surfaces in the Lorentzian case).

\begin{table}
\begin{subtable}[t]{0.48\textwidth}
\centering
\begin{tabular}[t]{cccc}
\toprule $c_1$& $c_2$& $c_3$&Lie group\\
\midrule
$+$& $+$& $+$& $\SU(2)$\\
$+$& $+$& $-$& $\widetilde{\mathrm{Sl}}_2(\R)$\\
$+$& $+$& $0$& $\widetilde{\mathrm{E}}(2)$\\
$+$& $-$& $0$& $\mathrm{Sol}_3$\\
$+$& $0$& $0$& $\mathrm{Nil}_3$\\
$0$& $0$& $0$& $\R^3$\\
\bottomrule
\end{tabular}
\caption{Riemannian case\\
$\epsilon_1=\epsilon_2=\epsilon_3=1$.}\label{fig:unimodular-mlg-riemannian}
\end{subtable}
\begin{subtable}[t]{0.48\textwidth}
\centering
\begin{tabular}[t]{cccc}
\toprule $c_1$& $c_2$& $c_3$&Lie group\\
\midrule
$+$& $+$& $+$& $\SU(2)$\\
$+$& $+$& $-$& $\widetilde{\mathrm{Sl}}_2(\R)$\\
$+$& $-$& $+$& $\widetilde{\mathrm{Sl}}_2(\R)$\\
$+$& $+$& $0$& $\widetilde{\mathrm{E}}(2)$\\
$+$& $0$& $+$& $\widetilde{\mathrm{E}}(2)$\\
$+$& $-$& $0$& $\mathrm{Sol}_3$\\
$+$& $0$& $-$& $\mathrm{Sol}_3$\\
$+$& $0$& $0$& $\mathrm{Nil}_3$\\
$0$& $0$& $+$& $\mathrm{Nil}_3$\\
$0$& $0$& $0$&$\R^3$\\
\bottomrule
\end{tabular}
\caption{Lorentzian case\\
$\epsilon_1=\epsilon_2=1,\epsilon_3=-1$}\label{fig:unimodular-mlg-lorentzian}
\end{subtable}
\caption{Underlying Lie group structures depending on the signs of the structure constants, see~\cite{Milnor,Rahmani}.}\label{fig:unimodular-mlg}
\end{table}

\subsection{The isometry group}
Consider the real constants $\mu_1,\mu_2,\mu_3$ given by 
\[\mu_1=\frac{-\epsilon_1c_1+\epsilon_2c_2+\epsilon_3c_3}{2},\quad\mu_2=\frac{\epsilon_1c_1-\epsilon_2c_2+\epsilon_3c_3}{2},\quad\mu_3=\frac{\epsilon_1c_1+\epsilon_2c_2-\epsilon_3c_3}{2}.\]
Koszul formula along with the Lie brackets in~\eqref{eqn:lie-bracket-unimodular-frame} easily yield the Levi-Civita connection $\overline\nabla$ of $G$ in the orthornormal frame $\{E_1,E_2,E_3\}$:
\begin{equation}\label{eqn:nabla-unimodular}
\begin{array}{lclcl}
\overline\nabla_{E_1}E_1=0,&&
\overline\nabla_{E_1}E_2=\epsilon_3\mu_1E_3,&&
\overline\nabla_{E_1}E_3=-\epsilon_2\mu_1E_2,\\
\overline\nabla_{E_2}E_1=-\epsilon_3\mu_2E_3,&&
\overline\nabla_{E_2}E_2=0,&&
\overline\nabla_{E_2}E_3=\epsilon_1\mu_2E_1,\\
\overline\nabla_{E_3}E_1=\epsilon_2\mu_3E_2,&&
\overline\nabla_{E_3}E_2=-\epsilon_1\mu_3E_1,&&
\overline\nabla_{E_3}E_3=0.
\end{array}
\end{equation}
The Riemann curvature tensor defined as
\[\overline R(X,Y)Z=\overline\nabla_X\overline\nabla_YZ-\overline\nabla_Y\overline\nabla_XZ-\overline\nabla_{[X,Y]}Z,\quad X,Y,Z\in\X(G),\]
can be computed from the above expression of the Levi-Civita connection:
\begin{equation}\label{eqn:R-Ei}
\begin{array}{l}
\overline R(E_1,E_2)E_1=\epsilon_2(\epsilon_3\mu_1\mu_2-c_3\mu_3)E_2,\\
\overline R(E_1,E_2)E_2=-\epsilon_1(\epsilon_3\mu_1\mu_2-c_3\mu_3)E_1,\\
\overline R(E_1,E_2)E_3=0,\\
\overline R(E_1,E_3)E_1=\epsilon_3(\epsilon_2\mu_1\mu_3-c_2\mu_2)E_3,\\
\overline R(E_1,E_3)E_2=0,\\
\overline R(E_1,E_3)E_3=-\epsilon_1(\epsilon_2\mu_1\mu_3-c_2\mu_2)E_1,\\
\overline R(E_2,E_3)E_1=0,\\
\overline R(E_2,E_3)E_2=\epsilon_3(\epsilon_1\mu_2\mu_3-c_1\mu_1)E_3,\\
\overline R(E_2,E_3)E_3=-\epsilon_2(\epsilon_1\mu_2\mu_3-c_1\mu_1)E_2.
\end{array}
\end{equation}
The following result follows by just checking that both sides coincide in the orthonormal frame $\{E_1,E_2,E_3\}$, which is a long but straightforward computation (see also~\cite[Lem.~2.2]{MS}, where the Riemannian case is studied).

\begin{lemma}\label{lemma:R}
In the previous situation,
\begin{align*}
\overline R&=(\epsilon_1\mu_2\mu_3-c_1\mu_1)\epsilon_2\epsilon_3\overline R_1+(\epsilon_2\mu_1\mu_3-c_2\mu_2)\epsilon_1\epsilon_3\overline R_2+(\epsilon_3\mu_1\mu_2-c_3\mu_3)\epsilon_1\epsilon_2\overline R_3,
\end{align*}
where, for $i\in\{1,2,3\}$ and $X,Y,Z\in\X(G)$, the tensor $\overline R_i$ is given by
\begin{align*}
\overline R_i(X,Y)Z&=\langle X,Z\rangle Y-\langle Y,Z\rangle X-\epsilon_i\langle Z,E_i\rangle\langle X,E_i\rangle Y\\&\quad+\epsilon_i\langle Z,E_i\rangle\langle Y,E_i\rangle X-\epsilon_i\langle Y,E_i\rangle\langle X,Z\rangle E_i+\epsilon_i\langle X,E_i\rangle\langle Y,Z\rangle E_i.
\end{align*}
\end{lemma}

Note that a metric Lie group has isometry group of dimension $3$, $4$ or $6$. We will now give a classification of the most symmetric cases.

\begin{proposition}\label{prop:dim-iso-4-6}
Let $G$ be a simply connected semi-Riemannian metric Lie group of diagonalizable type. The following classification is up to relabeling indexes and globally changing the sign of the metric.
	\begin{enumerate}[label=\emph{(\alph*)}]
		\item The isometry group of $G$ has dimension $6$ if and only if either $\epsilon_1c_1=\epsilon_2c_2=\epsilon_3c_3$ or $\epsilon_1c_1=\epsilon_2c_2$ and $c_3=0$. More precisely:
		\begin{enumerate}[label=\emph{(\roman*)}]
			\item If $\epsilon_1c_1=\epsilon_2c_2=\epsilon_3c_3$, we get two subcases:
			\begin{itemize}
			 	\item If $\epsilon_1=\epsilon_2=\epsilon_3=1$, then $G$ is a Riemannian space form of constant curvature $\frac{1}{4}c_1^2$, that is, either the round sphere $\mathbb{S}^3\cong\SU(2)$, or the flat Euclidean space $\R^3$.
			 	\item If $\epsilon_1=\epsilon_2=1$ and $\epsilon_3=-1$, then $G$ is a Lorentzian space form of constant curvature $\frac{-1}{4}c_1^2$, that is, either the universal cover of the anti-de Sitter space $\mathbb{H}^3_1\cong\widetilde{\mathrm{SL}}_2(\R)$, or the flat Minkowski space $\mathbb{L}^3$.
			 \end{itemize}  
			
			\item If $\epsilon_1c_1=\epsilon_2c_2\neq 0$ and $c_3=0$, we get two subcases: 
			\begin{itemize}
			 	\item If $\epsilon_1=\epsilon_2=1$, then $G$ is $\widetilde{\mathrm{E}}(2)$ with a flat metric (globally isometric to $\R^3$ or $\mathbb{L}^3$ depending on the sign of $\epsilon_3$).
			 	\item If $\epsilon_1=\epsilon_3=1$ and $\epsilon_2=-1$, then $G$ is isomorphic to $\mathrm{Sol}_3$ with a Lorentzian flat metric (globally isometric to $\mathbb{L}^3$).
			 \end{itemize}  
		\end{enumerate}
		\item The isometry group of $G$ has dimension $4$ if and only if $\epsilon_1c_1=\epsilon_2c_2\neq\epsilon_3c_3\neq 0$. These spaces admit semi-Riemannian submersions with constant bundle curvature over a complete (Riemannian or Lorentzian) surface of constant curvature whose fibers are the integral curves of a unit (spacelike or timelike) Killing vector field. In other words, we recover the Bianchi-Cartan-Vr\u{a}nceanu spaces $\mathbb{E}(\kappa,\tau)$, their Lorentzian counterparts $\mathbb{L}(\kappa,\tau)$, and another new family $\widehat{\mathbb{L}}(\kappa,\tau)$. (Explicit models will be given later on in Remark~\ref{rmk:extension-previous-metrics}.)
	\end{enumerate}
\end{proposition}

\begin{proof}
The sectional curvature of the plane spanned by $E_i$ and $E_j$ is given by $\epsilon_i\epsilon_j\langle\overline R(E_i,E_j)E_j,E_i\rangle$, so that $G$ has constant curvature whenever this quantity does not depend on the choice of $i\neq j$. By Equation~\eqref{eqn:R-Ei}, this amounts to
\begin{equation}\label{eqn:csc1}
\epsilon_2\epsilon_3(-\epsilon_1\mu_2\mu_3+c_1\mu_1)=\epsilon_1\epsilon_3(-\epsilon_2\mu_1\mu_3+c_2\mu_2)=\epsilon_1\epsilon_2(-\epsilon_3\mu_1\mu_2+c_3\mu_3).
\end{equation}
After some manipulations,~\eqref{eqn:csc1} can be written equivalently as
\begin{equation}\label{eqn:csc2}\mu_1\mu_2=\mu_2\mu_3=\mu_3\mu_1.
\end{equation}
This implies that either $\mu_1=\mu_2=\mu_3$ (i.e., $\epsilon_1c_1=\epsilon_2c_2=\epsilon_3c_3$) or $\mu_1=\mu_2=0$ and $\mu_3\neq 0$ (i.e., $\epsilon_1c_1=\epsilon_2c_2\neq 0$ and $c_3=0$) up to relabeling indexes. Item (a) follows from checking the Lie group structures in Table~\ref{fig:unimodular-mlg} and computing the sectional curvature of $G$ by means of Equation~\eqref{eqn:R-Ei} in each subcase.

As for item (b), observe first that the frame $\{E_1,E_2,E_3\}$ diagonalizes the Ricci operator, that is, we have $\Ric(E_i,E_j)=0$ if $i\neq j$ and $\Ric(E_i,E_i)=\epsilon_i\lambda_i$, with $\lambda_1=2\epsilon_1\epsilon_2\epsilon_3\mu_2\mu_3$, $\lambda_2=2\epsilon_1\epsilon_2\epsilon_3\mu_3\mu_1$ and $\lambda_3=2\epsilon_1\epsilon_2\epsilon_3\mu_1\mu_2$. If $\{X_1,X_2,X_3\}$ is any other orthonormal frame that diagonalizes the Ricci operator with $\Ric(X_i,X_i)=\epsilon_id_i$, and we express $X_i=\sum_{k=1}^3a_{ik}E_k$, it follows that $AL\mathcal{E}A^t=D\mathcal{E}$, where $L=\operatorname{diag}(\lambda_1,\lambda_2,\lambda_3)$, $D=\operatorname{diag}(d_1,d_2,d_3)$ and $\mathcal{E}=\operatorname{diag}(\epsilon_1,\epsilon_2,\epsilon_3)$. Since $A$ is a change-of-basis for orthonormal frames, we get $\mathcal{E}A^t=A^{-1}\mathcal{E}$, which gives $ALA^{-1}=D$. From standard diagonalizability arguments, it follows that $d_i$ are a permutation of the $\lambda_i$ and each $X_i$ lies in the corresponding eigenspace (spanned by more than one of the $E_i$ if some $\lambda_i$ has higher multiplicity). In particular, if $\lambda_1,\lambda_2,\lambda_3$ are all distinct, then $\{E_1,E_2,E_3\}$ is determined up to changes of signs, which means that the stabilizer of any point of $G$ is finite and hence $\dim(\Iso(G))=3$. Also, if $\lambda_1=\lambda_2=\lambda_3$, then~\eqref{eqn:csc2} holds and $\dim(\Iso(G))=6$.

We are left with the case just two of the eigenvalues coincide, so assume $\lambda_1=\lambda_2\neq\lambda_3$ with no loss of generality (i.e., $\mu_2\mu_3=\mu_1\mu_3\neq \mu_1\mu_2$), which gives two cases up to relabeling indexes and changing the sign of the metric:
\begin{enumerate}
	\item Assume $\mu_1=\mu_2$ and $\mu_3\neq 0$. If $\epsilon_1=\epsilon_2=1$, this means that $c_1=c_2\neq\epsilon_3c_3\neq 0$, so there exist $\kappa,\tau\in\R$, $\tau\neq 0$, such that $c_1=c_2=\frac{\kappa}{2\tau}$ and $c_3=2\tau$, and it is well known that $G$ is isometric to the Bianchi-Cartan-Vr\u{a}nceanu space $\E(\kappa,\tau)$ in the Riemannian case $\epsilon_3=1$ or its Lorentzian counterpart $\mathbb{L}(\kappa,\tau)$ if $\epsilon_3=-1$, see~\cite{Daniel09,MeeksPerez12,DLM}. Both $\E(\kappa,\tau)$ and $\mathbb{L}(\kappa,\tau)$ are Killing submersions with constant bundle curvature $\tau$ over the complete simply connected Riemannian surface $\mathbb{M}^2(\kappa)$ of constant curvature $\kappa$, see also~\cite{Man14}. 

	If $\epsilon_1=1$ and $\epsilon_2=-1$, we can also set $c_1=-c_2=\frac{-\kappa}{2\tau}$ and $c_3=2\tau\neq 0$. This Lorentzian space $G$ admits a semi-Riemannian Killing submersion (in which the Killing vector field is spacelike) with constant bundle curvature $\tau$ over the complete simply connected Lorentzian surface $\mathbb{M}^2_1(\kappa)$ of constant curvature $\kappa$. We also get an isometry group of dimension $4$, since the $3$-dimensional group $\Iso(\mathbb{M}^2_1(\kappa))$ can be lifted to $G$ up to vertical translations with respect to the aforesaid submersion. Although the theory of semi-Riemannian Killing submersions is not established yet, these are standard arguments that can be straightforwardly adapted from~\cite{Man14} to the Lorentzian case. Note that models given by Remark~\ref{rmk:extension-previous-metrics} enable explicit verifications of these assertions.

	\item Assume now $\mu_3=0$ and $\mu_1,\mu_2\neq 0$, which implies that $\lambda_3$ is not equal to $\lambda_1$ or $\lambda_2$, whence any isometry $\Phi\in\Iso(G)$ must preserve $E_3$ or send it to $-E_3$. Therefore, $\Iso(G)$ has dimension $4$ if and only if there is a continuous $1$-parameter subgroup in the stabilizer (which must preserve the orientation, the causality and the direction $E_3$). This directs our attention to isometries that rotate $E_1$ and $E_2$ (in a Riemannian or Lorentzian sense).

	If $\epsilon_1=\epsilon_2=1$, assume $\Phi\in\Iso(G)$ is such that $\Phi_*E_1=\cos(\theta)E_1+\sin(\theta)E_2$, $\Phi_*E_2=-\sin(\theta)E_1+\cos(\theta)E_2$ and $\Phi_*E_3=E_3$ for some function $\theta\in\mathcal C^\infty(G)$. The identities $\langle[E_1,E_3],E_2\rangle=-c_2$ and $\langle[E_2,E_3],E_1\rangle=c_1$ can be pushforwarded by means of $\Phi$ to obtain
		\begin{equation}\label{eqn:csc3}
			\begin{aligned}
		 		 -c_2&=\langle[\Phi_*E_1,\Phi_*E_3],\Phi_*E_2\rangle=-c_1\sin^2(\theta)-c_2\cos^2(\theta)-E_3(\theta),\\
		 		 c_1&=\langle[\Phi_*E_2,\Phi_*E_3],\Phi_*E_1\rangle=c_2\sin^2(\theta)+c_1\cos^2(\theta)+E_3(\theta).
				\end{aligned}
		\end{equation}
		Adding both expressions in~\eqref{eqn:csc3}, we get $(c_1-c_2)(1-\cos(2\theta))=0$. We can assume $c_1\neq c_2$ since the case $c_1=c_2$ has been analyzed in the above item (1), so we get $\theta\equiv 0$ o $\theta\equiv \pi$. In particular, the stabilizers of $G$ are discrete and $G$ has isometry group of dimension $3$.

		If $\epsilon_1=1$ and $\epsilon_2=-1$, we are looking for $\Phi\in\Iso(G)$ with $\Phi_*E_1=\cosh(\theta)E_1+\sinh(\theta)E_2$, $\Phi_*E_2=\sinh(\theta)E_1+\cosh(\theta)E_2$ and $\Phi_*E_3=E_3$ for some $\theta\in\mathcal C^\infty(G)$. Applying $\Phi$ to $\langle[E_1,E_3],E_2\rangle=-\epsilon_2c_2$ and $\langle[E_2,E_3],E_1\rangle=\epsilon_1c_1$, the results add to $(c_1+c_2)(1-\cosh(2\theta))=0$. Since we can assume $c_1\neq -c_2$ by item (1), it must be $\theta\equiv 0$ and we are done as in the previous case.\qedhere
	\end{enumerate}
\end{proof}

\begin{remark}\label{rmk:stabilizer}
The isometry group of a metric Lie group $G$ is determined by the left-translations by elements of $G$ itself and by the stabilizer $\Stab_p$ of any $p\in G$ (all stabilizers are conjugate by left-translations). It is well known that each element of $\Stab_p$ preserves left-invariant vector fields so it becomes an isometry of the Lie algebra $\mathfrak{g}$. Thus, $\Stab_p$ can be identified naturally with a subgroup of $\OO(3)$ in the Riemannian case or $\OO(2,1)$ in the Lorentzian case. If $\dim(\Iso(G))=6$, then $\Stab_p$ is well known is all $\OO(3)$ or $\OO(2,1)$. If $\dim(\Iso(G))=4$, then $\Stab_p$ contains a Lie subgroup $\SO(2)$ or $\SO(1,1)$ of rotations about the unit Killing direction. If $\dim(\Iso(G))=3$, then $\Stab_p$ is discrete and contains either $4$ or $8$ elements, depending on whether or not $G$ admits orientation-reversing isometries (see~\cite[Props.~2.21 and~2.24]{MeeksPerez12}, although the proof easily extends to the semi-Riemannian case since elements of $\Stab_p$ must preserve the directions that diagonalize the Ricci operator by the same arguments as in the proof of Proposition~\ref{prop:dim-iso-4-6}).
\end{remark}

\subsection{A unified model in coordinates}
Given real constants $c_1,c_2,c_3\in\mathbb{R}$, we will work in the open domain of the plane
\[D = \{(x,y)\in \mathbb{R}^2: \lambda(x,y) > 0\},\qquad \lambda(x,y) = \bigr(1+\tfrac{c_3}{4}(c_2 x^2+c_1 y^2)\bigl)^{-1}.\]
Consider the functions $c,s:\mathbb{R}\to\mathbb{R}$ defined by
\[c(z)=\begin{cases}
\cos(\sqrt{c_1c_2}z)&\text{if }c_1c_2>0,\\
1&\text{if }c_1c_2=0,\\
\cosh(\sqrt{-c_1c_2}z)&\text{if }c_1c_2<0,
\end{cases}\qquad
s(z)=\begin{cases}
\frac{\sin(\sqrt{c_1c_2}z)}{\sqrt{c_1c_2}}&\text{if }c_1c_2>0,\\
z&\text{if }c_1c_2=0,\\
\frac{\sinh(\sqrt{-c_1c_2}z)}{\sqrt{-c_1c_2}}&\text{if }c_1c_2<0.
\end{cases}\]
Observe that $c$ and $s$ are the unique solutions to the \textsc{ode} system given by $c'(z)=-c_1c_2 s(z)$ and $s'(z)=c(z)$ with initial conditions $c(0)=1$ and $s(0)=0$, which comes in handy for computations. This way, it is easy to check that
\begin{equation}\label{eqn:G-frame}
	\begin{aligned}
		E_1 & = \tfrac{1}{\lambda(x,y)}\left(c(z)\partial_x + c_2s(z)\partial_y\right)+\tfrac{c_3}{2}\left(c_2x\,s(z)- y\, c(z)\right)\partial_z,\\
		E_2 & = \tfrac{1}{\lambda(x,y)}\left(-c_1s(z)\partial_x + c(z)\partial_y\right)+\tfrac{c_3}{2}\left(x\,c(z)+ y\,c_1s(z)\right)\partial_z,\\
		E_3 & = \partial_z,
	\end{aligned}\end{equation}
define a global frame in $D\times\R$ satisfying~\eqref{eqn:lie-bracket-unimodular-frame}. Therefore, we shall also consider the unique semi-Riemannian metric in $D\times\R$ with $\langle E_i, E_i\rangle = \epsilon_i$, which makes $D\times\R$ locally isometric to the unimodular metric Lie group $G$ of diagonalizable type with structure constants $c_1,c_2,c_3$ and signs $\epsilon_1,\epsilon_2,\epsilon_3$ (note that $\{E_1,E_2,E_3\}$ becomes left-invariant through this local isometry). The metric in $D\times\R$ can be written as 
\begin{equation}\label{eqn:CartanMetric}
\begin{aligned}
	\df s^2&=\epsilon_1\lambda(x,y)^2(c(z)\df x+c_1s(z)\df y)^2+\epsilon_2\lambda(x,y)^2(c_2s(z)\df x-c(z)\df y)^2\\
	&\qquad+\epsilon_3\Bigl(\df z+\tfrac{c_3\lambda(x,y)}{2}(y\,\df x-x\,\df y)\Bigr)^2.
\end{aligned}
\end{equation}

\begin{remark}\label{rmk:extension-previous-metrics}
If $c_1=c_2=\frac{\kappa}{2\epsilon_3\tau}$, $c_3=2\epsilon_3\tau$ and $\epsilon_1=\epsilon_2=1$, this model coincides with the Cartan model for $\E(\kappa,\tau)$-spaces ($\epsilon_3=1$), which was revisited and popularized by Daniel~\cite{Daniel07}, and generalized by Lee~\cite{Lee} to obtain their Lorentzian counterparts, namely the $\mathbb{L}(\kappa,\tau)$-spaces ($\epsilon_3=-1$). More explicitly, the metric in~\eqref{eqn:CartanMetric} becomes
\begin{equation}\label{eqn:Ekt-Lkt-metric}\df s^2=\lambda^2(\df x^2+\df y^2)+\epsilon_3\left(\df z+\epsilon_3\tau\lambda(y\,\df x-x\,\df y)\right)^2,
\end{equation}
where $\lambda=\bigl(1+\tfrac{\kappa}{4}(x^2+y^2)\bigr)^{-1}$. The sign of $\tau$ is chosen to agree with the definition of bundle curvature with the standard orientation and unit Killing $\partial_z$ (cf.~\cite[Ex.~1.4]{DLM}).

On the contrary, by setting $c_1=-c_2=\frac{-\kappa}{2\tau}$, $c_3=2\tau$, and signs $\epsilon_1=\epsilon_3=1,\epsilon_2=-1$, we find the family $\widehat{\mathbb{L}}(\kappa,\tau)$ that has shown up in Proposition~\ref{prop:dim-iso-4-6}, with metrics
\begin{equation}\label{eqn:darkLkt-metric}
	\df s^2=\lambda^2(\df x^2-\df y^2)+\left(\df z+\tau\lambda(y\,\df x-x\,\df y)\right)^2,
\end{equation}
where $\lambda=(1+\frac{\kappa}{4}(x^2-y^2))^{-1}$. Notice that $E_1=\tfrac{1}{\lambda}\partial_x-\tau y\partial_z$, $E_2=\tfrac{1}{\lambda}\partial_y+\tau x\partial_z$ and $E_3=\partial_z$ form a global orthonormal frame with respect to~\eqref{eqn:darkLkt-metric}, where $E_2$ is timelike and $E_3$ is a spacelike unit Killing vector field. The natural submersion $(x,y,z)\mapsto (x,y)$ is semi-Riemannian with constant bundle curvature $\tau=\frac{1}{2}\langle[E_1,E_2],E_3\rangle$ (cf.~\cite[Eqn.~(2.10)]{DLM}) over the base surface $D$ equipped with the Lorentzian metric $\lambda^2(\df x^2-\df y^2)$ of constant curvature $\kappa$. This calculation of the curvature follows from~\cite[Prop.~3.44]{Oneill}, for we can work out
\[K_D=\tfrac{-1}{\lambda^2}\left(\bigl(\tfrac{\lambda_x}{\lambda}\bigr)_x-\bigl(\tfrac{\lambda_y}{\lambda}\bigr)_y\right)=\kappa.\]
\end{remark}

\begin{remark}\label{rmk:analyticity}
The functions $c$ and $s$ depend analytically on the parameters $c_1$ and $c_2$. Indeed, by considering the auxiliary entire analytic functions
\[\alpha(t)=\sum_{n=0}^\infty\frac{(-1)^nt^n}{(2n)!},\qquad \beta(t)=\sum_{n=0}^\infty\frac{(-1)^nt^n}{(2n+1)!},\]
it follows that $c(z)=c_1c_2z\,\alpha(c_1c_2z^2)$ and $s(z)=z\,\beta(c_1c_2z^2)$. Thus, the whole family of metrics~\eqref{eqn:CartanMetric} depends analytically on the parameters $c_1,c_2,c_3\in\mathbb{R}$.
\end{remark}

Our approach to these models is in turn based on the canonical frame for Killing submersions, which uses $\partial_z$ as the distinguished Killing vector field, see~\cite{Man14,DLM}. The algebraic Killing vector fields in a metric Lie group are the right-invariant vector fields; in constrast, our model for metric Lie groups takes $\partial_z$ as a left-invariant vector field that diagonalizes the Ricci operator. Our approach has the advantage that even in the same Lie group we have three models depending on which of the directions $E_1,E_2,E_3$ is chosen to become $\partial_z$, and this amounts to rotate the indexes cyclically. The rest of this subsection will be devoted to analyze if the models are global. To this end, we will begin by discussing the relation between our models and the traditional models for the three big families of unimodular Lie groups (namely, $\widetilde{\mathrm{SL}}_2(\R)$, $\mathrm{SU}(2)$ and $\R^2\rtimes_A\R$).

\begin{example}[Special linear group]\label{ex:SL2R}
If $c_1,c_2>0$ and $c_3<0$, then $G$ is $\widetilde{\mathrm{SL}}_2(\R)$, the universal cover of $\mathrm{SL}_2(\R)$, equipped with a left-invariant metric. The group $\mathrm{SL}_2(\R)$ is usually represented by the matrices $(\begin{smallmatrix}a&b\\\bar{b}&\bar{a}\end{smallmatrix})\in\mathcal M_2(\C)$ with determinant $1$ (the product in $\mathrm{SL}_2(\R)$ is just the product of matrices), so that it can be identified with the quadric $Q\equiv \{(a,b)\in\C^2:|a|^2-|b|^2=1\}$ that represents the unit sphere in $\R^4_2$. The frame given by $X_1=\frac{1}{2}\sqrt{c_2c_3}(i\bar{b},i\bar{a})$, $X_2=\frac{1}{2}\sqrt{c_1c_3}(\bar{b},\bar{a})$ and $X_3=\frac{1}{2}\sqrt{c_2c_3}(ia,ib)$ is left-invariant with $[X_1,X_2]=c_3X_3$, $[X_2,X_3]=c_1X_1$ and $[X_3,X_1]=c_2X_2$ and the map $F:D\times\R\to Q$ given by
\begin{equation}\label{eqn:isometry-cylinder-quadric-SL}
F(x,y,z)=\left(\frac{e^{\frac{i}{2}\sqrt{c_1c_2}z}}{\sqrt{1+\frac{c_3}{4}(c_2x^2+c_1y^2)}},\frac{(\sqrt{-c_1c_3}y+i\sqrt{-c_2c_3}x)e^{\frac{i}{2}\sqrt{c_1c_2}z}}{\sqrt{1+\frac{c_3}{4}(c_2x^2+c_1y^2)}}\right)
\end{equation}
is a Riemannian covering map with $F_*E_i=X_i$. This means that $D\times\R$ is globally isometric to $\widetilde{\mathrm{SL}}_2(\R)$.
\end{example}

\begin{example}[Special unitary group]\label{ex:SU2}
If $c_1,c_2,c_3>0$, then $G$ is isomorphic and isometric to $\SU(2)$ with a left-invariant metric. This group contains (by definition) the unitary matrices $(\begin{smallmatrix}a&b\\-\bar{b}&\bar{a}\end{smallmatrix})\in\mathcal M_2(\C)$ with determinant $1$, when the product of matrices is considered. This yields an identification with $\mathbb{S}^3\equiv\{(a,b)\in\C^2:|a|^2+|b|^2=1\}$, where the vector fields $X_1=\frac{1}{2}\sqrt{c_2c_3}(-i\bar{b},i\bar{a})$, $X_2=\frac{1}{2}\sqrt{c_1c_3}(-\bar{b},\bar{a})$ and $X_3=\frac{1}{2}\sqrt{c_2c_3}(ia,ib)$ become left-invariant with $[X_1,X_2]=c_3X_3$, $[X_2,X_3]=c_1X_1$ and $[X_3,X_1]=c_2X_2$. The map $F:D\times\R\to\mathbb{S}^3-\{(e^{i\theta},0):\theta \in \mathbb{R}\}$ given by
\begin{equation}\label{eqn:isometry-cylinder-quadric-SU}
F(x,y,z)=\left(\frac{e^{\frac{i}{2}\sqrt{c_1c_2}z}}{\sqrt{1+\frac{c_3}{4}(c_2x^2+c_1y^2)}},\frac{(\sqrt{c_1c_3}y+i\sqrt{c_2c_3}x)e^{\frac{i}{2}\sqrt{c_1c_2}z}}{\sqrt{1+\frac{c_3}{4}(c_2x^2+c_1y^2)}}\right)
\end{equation} is again a Riemannian covering with $F_*E_i=X_i$. Observe that~\eqref{eqn:isometry-cylinder-quadric-SU} and~\eqref{eqn:isometry-cylinder-quadric-SL} differ just by the signs in the radicands and this makes the difference between taking values in $\mathbb{S}^3$ or in $Q$. Note also that $\{(e^{i\theta},0):\theta \in \mathbb{R}\}$ is just an integral curve of $X_3$, whence our cylinder model $D\times\R$ consists in removing such a curve from $\SU(2)$ and then considering the universal cover.
\end{example}

\begin{example}[Semidirect products]\label{ex:semidirect-product} If $c_3=0$ and $c_1,c_2\in\mathbb{R}$ are arbitrary real numbers, then $G$ is isomorphic to the semidirect product $\R^2\rtimes_{A}\R$, with $A=(\begin{smallmatrix}0&-c_1\\c_2&0\end{smallmatrix})$. This matrix can be exponentiated as
\[e^{zA}=\begin{pmatrix}c(z)&-c_1s(z)\\c_2s(z)&c(z)\end{pmatrix}.\]
In view of~\cite[\S2.2]{MeeksPerez12}, this means that $X_1=c(z)\partial_x+c_2s(z)\partial_y$, $X_2=-c_1s(z)\partial_x+c(z)\partial_y$ and $X_3=\partial_z$ define a global left-invariant frame for $G$, viewed as $\R^3\equiv \R^2\rtimes_{A}\R$. Since $c_3=0$, we get from~\eqref{eqn:G-frame} that $X_i=E_i$ for all $i$, whence the metric~\eqref{eqn:CartanMetric} coincides with the standard metric as a semidirect product.
\end{example}

\begin{proposition}\label{prop:global-model}
The cylinder model $D\times\mathbb{R}$ endowed with the generalized Cartan metric~\eqref{eqn:CartanMetric} is a global model if and only if $c_1c_3\leq 0$ and $c_2c_3\leq 0$.
\end{proposition}

\begin{proof}
If $c_1c_3>0$, then $\gamma(t)=(0,\frac{2}{\sqrt{c_1c_3}}\tan(\frac{1}{2}\sqrt{c_1c_3}t),0)$ is an integral curve of $E_2$, and hence has unit speed. It is defined only for $|t|<\frac{\pi}{\sqrt{c_1c_3}}$ but it clearly diverges in the model, whence $E_2$ is not a complete vector field, and then the model is not global. A similar argument works if $c_2c_3>0$ by taking $\gamma(t)=(\frac{2}{\sqrt{c_2c_3}}\tan(\frac{1}{2}\sqrt{c_2c_3}t),0,0)$ instead, which is an integral curve of $E_1$. Hence, we will assume that $c_1c_3\leq 0$ and $c_2c_3\leq 0$, and prove that the model is global.

If $c_3=0$, then $D\times\mathbb{R}=\R^2\times\mathbb{R}$ is globally isometric to a semidirect product with a left-invariant metric, and hence we are done (see Example~\ref{ex:semidirect-product}). Next, by changing all the signs of the $c_i$, we can assume that $c_3<0$ and $c_1,c_2\geq 0$. On the one hand, if $c_1>0$ and $c_2>0$, then we have seen in Example~\ref{ex:SL2R} that $D\times\mathbb{R}$ is globally isometric to $\widetilde{\mathrm{SL}}_2(\R)$, so we are done too. On the other hand, if $c_1=c_2=0$, then $D\times\R$ is $\R^3$ with the global orthonormal frame $E_1=\partial_x-\frac{c_3}{2}y\partial_z$, $E_2=\partial_y+\frac{c_3}{2}x\partial_z$ and $E_3=\partial_z$. This is the standard model for $\Nil(-\frac{c_3}{2})=\E(0,-\frac{c_3}{2})$, which is also global. We are left with the case $c_1=0$ and $c_2>0$ ($c_1>0$ and $c_2=0$ can be reasoned similarly). Note that this means $D\times\R$ is the slab of $\R^3$ defined by the inequality $1+\frac{c_2c_3}{4}x^2>0$. The map $F:D\times\R\to\R^3$ given by
	\[F(x,y,z)=\left(\frac{y-c_2xz}{1+\frac{c_2c_3}{4}x^2},\frac{\frac{c_3}{2}xy+z-\frac{c_2c_3}{4}x^2z}{1+\frac{c_2c_3}{4}x^2},\frac{2\arctanh(\tfrac{1}{2}\sqrt{-c_2c_3}x)}{\sqrt{-c_2c_3}}\right)\]
	is a bijection that takes the orthonormal frame $\{E_1,E_2,E_3\}$ in~\eqref{eqn:G-frame} into
	\begin{align*}
		F_*E_1&=\cosh(\sqrt{-c_2c_3}z)\partial_x+c_3\tfrac{\sinh(\sqrt{-c_2c_3}z)}{\sqrt{-c_2c_3}}\partial_y,\\
		F_*E_2&=-c_2\tfrac{\sinh(\sqrt{-c_2c_3}z)}{\sqrt{-c_2c_3}}\partial_x+\cosh(\sqrt{-c_2c_3}z)\partial_y,\\
		F_*E_3&=\partial_z.
	\end{align*}
	By comparison with Example~\ref{ex:semidirect-product} after rotating the indexes, $F$ is a global isometry from $D\times\R$ to the semidirect product $\R^2\rtimes_A\R$ with its standard metric, being $A=(\begin{smallmatrix}0&-c_2\\c_3&0\end{smallmatrix})$. We conclude that $D\times\R$ is also a global model in this case.
\end{proof}

Note that Riemannian metric Lie groups are always (geodesically) complete, so we can talk of completeness or globality, equally. However, this question is trickier in the Lorentzian world, in which there exist incomplete examples of metric Lie groups. This does affect our family of spaces, since some left-invariant metrics of diagonalizable type in $\mathrm{SL}_2(\R)$ are actually incomplete, see the recent work of Elshafei--Ferreira--Reis~\cite{ElsFerRei} and the references therein. Anyway, we will not use completeness in our arguments since most of them are purely local.

\section{The fundamental theorem}\label{sec:fundamental}

Consider an isometric immersion $\phi:\Sigma\to G$ of an oriented semi-Riemannian (non-degenerate) surface $\Sigma$ into the unimodular metric Lie group $G$ of diagonalizable type with structure constants $c_1,c_2,c_3$ and signs $\epsilon_1,\epsilon_2,\epsilon_3$. Let $\{e_1,e_2\}$ be a positively oriented orthonormal frame of $\Sigma$, where we will write $\hat\epsilon_i=\langle e_i,e_i\rangle$ for $i\in\{1,2\}$. This frame can be completed to a positively oriented orthonormal frame $\{e_1,e_2,N\}$, where we identify $T_p\Sigma$ and $\df\phi_p(T_p\Sigma)\subset T_{\phi(p)}G$ as usual. We will also consider $\hat\epsilon_3=\langle N,N\rangle$, the so called \emph{sign} of the immersion. Since both $\{e_1,e_2,N\}$ and $\{E_1,E_2,E_3\}$ are positively oriented, we find that $\{\hat\epsilon_1,\hat\epsilon_2,\hat\epsilon_3\}$ is a cyclic permutation of $\{\epsilon_1,\epsilon_2,\epsilon_3\}$. It is indeed possible to relabel indexes so $\hat\epsilon_i=\epsilon_i$, but until Section~\ref{sec:change-of-frame} we will keep alive both sets of signs to get fundamental equations that apply easily to the particular cases we are interested in:
\begin{itemize}
	\item Riemannian case: $(\hat\epsilon_1,\hat\epsilon_2,\hat\epsilon_3)=(\epsilon_1,\epsilon_2,\epsilon_3)=(1,1,1)$.
	\item Lorentzian spacelike case: $(\hat\epsilon_1,\hat\epsilon_2,\hat\epsilon_3)=(\epsilon_1,\epsilon_2,\epsilon_3)=(1,1,-1)$.
	\item Lorentzian timelike case: $(\hat\epsilon_1,\hat\epsilon_2,\hat\epsilon_3)=(1,-1,1)$ and $(\epsilon_1,\epsilon_2,\epsilon_3)=(1,1,-1)$.
\end{itemize}

The orientation-depending features will be in turn captured by the operator $J$ in the tangent bundle defined by $Jv=N\times v$ for all $v\in T_p\Sigma$ and $p\in\Sigma$, which satisfies $Je_1 = \hat{\epsilon}_2 e_2$ and $Je_2 = -\hat{\epsilon}_1 e_1$, and hence $J^2=-\hat{\epsilon}_1\hat{\epsilon}_2\mathrm{id}$. If $\hat{\epsilon}_1\hat{\epsilon}_2=1$, then $\Sigma$ is spacelike and $J$ is the usual $\frac\pi2$-rotation, whence $\{v,Jv,N\}$ is positively oriented for all non-zero $v$; if $\hat\epsilon_1\hat\epsilon_2=-1$, then $\Sigma$ is timelike and $J$ becomes an axial symmetry, so the orientation of $\{v,Jv,N\}$ depends on the causality of $v$ and the signs of $\hat\epsilon_1$ and $\hat\epsilon_2$ (it might even happen that $Jv=\pm v$ when $v$ is lightlike, in which case $\{v,Jv,N\}$ is not even a basis). Notice that $J$ is not an isometry of $T_p\Sigma$ if $\hat\epsilon_1\hat\epsilon_2=-1$ since we have $\langle Ju, Jv\rangle = \hat\epsilon_1 \hat\epsilon_2 \langle u, v\rangle$ for any tangent vectors $u$ and $v$.

\subsection{The fundamental data}\label{subsec:fundamental-data}
The tangent projections of the standard left-invariant vector fields and their angle functions are defined by
\begin{equation}\label{eqn:T-nu-definition}
T_i=E_i^\top=E_i-\hat\epsilon_3\nu_iN,\qquad \nu_i=\langle N,E_i\rangle,\qquad\text{for all }i\in\{1,2,3\}.
\end{equation}
Observe that $N=\sum_{i=1}^3\epsilon_i\nu_iE_i$ so these elements are not arbitrary, that is, they must satisfy some relations that follow from elementary linear algebra:
\begin{equation}\label{eqn:algebraic-relations}
\begin{aligned}
&\textstyle\sum_{i=1}^3\epsilon_i\nu_i^2=\langle N,N\rangle=\hat\epsilon_3,\qquad \textstyle\sum_{i=1}^3\epsilon_i\nu_iT_i=N^\top=0,\\
&\langle T_i,T_j\rangle = \epsilon_i\delta_i^j-\hat\epsilon_3\nu_i\nu_j,\quad \text{for all }i,j\in\{1,2,3\}.
\end{aligned}\end{equation}
From these relations and the fact that $JT_i=N\times E_i$, it is easy to show that 
\begin{equation}\label{eqn:JEi}
\begin{aligned}
JT_1&=\epsilon_2\epsilon_3(\nu_3T_2-\nu_2T_3),\\
JT_2&=\epsilon_3\epsilon_1(\nu_1T_3-\nu_3T_1),\\
JT_3&=\epsilon_1\epsilon_2(\nu_2T_1-\nu_1T_2).
\end{aligned}
\end{equation}

We will revisit briefly some basic facts about semi-Riemannian immersions, all of which are discussed by O'Neill~\cite[Ch.~3 and~4]{Oneill}, towards Gauss and Codazzi equations. The second fundamental form $\sigma$ of the immersion is the normal component of the connection, that is, we have
\[\overline\nabla_XY=\nabla_XY+\sigma(X,Y),\quad\text{for all }X,Y\in\mathfrak X(\Sigma),\]
and the non-degeneracy of the metric allows to write $\langle\sigma(X,Y),N\rangle=\langle SX,Y\rangle$, where $S$ is a smooth field of linear operators in $\Sigma$ called the \emph{shape operator} (associated to $N$). It follows that
\[S_p:T_p\Sigma\to T_p\Sigma,\qquad S_p(v) = -\overline\nabla_{v}N,\quad \text{for all } v\in T_p\Sigma.\]
This linear map is self-adjoint with respect to $\langle\cdot,\cdot\rangle$, but not necessarily diagonalizable if $\Sigma$ is Lorentzian. However, we can define the extrinsic and the mean curvature of the immersion, as the determinant and a half of the trace of $S$, respectively:
\begin{equation}\label{eqn:det-trace-A}
\begin{aligned}
\det(S)&=\hat\epsilon_1\hat\epsilon_2(\langle Se_1,e_1\rangle\langle Se_2,e_2\rangle-\langle Se_1,e_2\rangle^2),\\ 
H&=\tfrac{\hat\epsilon_3}{2}(\hat\epsilon_1\langle Se_1,e_1\rangle+\hat\epsilon_2\langle Se_2,e_2\rangle).
\end{aligned}\end{equation}
The sign of $H$ is chosen so that the mean curvature vector reads $\vec{H}=HN$. Note also that the traceless shape operator is $S-\hat\epsilon_3 H\,\mathrm{id}$, which will be useful later. The norm of the second fundamental form is $|\sigma|^2 = \sum_{i=1}^2 \hat\epsilon_i \langle S^2 e_i, e_i\rangle= 4H^2-2\det(S)$. The Gauss curvature is the intrinsic sectional curvature of the tangent plane, given by $K=\hat\epsilon_1\hat\epsilon_2\langle R(e_1,e_2)e_2,e_1\rangle$, where $R$ is the Riemann curvature tensor of the surface $\Sigma$. In this sense, the Gauss equation for this immersion reads
\begin{equation}\label{eqn:Gauss}
\begin{aligned}
K&=\hat\epsilon_1\hat\epsilon_2\langle\overline R(e_1,e_2)e_2,e_1\rangle+\hat\epsilon_3\det(S)\\
&=\hat\epsilon_3\det(S)-\hat\epsilon_1\hat\epsilon_2 \left(a_1\nu_1^2+a_2\nu_2^2+a_3\nu_3^2\right),\end{aligned}
\end{equation}
where we have considered the constant coefficients of the curvature tensor
\begin{equation}\label{eqn:ai}
a_1=\epsilon_1\mu_2\mu_3-c_1\mu_1,\qquad a_2=\epsilon_2\mu_1\mu_3-c_2\mu_2,\qquad a_3=\epsilon_3\mu_1\mu_2-c_3\mu_3.
\end{equation}
Note that~\eqref{eqn:Gauss} follows from Lemma~\ref{lemma:R}, for we can work out
\begin{align*}
\langle\overline R_i(e_1,e_2)e_2,e_1\rangle&=-\hat\epsilon_1\hat\epsilon_2+\epsilon_i\hat\epsilon_1\langle e_2,T_i\rangle^2+\epsilon_i\hat\epsilon_2\langle e_1,T_i\rangle^2\\
&=\hat\epsilon_1\hat\epsilon_2(-1+\epsilon_i\langle\hat\epsilon_2\langle e_2,T_i\rangle e_2+\hat\epsilon_1\langle e_1,T_i\rangle e_1,T_i\rangle)\\
&=\hat\epsilon_1\hat\epsilon_2(-1+\epsilon_i\langle T_i,T_i\rangle)=-\hat\epsilon_1\hat\epsilon_2\hat\epsilon_3\epsilon_i\nu_i^2.
\end{align*}
A similar computation, again by Lemma~\ref{lemma:R}, gives Codazzi equation:
\begin{equation}\label{eqn:codazzi}
\nabla_XSY-\nabla_YSX-S[X,Y]=-\overline R(X,Y)N=\epsilon_1\epsilon_2\epsilon_3(\langle X,T\rangle Y-\langle Y,T\rangle X),
\end{equation}
where we have employed the auxiliary tangent vector field
\begin{equation}\label{eqn:T}
\begin{aligned}
	T&=a_1\nu_1 T_1+a_2\nu_2 T_2+a_3\nu_3T_3\\
   &=2\epsilon_1\mu_2\mu_3\nu_1 T_1+2\epsilon_2\mu_3\mu_1\nu_2 T_2+2\epsilon_3\mu_1\mu_2\nu_3 T_3.
\end{aligned}
\end{equation}
The second expression for $T$ in~\eqref{eqn:T} can be deduced by writing the $c_i$ in terms of the $\mu_i$ and using the algebraic relations~\eqref{eqn:algebraic-relations}.

\begin{proposition}\label{prop:compatibility}
Let $\Sigma$ be an orientable Riemannian or Lorentzian surface isometrically immersed into the unimodular metric Lie group $G$ of diagonalizable type. The tuple $(S,J,T_1,T_2,T_3,\nu_1,\nu_2,\nu_3)$ satisfies the following conditions:
\begin{enumerate}[label=$($\textsc{\roman*}$)$]
	\item Gauss equation: $K=\hat\epsilon_3\det(S)-\hat\epsilon_1\hat\epsilon_2 (a_1\nu_1^2+a_2\nu_2^2+a_3\nu_3^2)$, \label{eqn:comp-i}
	\item Codazzi equation: $\nabla_XSY-\nabla_YSX-S[X,Y]=\epsilon_1\epsilon_2\epsilon_3(\langle X,T\rangle Y-\langle Y,T\rangle X)$,\label{eqn:comp-ii}
	\item Algebraic relations: $\langle T_i,T_j\rangle=\epsilon_i\delta^j_i-\hat\epsilon_3\nu_i\nu_j$, for all $i,j\in\{1,2,3\}$,\label{eqn:comp-iii}

  \item $\nabla_X T_1=\hat\epsilon_3\nu_1SX+\epsilon_2\epsilon_3(\mu_3\langle X,T_3\rangle T_2-\mu_2\langle X,T_2\rangle T_3)$,\label{eqn:comp-iv}

  \noindent $\nabla_X T_2=\hat\epsilon_3\nu_2SX+\epsilon_1\epsilon_3(\mu_1\langle X,T_1\rangle T_3-\mu_3\langle X,T_3\rangle T_1)$,

  \noindent $\nabla_X T_3=\hat\epsilon_3\nu_3SX+\epsilon_1\epsilon_2(\mu_2\langle X,T_2\rangle T_1-\mu_1\langle X,T_1\rangle T_2)$, for all $X\in\mathfrak{X}(\Sigma)$,

  \item $\nabla\nu_1=-ST_1+\epsilon_2\epsilon_3(\mu_3\nu_2 T_3-\mu_2\nu_3 T_2)$,\label{eqn:comp-v}

  \noindent $\nabla\nu_2=-ST_2+\epsilon_1\epsilon_3(\mu_1\nu_3 T_1-\mu_3\nu_1 T_3)$,

  \noindent $\nabla\nu_3=-ST_3+\epsilon_1\epsilon_2(\mu_2\nu_1 T_2-\mu_1\nu_2 T_1)$.
\end{enumerate}
\end{proposition}

\begin{proof}
Items~\ref{eqn:comp-i},~\ref{eqn:comp-ii} and~\ref{eqn:comp-iii} have been already discussed, so we will focus on~\ref{eqn:comp-iv} and~\ref{eqn:comp-v}. To this end, we will write $\overline{\nabla}_X E_i$ in two different ways. On the one hand,
\begin{equation}\label{prop:compatibility:eqn1}
\begin{aligned}
\overline\nabla_XE_i=\overline{\nabla}_X (T_i + \hat\epsilon_3 \nu_i N) &= \nabla_X T_i + \sigma(X,T_i) + \hat\epsilon_3\langle\nabla\nu_i,X\rangle N + \hat\epsilon_3 \nu_i \overline{\nabla}_X N\\
&= (\nabla_X T_i - \hat\epsilon_3 \nu_i SX) + \hat\epsilon_3(\langle ST_i,X\rangle + \langle \nabla\nu_i,X\rangle) N.
\end{aligned}\end{equation}
On the other hand, $\overline{\nabla}_X E_i = \sum_{j=1}^3 \epsilon_j \langle X, E_j\rangle \overline{\nabla}_{E_j}E_i$. Expanding this last identity by means of~\eqref{eqn:nabla-unimodular} and comparing with~\eqref{prop:compatibility:eqn1}, the tangent components give the equations for $\nabla_X T_i$ whilst the normal components give the equations for $\nabla \nu_i$.
\end{proof}

\begin{definition}
 	The above elements $S,J,T_1,T_2,T_3,\nu_1,\nu_2,\nu_3$ on a non-degenerate surface $\Sigma$ will be called the \emph{fundamental data}. All the equations in Proposition~\ref{prop:compatibility} are necessary to have an immersion and will be called the \emph{fundamental equations}.
 \end{definition} 

 \begin{remark}
 	This definition extends the previous definition in $\mathbb{E}(\kappa,\tau)$-spaces~\cite{Daniel07}. Note that the Riemannian or Lorentzian metric $\df s^2$ of $\Sigma$ is usually assumed explicitly as another datum, and this is also our case since we are prescribing $\Sigma\equiv(\Sigma,\df s^2)$ as a whole, see also~\cite{CMT}. Also, the rotation $J$ has been included to highlight the role of the orientation and how it must be compatible with the definition of the normal; equivalently, one can assume that $\Sigma$ is oriented \emph{a priori}.  

 	Notice that the fundamental equations have a high degree of redundancy, e.g., one can differentiate~\ref{eqn:comp-iii} in the direction of a vector field $X\in\mathfrak{X}(\Sigma)$ to obtain some relations between between~\ref{eqn:comp-iv} and~\ref{eqn:comp-v}. This is similar to the case of $\E(\kappa,\tau)$-spaces as pointed out in~\cite[Rmk.~4.2]{Daniel07}. However, our main goal is to prove that either $T_1,T_2,T_3$ or $\nu_1,\nu_2,\nu_3$ are enough to decide if the immersion exists and to recover it just in terms of some subset of the above compatibility equations.
 \end{remark}

\begin{remark}\label{rmk:stabilizer-data}
Given an isometric immersion $\phi:\Sigma\to G$ and $\Psi\in\Iso(G)$, we get another isometric immersion $\widetilde\phi=\Psi\circ\phi:\Sigma\to G$. If $\Psi$ is a left-translation, then $\phi$ and $\widetilde\phi$ clearly have the same fundamental data. If $\Psi\in\Stab_p$ preserves the orientation and is not the identity, then $\Psi$ is an axial symmetry about one of the directions $\{E_1,E_2,E_3\}$ (see~\cite[Prop.~2.21]{MeeksPerez12}). In our Cartan model, $\Psi(x,y,z)$ is equal to $(x,-y,-z)$, $(-x,y,-z)$ or $(-x,-y,z)$, so $\Psi$ preserves the orientation and the change of fundamental data from $\phi$ to $\widetilde\phi$ is straightforward. For instance, the symmetry about $E_1$ preserves $S$ and $J$ while it changes $(T_1,T_2,T_3)\mapsto(T_1,-T_2,-T_3)$ and $(\nu_1,\nu_2,\nu_3)\mapsto(\nu_1,-\nu_2,-\nu_3)$. The other rotations about $E_2$ and $E_3$ are similar. This observation will simplify some forthcoming discussions.

In the Riemannian case (which will be used in Section~\ref{sec:lawson}), the only space in which there are orientation-reversing isometries in the stabilizers is $\mathrm{Sol}_3$ when a metric homothetic to the standard one is considered (i.e., $c_1=-c_2$ and $c_3=0$). In our model, such an isometry reads $\Psi(x,y,z)=(y,x,z)$, and yields the changes $(T_1,T_2,T_3)\mapsto(T_2,T_1,T_3)$ and $(\nu_1,\nu_2,\nu_3)\mapsto(\nu_2,\nu_1,\nu_3)$. The particularity of this case is that swapping $c_1$ and $c_2$ is the same as changing the signs of all $c_1,c_2,c_3$.
\end{remark}

\subsection{The change-of-frame equation}\label{sec:change-of-frame}
Consider the global orthonormal frame $\{e_1,e_2,e_3\}$ along $\Sigma$, where we will denote $e_3=N$ in what follows and assume that the signs of $\Sigma$ and $G$ agree in the sense that $(\hat\epsilon_1,\hat\epsilon_2,\hat\epsilon_3)=(\epsilon_1,\epsilon_2,\epsilon_3)$. No generality is lost in this assumption for we can always relabel indexes, but it has the side effect that studying spacelike and timelike surfaces in the same space requires two separated cases, as explained at the beginning of Section~\ref{sec:fundamental} (this will actually take place when we distinguish cases A and B in Section~\ref{subsec:daniel-fundamental}).

The frame $\{e_1,e_2,e_3\}$ can be written in terms of the ambient frame $\{E_1,E_2, E_3\}$ as $e_\beta=\sum_{\alpha=1}^3M_\beta^\alpha E_\alpha$ (i.e., $M^\alpha_\beta=\epsilon_\alpha \langle e_\beta, E_\alpha\rangle$), and the matrix $M$ is a change-of-basis matrix that satisfies $M^t\mathcal{E}M=\mathcal{E}$ and $\det(M)=1$, where $\mathcal{E}$ is the constant diagonal matrix with entries $\{\epsilon_1,\epsilon_2,\epsilon_3\}$ in the diagonal, since both bases are orthonormal for the semi-Riemannian metric (in the notation $M^\alpha_\beta$, the index $\alpha$ indicates the row, whereas the index $\beta$ is the column). We will denote by $\SO^\epsilon_3(\mathbb{R})$ the Lie group of such matrices and by $\mathfrak{so}_3^\epsilon(\mathbb{R})$ its Lie algebra, which consists of matrices $A$ such that $\mathcal{E}A\mathcal{E}=-A^t$, or equivalently $\epsilon_\alpha A^\alpha_\beta + \epsilon_\beta A^\beta_\alpha = 0$ for all $\alpha,\beta\in\{1,2,3\}$. Hence, our assumption on the signs $\hat\epsilon_i=\epsilon_i$ allows us to work directly in this setting of Lie groups where the theory of moving frames applies smoothly.

The dual frame $\omega^1,\omega^2,\omega^3\in\Omega^1(\Sigma)$ is defined as the $1$-forms such that $\omega^i(e_j)=\delta_j^i$ and $\omega^3\equiv 0$ for $i,j\in\{1,2\}$. The curvature $1$-forms $\omega_j^i\in\Omega^1(\Sigma)$ are defined by
\begin{equation}\label{eqn:curvature-forms}
	\textstyle\nabla_{e_k}e_j=\sum_{i=1}^2\omega_j^i(e_k)e_i,\quad Se_k=-\sum_{i=1}^2\omega_3^i(e_k)e_i.
\end{equation}
Since the relations $\epsilon_j\omega_i^j + \epsilon_i\omega_j^i =0$ hold true for $i,j\in \{1,2\}$, this suggests to define additionally $\omega_j^3=-\epsilon_3\epsilon_j\omega^j_3$ and $\omega^3_3\equiv 0$ and gather all these $1$-forms in a matrix $\Omega=(\omega_\beta^\alpha)$ with $\mathfrak{so}_3^\epsilon$-symmetry. These forms have the following direct expressions, which are convenient in the forthcoming computations:
\begin{equation}\label{eqn:curvature-forms-as-products}
	\omega_j^i(e_k)=\epsilon_i \langle\nabla_{e_k}e_j,e_i\rangle,\quad \omega_3^j(e_k)=-\epsilon_3\epsilon_j\omega_j^3(e_k)=-\epsilon_j\langle Se_k,e_j\rangle,\quad \omega_3^3\equiv 0,
\end{equation}
for $i,j,k\in\{1,2\}$. It is worth emphasizing that $\omega^i_j$ are purely intrinsic for $i,j\in\{1,2\}$, whereas $\omega^i_3$ and $\omega^3_j$ need some extrinsic control coming from the shape operator. This helps understand the idea behind moving frames. 

The exterior derivatives of all these $1$-forms are well known (e.g., the proof in~\cite[Prop.~2.4]{Daniel09} can be easily adapted to the semi-Riemannian case):
\begin{align}
	\textstyle \df\omega^i + \sum_{k=1}^2\omega^i_k \wedge \omega^k &= 0,\label{eqn:domega1}\\
	\textstyle \sum_{k=1}^2 \omega^3_k \wedge \omega^k &= 0,\label{eqn:domega2}\\
	\textstyle \df\omega^i_j + \sum_{k=1}^2 \omega^i_k \wedge \omega^k_j &= \textstyle \frac{\epsilon_i}{2}\sum_{h,l=1}^2  \langle R(e_h, e_l)e_j,e_i\rangle \omega^h \wedge \omega^l,\label{eqn:domega3}\\
	\textstyle \df\omega^3_j + \sum_{k=1}^2 \omega^3_k \wedge \omega^k_j &= \textstyle \frac{\epsilon_3}{2}\sum_{h,l=1}^2  \langle \nabla_{e_h} Se_l - \nabla_{e_l}Se_h - S[e_h, e_l], e_j\rangle \omega^h\wedge \omega^l.\label{eqn:domega4}
\end{align}

\begin{lemma}\label{lem:lie-equation}
	Under the above assumptions,
	\begin{equation}\label{lem:lie-equation:eqn1}
	M^{-1}\df M = \Omega + L(M),
	\end{equation}
	where 
	\begin{equation}\label{lem:lie-equation:eqn2}
	L(M)^\alpha_\beta = -\epsilon_\alpha\sum_{i=1}^2\left(\sum_{\gamma, \delta, \zeta=1}^3\epsilon_\zeta M^\zeta_\alpha M^\gamma_i M^\delta_\beta \overline{\Gamma}{}^{\zeta}_{\gamma \delta}\right)\omega^i,
	\end{equation}
 and $\overline\Gamma_{ij}^k=\epsilon_k\langle\overline\nabla_{E_i}E_j,E_k\rangle$ denote the ambient Christoffel symbols.
\end{lemma}

\begin{proof}
	Let $k\in\{1,2\}$ and $\beta\in\{1,2,3\}$. Since $e_\beta = \sum_{\delta=1}^3 M^\delta_\beta E_\delta$, we deduce that
	\[\overline{\nabla}_{e_k}e_\beta=\sum_{\delta=1}^3\left(\df M^\delta_\beta(e_k)E_\delta + M^\delta_\beta \overline{\nabla}_{e_k}E_\delta\right) = 
	\sum_{\zeta=1}^3 \df M^\zeta_\beta(e_k)E_\zeta + \!\!\!\!\sum_{\zeta,\gamma,\delta=1}^3\!\! M^\delta_\beta M^\gamma_{k} \overline{\Gamma}{}^\zeta_{\gamma \delta}E_\zeta.\]
	We can also express $\overline{\nabla}_{e_k}e_\beta = \sum_{\delta=1}^3 \omega^\delta_\beta(e_k)e_\delta$ by means of the curvature forms~\eqref{eqn:curvature-forms}, whence a change of basis yields $\overline{\nabla}_{e_k}e_\beta = \sum_{\zeta, \delta=1}^3 \omega^\delta_\beta(e_k) M^\zeta_\delta E_\zeta$. By matching coefficients in the above two expressions, we reach
	\begin{equation}\label{lem:lie-equation:eqn3}
	\df M^\zeta_\beta(e_k) = \sum_{\delta=1}^3 M^\zeta_\delta\omega^\delta_\beta(e_k)- \sum_{i=1}
	^2\sum_{\gamma,\delta=1}^3 M^\delta_\beta M^\gamma_i \overline{\Gamma}{}^\zeta_{\gamma \delta}\omega^i(e_k),
	\end{equation}
	for all $\zeta\in\{1,2,3\}$. The condition $M^{-1}=\mathcal{E}M^t\mathcal{E}$ tells us that the element in row $\alpha$ and column $\zeta$ of $M^{-1}$ equals $\epsilon_\alpha\epsilon_\zeta M^\zeta_\alpha$, so that multiplying both sides of~\eqref{lem:lie-equation:eqn3} by $\epsilon_\alpha\epsilon_\zeta M^\zeta_\alpha$ and taking the sum over $\zeta$ yields the desired identity~\eqref{lem:lie-equation:eqn1}.
\end{proof}

Notice that $L(M)$ has $\mathfrak{so}^\epsilon_3$-symmetry and there are only six non-vanishing Christoffel symbols in view of~\eqref{eqn:nabla-unimodular}, so each $1$-form $L(M)^\alpha_\beta$ in Lemma~\ref{lem:lie-equation} has at most six terms in $\omega^1$ and six terms in $\omega^2$. More precisely, we can expand~\eqref{lem:lie-equation:eqn2} as
\begin{align*}
L(M)^\alpha_\beta(e_k)&=\epsilon_\alpha\Bigl(\mu_1M^1_k(M^2_\alpha M^3_\beta-M^3_\alpha M^2_\beta)+\mu_2M^2_k(M^3_\alpha M^1_\beta-M^1_\alpha M^3_\beta)\\
&\qquad\qquad+\mu_3M^3_k(M^1_\alpha M^2_\beta-M^2_\alpha M^1_\beta)\Bigr).
\end{align*}
The terms of the form $M^\gamma_\alpha M^\delta_\beta-M^\delta_\beta M^\gamma_\alpha$ show up in the adjoint matrix $\adj(M^t)$, which can be retrieved as $M^{-1}=\mathcal{E}M^t\mathcal{E}$. This allows us to rewrite the matrix of $1$-forms $L(M)$ in a much simpler way:
\begin{equation}\label{eqn:L}
	L(M) =\left(\begin{smallmatrix}
		0 & \epsilon_1\epsilon_3\sum_{\gamma=1}^3\epsilon_\gamma\mu_\gamma M^\gamma_3 \eta^\gamma & -\epsilon_1\epsilon_2\sum_{\gamma=1}^3\epsilon_\gamma\mu_\gamma M^\gamma_2 \eta^\gamma\\
		-\epsilon_2\epsilon_3\sum_{\gamma=1}^3\epsilon_\gamma\mu_\gamma M^\gamma_3 \eta^\gamma & 0 & \epsilon_1\epsilon_2\sum_{\gamma=1}^3\epsilon_\gamma\mu_\gamma M^\gamma_1 \eta^\gamma\\
		\epsilon_2\epsilon_3\sum_{\gamma=1}^3\epsilon_\gamma\mu_\gamma M^\gamma_2 \eta^\gamma & -\epsilon_1\epsilon_3\sum_{\gamma=1}^3\epsilon_\gamma\mu_\gamma  M^\gamma_1 \eta^\gamma & 0
	\end{smallmatrix}\right),
\end{equation}
where we have considered the $1$-forms $\eta^\gamma = M^\gamma_1 \omega^1 + M^\gamma_2 \omega^2$ for each $\gamma\in\{1,2,3\}$, which are just the entries of the vector of $1$-forms $M\omega$. Consequently, we can fully expand the first and second columns of the identity $\df M=M\Omega+ML(M)$ and check that they equivalent to the following six equations:
\begin{equation}\label{eqn:dM}
	\begin{aligned}
	\df M^1_1&=M^1_2\omega^2_1+M^1_3\omega^3_1+\epsilon_1(\mu_3-\mu_2)M^2_1M^3_1\omega^1+\epsilon_1(\mu_3M^2_1M^3_2-\mu_2M^3_1M^2_2)\omega^2,\\
	\df M^1_2&=M^1_1\omega^1_2+M^1_3\omega^3_2+\epsilon_1(\mu_3M^3_1M^2_2-\mu_2M^2_1M^3_2)\omega^1+\epsilon_1(\mu_3-\mu_2)M^3_2M^2_2\omega^2,\\
	\df M^2_1&=M^2_2\omega^2_1+M^1_3\omega^3_1+\epsilon_2(\mu_1-\mu_3)M^3_1M^1_1\omega^1+\epsilon_2(\mu_1M^3_1M^1_2-\mu_3M^1_1M^3_2)\omega^2,\\
	\df M^2_2&=M^2_1\omega^1_2+M^2_3\omega^3_2+\epsilon_2(\mu_1M^1_1M^3_2-\mu_3M^3_1M^1_2)\omega^1+\epsilon_2(\mu_1-\mu_3)M^1_2M^3_2\omega^2,\\
	\df M^3_1&=M^3_2\omega^2_1+M^3_3\omega^3_1+\epsilon_3(\mu_2-\mu_1)M^2_1M^1_1\omega^1+\epsilon_3(\mu_2M^1_1M^2_2-\mu_1M^2_1M^1_2)\omega^2,\\
	\df M^3_2&=M^3_1\omega^1_2+M^3_3\omega^3_2+\epsilon_3(\mu_2M^2_1M^1_2-\mu_1M^1_1M^2_2)\omega^1+\epsilon_3(\mu_2-\mu_1)M^1_2M^2_2\omega^2.
\end{aligned}
\end{equation}
We will also need the wedge products $\eta^\alpha\wedge\eta^\beta$, which can be computed as
\begin{equation}\label{eqn:eta-wedge}
\begin{aligned}
	\eta^1 \wedge\eta^2 &=(M^1_1 M^2_2 - M^1_2 M^2_1)\;\omega^1\wedge \omega^2 = M^3_3\;\omega^1 \wedge \omega^2,\\
	\eta^2 \wedge \eta^3 &= (M^2_1 M^3_2 - M^2_2 M^3_1)\;\omega^1\wedge \omega^2 = \epsilon_1\epsilon_3 M^1_3\;\omega^1 \wedge \omega^2,\\
	\eta^3 \wedge \eta^1 &= (M^1_2 M^3_1 - M^1_1 M^3_2 )\;\omega^1 \wedge \omega^2 = \epsilon_2\epsilon_3 M^2_3\;\omega^1 \wedge \omega^2.
\end{aligned}\end{equation}

Lemma~\ref{lem:lie-equation} provides a necessary equation for the change-of-frame matrix of an isometric immersion $\phi:\Sigma\to G$. Note that $M$ is not unique since it depends upon the choice of the positive orthonormal frame $\{e_1,e_2,e_3\}$. However, we will now prove the converse statement that, if there is a matrix valued smooth map $M:\Sigma\to\mathrm{SO}_3^\epsilon(\R)$ satisfying $M^{-1}\df M=\Omega+L(M)$, then there is a unique immersion up to left-translations. The aforesaid non-uniqueness of $M$ can be consequently explained by the elements of the stabilizers.

Consider the generalized Cartan model $G\cong D\times\R$ in~\eqref{eqn:CartanMetric} and write the immersion in coordinates $\phi=(\phi_1,\phi_2,\phi_3):\Sigma\to D\times\R\subset\R^3$. Notice that the following equations must hold true for all $\alpha\in\{1,2,3\}$:
\begin{equation}\label{eqn:second-equation}
	\begin{aligned}
		\df\phi_\alpha=\sum_{k=1}^2\df\phi_\alpha(e_k)\omega^k&=\sum_{k=1}^2\sum_{\beta=1}^3M^\beta_k\pi_\alpha(E_\beta)\omega^k\\
		&=\sum_{k=1}^2\sum_{\beta,\gamma=1}^3B^\beta_\gamma M^\beta_k\delta^\gamma_\alpha\omega^k=\sum_{k=1}^2\sum_{\beta=1}^3B^\beta_\alpha M^\beta_k\omega^k,
\end{aligned}\end{equation}
where $\pi_\alpha$ is the $\alpha$-th coordinate in the basis $\{\partial_x,\partial_y,\partial_z\}$ and $B$ is the change-of-frame matrix from $\{E_1,E_2,E_3\}$ to $\{\partial_x,\partial_y,\partial_z\}$. From~\eqref{eqn:G-frame}, we get that
\begin{equation}\label{eqn:B}
B(x,y,z)=\begin{pmatrix}
	\frac{c(z)}{\lambda(x,y)} & \frac{-c_1s(z)}{\lambda(x,y)} & 0\\
	\frac{c_2s(x)}{\lambda(x,y)} & \frac{c(z)}{\lambda(x,y)} & 0\\
\frac{c_3}{2}\left(c_2x\,s(z)-yc(z)\right) & \frac{c_3}{2}\left(x\,c(z)+c_1y\,s(z)\right) & 1
\end{pmatrix}.
\end{equation}
so that~\eqref{eqn:second-equation} can be written as a product of matrices $\df\phi=(B\circ\phi)M\omega$, where $\omega=(\omega^1,\omega^2,\omega^3)^t$. This equation depends on our choice of model and encodes the relation between the matrix $M$ and the coordinates of the immersion into $D\times\R$.

\begin{proposition}\label{prop:second-integration}
Let $\Sigma$ be a simply connected semi-Riemannian surface and assume that $M:\Sigma\to\mathrm{SO}_3^\epsilon(\R)$ is a smooth matrix-valued function satisfying $M^{-1}\df M=\Omega+L(M)$ in $\Sigma$, where the entries $\omega^3_1$ and $\omega^3_2$ of $\Omega\in\mathfrak{so}_3^\epsilon(\Omega^1(\Sigma))$ are any $1$-forms satisfying $\omega^3_1\wedge\omega^1+\omega^3_2\wedge\omega^2=0$. Given $p_0\in\Sigma$ and $q_0\in G$, there is a unique smooth immersion $\phi:\Sigma\rightarrow G$ such that $\phi(p_0)=q_0$ and $\df\phi = (B \circ \phi)M\omega$ in $\Sigma$.
\end{proposition}

\begin{proof}
We will proceed as in the proof of~\cite[Prop.~4.11]{Daniel07} by showing the integrability of the kernel of $\Lambda=B^{-1}\df\chi-M\omega$, where $\chi=(x,y,z)^t$ is the position vector. Note that $\Lambda$ is an array containing three $1$-forms whose kernel is a $3$-dimensional distribution in the $5$-dimensional manifold $\Sigma\times D\times\R$. However, our argument to prove integrability will be different from~\cite{Daniel07}. 

Since $c(z)^2+c_1c_2s(z)^2=1$, we can invert the matrix $B$ in~\eqref{eqn:B} to get
\begin{equation}\label{lem:second-integral:eqn1}
B^{-1}=\begin{pmatrix}
	\lambda(x,y)c(z) & c_1 \lambda(x,y)s(z) & 0 \\
  -c_2\lambda(x,y) s(z) &\lambda(x,y) c(z) & 0 \\
 \frac{c_3}{2}y\lambda(x,y) & -\frac{c_3}{2}  x\lambda(x,y)  & 1
\end{pmatrix}.
\end{equation}
Since $M\omega=\eta$, we can use~\eqref{lem:second-integral:eqn1} to rewrite the vector of $1$-forms $\Lambda$ as
\begin{equation}\label{lem:second-integral:eqn2}
\Lambda=B^{-1}\df\chi-\eta=\begin{pmatrix}
	\lambda(x,y)c(z)\df x+ c_1 \lambda(x,y)s(z)\df y-\eta^1\\
	-c_2\lambda(x,y) s(z)\df x +\lambda(x,y) c(z)\df y-\eta^2\\
	\frac{c_3}{2}y\lambda(x,y)\df x-\frac{c_3}{2}  x\lambda(x,y)\df y+\df z-\eta^3
\end{pmatrix}
\end{equation}
In order to compute $\df\Lambda$, we will employ Equation~\eqref{lem:lie-equation:eqn1}, which is assumed by hypothesis. We will begin by working out $\df\eta^\gamma$ as
\begin{equation}\label{lem:second-integral:eqn3}
\df \eta^\gamma=\df(M^\gamma_1\omega^1+M^\gamma_2\omega^2)=\df M^\gamma_1\wedge\omega^1+M^\gamma_1\df\omega^1+\df M^\gamma_2\wedge\omega^2+M^\gamma_2\df\omega^2.
\end{equation}
We can further compute $\df M^\gamma_1$ and $\df M^\gamma_2$ by means of~\eqref{eqn:dM}, as well as $\df\omega^1=-\omega^1_2\wedge\omega^2$ and $\df\omega^2=-\omega^2_1\wedge\omega^1$ in view of the intrinsic equation~\eqref{eqn:domega1}. Most terms in~\eqref{lem:second-integral:eqn3} simplify after these substitutions, and we must also use the assumption that $\omega^3_1\wedge\omega^1+\omega^3_2\wedge\omega^2=0$. We finally use~\eqref{eqn:eta-wedge} to get the differentials
\begin{equation}\label{lem:second-integral:eqn4}
\begin{aligned}
	\df\eta^1&=-\epsilon_1\epsilon_3c_1M^1_3\omega^1\wedge \omega^2=-c_1\eta^2\wedge\eta^3,\\
	\df\eta^2&=-\epsilon_2\epsilon_3c_2M^2_3\omega^1\wedge \omega^2=-c_2\eta^3\wedge\eta^1,\\
	\df\eta^3&=-c_3M^3_3\omega^1\wedge \omega^2=-c_3\eta^1\wedge\eta^2.\\
\end{aligned}
\end{equation}
Now the strategy is as follows: we will differentiate $\Lambda$ in~\eqref{lem:second-integral:eqn2} by considering the differentials in~\eqref{lem:second-integral:eqn4}, so that everything will be expressed in terms of $\df x\wedge\df y$, $\df y\wedge\df z$, $\df z\wedge\df x$ and $\eta^i\wedge\eta^j$. Then, we will check that $\df\Lambda(\xi_1,\xi_2)=0$ at elements $\xi_1,\xi_2$ of a tangent space of $\Sigma\times D\times\R$ such that $\Lambda(\xi_1)=\Lambda(\xi_2)=0$, but this amounts to substituting (in the expression for $\df\Lambda$) the values of $\eta^1,\eta^2,\eta^3$ in terms of $\df x,\df y,\df z$ that arise when all entries of~\eqref{lem:second-integral:eqn2} are equal to zero. We will do this for the first entry of $\df\Lambda$ because the other two are similar.
\begin{align*}
\df\Lambda^1&=\bigl(\tfrac{-c_2c_3}{2}x\lambda^2c(z)\df x-\tfrac{c_1c_3}{2}y\lambda^2c(z)\df y-c_1c_2\lambda s(z)\df z\bigr)\wedge\df x\\
&\qquad+c_1\bigl(\tfrac{-c_2c_3}{2}x\lambda^2s(z)\df x-\tfrac{c_1c_3}{2}y\lambda^2s(z)\df y+\lambda c(z)\df z\bigr)\wedge\df y-\df\eta^1\\
&=\tfrac{c_1c_3}{2}\lambda^2(yc(z)-c_2xs(z))\df x\wedge\df y\\&
\qquad-c_1\lambda c(z)\df y\wedge\df z-c_1c_2\lambda s(z)\df z\wedge\df x+c_1\eta^2\wedge\eta^3.
\end{align*}
This vanishes at elements with $\Lambda=0$, since in that case we get from~\eqref{lem:second-integral:eqn2} that
\begin{align*}
\eta^2\wedge\eta^3&=(-c_2\lambda s(z)\df x +\lambda c(z)\df y)\wedge(\tfrac{c_3}{2}y\lambda\df x-\tfrac{c_3}{2}  x\lambda\df y+\df z)\\
&=\tfrac{-c_3}{2}\lambda^2(yc(z)-c_2xs(z))\df x\wedge\df y+\lambda c(z)\df y\wedge\df z+c_2\lambda s(z)\df z\wedge\df x.
\end{align*}
The existence of a local immersion around $p_0\in\Sigma$ follows as in the last paragraph of~\cite[Prop.~4.11]{Daniel07}. The global existence of the immersion follows from the simple connectedness of $\Sigma$ by a standard continuation argument (see~\cite[Thm.~3.3]{Daniel09}).
\end{proof}

\begin{remark}\label{rmk:two-columns}
The above proof shows that if we just check that the first two columns of $\df M$ and $M\Omega+ML(M)$ agree, then the result still holds. This is not surprising because the condition $M\in\mathrm{SO}_3^\epsilon(\R)$ implies that two columns (or rows) of $M$ determine the third one since they form a positive orthonormal frame with respect to the metric given by $\mathcal{E}$. This observation will simplify some arguments later.
\end{remark}

\subsection{Recovering the immersion from the tangent projections}\label{sec:tangent}

We have seen in the previous section necessary and sufficient conditions on the (derivatives of the) change-of-frame matrix $M$ for the existence of an immersion into the semi-Riemannian metric Lie group $G$. This raises the natural question of finding geometric conditions that ensure the existence of such an $M$ satisfying~\eqref{lem:lie-equation:eqn1} and how to define $\Omega$ so that~\eqref{eqn:domega2} also holds true.

We will begin by the simple observation that the tangent vector fields $T_1,T_2,T_3$ and the angle functions $\nu_1,\nu_2,\nu_3$ determine uniquely the matrix $M$ \emph{a posteriori}. By this, we mean that, if the immersion exists, we can recover $M$ as
\begin{equation}\label{eqn:M}
M=\begin{pmatrix}
\epsilon_1\langle T_1,e_1\rangle&\epsilon_1\langle T_1,e_2\rangle&\epsilon_1\nu_1\\
\epsilon_2\langle T_2,e_1\rangle&\epsilon_2\langle T_2,e_2\rangle&\epsilon_2\nu_2\\
\epsilon_3\langle T_3,e_1\rangle&\epsilon_3\langle T_3,e_2\rangle&\epsilon_3\nu_3
\end{pmatrix}.
\end{equation}
The last column is a bit irrelevant in the sense that it can recovered uniquely as the semi-Riemannian cross product of the first two columns:
\begin{equation}\label{eqn:nui-from-Ti}
\begin{aligned}
	\epsilon_1\nu_1&=\epsilon_1\epsilon_2(\langle T_2,e_1\rangle\langle T_3,e_2\rangle-\langle T_2,e_2\rangle\langle T_3,e_1\rangle)=\langle JT_2,T_3\rangle,\\
\epsilon_2\nu_2&=\epsilon_1\epsilon_2(\langle T_3,e_1\rangle\langle T_1,e_2\rangle-\langle T_3,e_2\rangle\langle T_1,e_1\rangle)=\langle JT_3,T_1\rangle,\\
\epsilon_3\nu_3&=\epsilon_1\epsilon_2(\langle T_1,e_1\rangle\langle T_2,e_2\rangle-\langle T_1,e_2\rangle\langle T_2,e_1\rangle)=\langle JT_1,T_2\rangle.
\end{aligned}
\end{equation}
Next result shows that the shape operator is also uniquely determined and the compatibility equations~\ref{eqn:comp-iv} are enough to recover the immersion. 

\begin{theorem}\label{thm:fundamental}
Let $G$ be the semi-Riemannian unimodular metric Lie group of diagonalizable type with structure constants $c_1,c_2,c_3$ and signs $\epsilon_1,\epsilon_2,\epsilon_3$.	Let $\Sigma$ be a simply connected semi-Riemannian surface with a positive orthonormal frame $\{e_1,e_2\}$ with signature $\epsilon_1,\epsilon_2$, and assume that $\Sigma$ has three tangent vectors fields $T_1,T_2,T_3$ such that $\sum_{\alpha=1}^3\epsilon_\alpha\langle T_\alpha,e_i\rangle\langle T_\alpha,e_j\rangle=\epsilon_i\delta^j_i$ for all $i,j\in\{1,2\}$. Consider the smooth functions $\nu_1,\nu_2,\nu_3$ given by~\eqref{eqn:nui-from-Ti}. If the compatibility equations~\ref{eqn:comp-iv} hold true for some field of self-adjoint operators $S$, then there is a unique isometric immersion $\phi:\Sigma\to G$ for which $T_1,T_2,T_3$ are the tangent components of the left-invariant frame in $G$. 

	Moreover, the only possible candidate for $S$ in the above conditions is
	\begin{equation}\label{thm:fundamental:eqn1}
	SX=\sum_{\alpha=1}^3\epsilon_\alpha\bigl(\mu_\alpha\langle T_\alpha,X\rangle JT_\alpha-\langle\nabla\nu_\alpha,X\rangle T_\alpha\bigr),
	\end{equation}
	which is self-adjoint if and only if $\sum_{\alpha=1}^3\left(\epsilon_\alpha\langle\nabla\nu_\alpha,JT_\alpha\rangle+\epsilon_1\epsilon_2\epsilon_3c_\alpha\nu_\alpha^2\right)=0$.
\end{theorem}

\begin{proof}
Define $M^\alpha_j=\epsilon_\alpha\langle T_\alpha,e_\beta\rangle$ and $M^\alpha_3=\epsilon_\alpha\nu_\alpha$ for $\alpha\in\{1,2,3\}$ and $k\in\{1,2\}$. The condition $\sum_{\alpha=1}^3\epsilon_\alpha\langle T_\alpha,e_i\rangle\langle T_\alpha,e_j\rangle=\epsilon_i\delta^j_i$ means that the first two columns of $M$ are orthogonal with norms $\epsilon_1$ and $\epsilon_2$, respectively, which guarantees that the third column has norm $\epsilon_3$ and $M\in\SO^\epsilon_3(\R)$. This is equivalent to satisfying the algebraic conditions~\ref{eqn:comp-iii}. We also define $\omega^3_j(e_k)=\epsilon_3\langle Se_k,e_j\rangle$ so that $\omega^3_1\wedge\omega^1+\omega^3_2\wedge\omega^2=0$ is clearly equivalent to the fact that $S$ is self-adjoint.

In view of Remark~\ref{rmk:two-columns}, we just need to check that the first two columns of $\df M$ and $M\Omega+ML(M)$ agree, which is equivalent to check that the six identities in~\eqref{eqn:dM} are satisfied. We will show how to deal with the first one since the others are very similar. Since we are assuming~\ref{eqn:comp-iv}, we can expand the differential of $M^1_1$ as
\begin{align*}
\df M^1_1(e_k)&=\epsilon_1 e_k(\langle T_1,e_1\rangle)=\epsilon_1\langle\nabla_{e_k}T_1,e_1\rangle+\epsilon_1\langle T_1,\nabla_{e_k}e_1\rangle\\
&=\epsilon_1\epsilon_3\nu_1\langle Se_1,e_k\rangle+\epsilon_1\epsilon_2\epsilon_3(\mu_3\langle T_2,e_1\rangle\langle T_3,e_k\rangle-\mu_2\langle T_3,e_1\rangle\langle T_2,e_k\rangle)\\
&\qquad+\epsilon_1\epsilon_2\langle T_1,e_2\rangle\langle\nabla_{e_k}e_1,e_2\rangle\\
&=M^1_3\omega^3_1(e_k)+\epsilon_1\mu_3M^2_1M^3_k-\epsilon_1\mu_2M^3_1M^2_k+M^1_2\omega^2_1(e_k),
\end{align*}
and by noticing that $M^\alpha_k=M^\alpha_1\omega^1(e_k)+M^\alpha_2\omega^2(e_k)$, the first identity in~\eqref{eqn:dM} follows readily. Note that $M$ has been uniquely determined by $T_1,T_2,T_3$, so the uniqueness of the immersion follows from Proposition~\ref{prop:second-integration}.

Finally, notice that $S$ must coincide with the shape operator of the immersion (provided that it exists). Moreover, we can employ~\ref{eqn:comp-iv} and~\eqref{eqn:JEi} to get
\begin{equation}\label{thm:fundamental:eqn2}
 \sum_{\alpha=1}^3\epsilon_\alpha\nu_\alpha\nabla_XT_\alpha=\epsilon_3\sum_{\alpha=1}^3\epsilon_\alpha\nu_\alpha^2SX-\sum_{\alpha=1}^3\epsilon_\alpha\mu_\alpha\langle X,T_\alpha\rangle JT_\alpha.
\end{equation}
Observe that $\sum_{\alpha=1}^3\epsilon_\alpha\nu_\alpha\nabla_XT_\alpha=\sum_{\alpha=1}^3\epsilon_\alpha\langle\nabla\nu_\alpha T_\alpha\rangle$ by differentiating the algebraic relation $\sum_{\alpha=1}^3\epsilon_\alpha\nu_\alpha T_\alpha$ in the direction of $X$. Therefore, since $\sum_{\alpha=1}^3\epsilon_\alpha\nu_\alpha^2=\epsilon_3$, we infer~\eqref{thm:fundamental:eqn1} from~\eqref{thm:fundamental:eqn2}. This operator $S$ is self-adjoint if and only if $\langle Se_1,e_2\rangle=\langle Se_2,e_1\rangle$, which can be rewritten as
\begin{equation}\label{thm:fundamental:eqn3}
	\begin{aligned}
\textstyle\sum_{\alpha=1}^3\epsilon_\alpha\mu_\alpha\bigl(\langle T_\alpha,e_1\rangle&\langle JT_\alpha,e_2\rangle-\langle T_\alpha,e_2\rangle\langle JT_\alpha,e_1\rangle\bigr)\\
&=\textstyle\sum_{\alpha=1}^3\epsilon_\alpha\bigl(\langle\nabla\nu_\alpha,e_1\rangle\langle T_\alpha,e_2\rangle-\langle\nabla\nu_\alpha,e_2\rangle\langle T_\alpha,e_1\rangle\bigr).
\end{aligned}
\end{equation}
Since $Je_1=\epsilon_2e_2$ and $Je_2=-\epsilon_1e_1$, Equation~\eqref{thm:fundamental:eqn3} can be expressed as 
\[\sum_{\alpha=1}^3\epsilon_\alpha\mu_\alpha\langle T_\alpha,T_\alpha\rangle=-\sum_{\alpha=1}^3\epsilon_1\epsilon_2\langle\nabla\nu_\alpha,JT_\alpha\rangle,\]
but we also have $\sum_{\alpha=1}^3\epsilon_\alpha\mu_\alpha\langle T_\alpha,T_\alpha\rangle=\epsilon_3\sum_{\alpha=1}^3c_\alpha\nu_\alpha^2$ by just algebraic manipulation, which yields the last equivalence in the statement.
\end{proof}

\subsection{Recovering the immersion from the angle functions}\label{sec:angles}
Our next goal is to study up to what extent the angle functions and the orientation determine the isometric immersion. Given a Riemannian or Lorentzian surface $\Sigma$, we will prescribe $\nu_1,\nu_2,\nu_3\in\mathcal{C}^\infty(\Sigma)$ and the intrinsic rotation $J$. The vector fields
\begin{equation}\label{eqn:Xi}
	\begin{aligned}
X_1 &= J\nabla\nu_1 + \epsilon_2\epsilon_3\left(\nu_3\nabla\nu_2 - \nu_2\nabla\nu_3\right),\\
X_2 &= J\nabla\nu_2 + \epsilon_3\epsilon_1\left(\nu_1\nabla\nu_3 - \nu_3\nabla\nu_1\right),\\
X_3 &= J\nabla\nu_3 + \epsilon_1\epsilon_2\left(\nu_2\nabla\nu_1 - \nu_1\nabla\nu_2\right).
\end{aligned}
\end{equation}
will be the key to characterize the tangent projections. Next lemma shows that they satisfy some rather non-evident properties by just assuming that the Gauss map takes values in the unit sphere.

\begin{lemma}\label{lem:Xi}
	If $\sum_{\alpha=1}^3\epsilon_\alpha\nu_\alpha^2=\epsilon_3$, then the following relations hold true:
	\begin{enumerate}[label=\emph{(\alph*)}]
		\item $(\epsilon_\beta-\epsilon_3\nu_\beta^2)\langle X_\alpha,X_\alpha\rangle =(\epsilon_\alpha-\epsilon_3\nu_\alpha^2)\langle X_\beta,X_\beta\rangle$ for all $\alpha,\beta\in\{1,2,3\}$; in particular, there is a unique $\psi\in\mathcal C^\infty(\Sigma)$ such that $\psi=\frac{\langle X_\alpha,X_\alpha\rangle}{\epsilon_\alpha-\epsilon_3\nu_\alpha^2}$ for all $\alpha\in\{1,2,3\}$ (unless at points where $\epsilon_\alpha-\epsilon_3\nu_\alpha^2=0$, at which $\langle X_\alpha,X_\alpha\rangle$ also vanishes);
		\item $\langle X_\alpha,X_\beta\rangle =(\epsilon_\alpha\delta^\alpha_\beta-\epsilon_3\nu_\alpha\nu_\beta)\psi$ for all $\alpha,\beta\in\{1,2,3\}$;
		\item $\sum_{\alpha=1}^3\!\epsilon_\alpha\langle X_\alpha,X_\alpha\rangle=2\psi$, $\sum_{\alpha=1}^3\!\epsilon_\alpha\langle X_\alpha,\nabla\nu_\alpha\rangle=0$, $\sum_{\alpha=1}^3\!\epsilon_\alpha\langle X_\alpha,J\nabla\nu_\alpha\rangle=\psi$.
	\end{enumerate}
\end{lemma}

\begin{proof}
	By differentiating the condition $\sum_{\alpha=1}^3\epsilon_\alpha\nu_\alpha^2=\epsilon_3$ in the directions of the gradients $\nabla\nu_1,\nabla\nu_2,\nabla\nu_3$, we get the following relations:
	\begin{equation}\label{lem:Xi:eqn1}
		\begin{aligned}
			\nu_1\langle\nabla\nu_1,\nabla\nu_1\rangle&=-\epsilon_1\epsilon_2\nu_2\langle\nabla\nu_1,\nabla\nu_2\rangle-\epsilon_1\epsilon_3\nu_3\langle\nabla\nu_3,\nabla\nu_1\rangle,\\
			\nu_2\langle\nabla\nu_2,\nabla\nu_2\rangle&=-\epsilon_2\epsilon_3\nu_3\langle\nabla\nu_2,\nabla\nu_3\rangle-\epsilon_2\epsilon_1\nu_1\langle\nabla\nu_1,\nabla\nu_2\rangle,\\
			\nu_3\langle\nabla\nu_3,\nabla\nu_3\rangle&=-\epsilon_3\epsilon_1\nu_1\langle\nabla\nu_3,\nabla\nu_1\rangle-\epsilon_3\epsilon_2\nu_2\langle\nabla\nu_2,\nabla\nu_3\rangle.
		\end{aligned}
	\end{equation}
	If we differentiate in the directions of $J\nabla\nu_1,J\nabla\nu_2,J\nabla\nu_3$ instead, we reach
	\begin{equation}\label{lem:Xi:eqn2}
		\begin{aligned}
			\epsilon_2\nu_2\langle J\nabla\nu_1,\nabla\nu_2\rangle+\epsilon_3\nu_3\langle J\nabla\nu_1,\nabla\nu_3\rangle&=0,\\
			\epsilon_1\nu_1\langle J\nabla\nu_2,\nabla\nu_1\rangle+\epsilon_3\nu_3\langle J\nabla\nu_2,\nabla\nu_3\rangle&=0,\\
			\epsilon_1\nu_1\langle J\nabla\nu_3,\nabla\nu_1\rangle+\epsilon_2\nu_2\langle J\nabla\nu_3,\nabla\nu_2\rangle&=0.\\
		\end{aligned}
	\end{equation}
In order to show that $(\epsilon_\beta-\epsilon_3\nu_\beta^2)\langle X_\alpha,X_\alpha\rangle =(\epsilon_\alpha-\epsilon_3\nu_\alpha^2)\langle X_\beta,X_\beta\rangle$, we can assume $\alpha\neq\beta$. Although the computations are cumbersome and will be omitted, the strategy can be described clearly: first, assume $\alpha\neq\beta$ and substitute $X_\alpha$ and $X_\beta$ by their definitions in~\eqref{eqn:Xi}; second, expand out both sides to get combinations of three types of inner products, namely $\langle J\nabla\nu_i,\nabla\nu_j\rangle$, $\langle\nabla\nu_i,\nabla\nu_i\rangle$, and $\langle \nabla\nu_i,\nabla\nu_j\rangle$ (and recall that $J^2=-\epsilon_1\epsilon_2\mathrm{id}$); third and last, we can check that all terms of the form $\langle J\nabla\nu_i,\nabla\nu_j\rangle$ cancel out by means of~\eqref{lem:Xi:eqn2}, whereas all terms of the form $\langle\nabla\nu_i,\nabla\nu_i\rangle$ can be transformed into combinations of terms of the form $\langle \nabla\nu_i,\nabla\nu_j\rangle$ using~\eqref{lem:Xi:eqn1}, which in turn cancel out when grouped with the rest of terms. In this process, it is important to point out that (by collecting and simplifying the coefficients by means of $\sum_{\alpha=1}^3\epsilon_\alpha\nu_\alpha^2=\epsilon_3$) it is not necessary to divide by any of the angles $\nu_i$ in order to apply~\eqref{lem:Xi:eqn1} and~\eqref{lem:Xi:eqn2}. 

Given some $p\in\Sigma$, since $\sum_{\alpha=1}^3\epsilon_\alpha\nu_\alpha^2=\epsilon_3$, there must be some $\gamma\in\{1,2,3\}$ such that $\epsilon_\gamma-\epsilon_3\nu_\gamma^2\neq 0$ at $p$, so we can define $\psi=\frac{\langle X_\gamma,X_\gamma\rangle}{\epsilon_\gamma-\epsilon_3\nu_\gamma^2}$ in a neighborhood of $p$. The formula in item (a) shows that this definition is independent of the choice of $\gamma$.

As for item (b), we can reason likewise by writing the identity we want to prove as $(\epsilon_\gamma-\epsilon_3\nu_\gamma^2)\langle X_\alpha,X_\beta\rangle-(\epsilon_\alpha\delta^\alpha_\beta-\epsilon_3\nu_\alpha\nu_\beta)\langle X_\gamma,X_\gamma\rangle=0$ for some $\gamma\in\{1,2,3\}$ such that $\epsilon_\gamma-\epsilon_3\nu_\gamma^2\neq 0$. The identities~\eqref{lem:Xi:eqn1} and~\eqref{lem:Xi:eqn2} apply as explained for item (a). 

The first identity in item (c) follows by adding those in item (b) with $\alpha=\beta$. The second identity in (c) is a direct consequence of~\eqref{eqn:Xi}, whereas the third identity in (c) can be transformed into $\sum_{\alpha=1}^3\epsilon_\alpha\langle X_\alpha,2J\nabla\nu_\alpha-X_\alpha\rangle=0$ by means of the first one and then apply the same strategy as in item (a).
\end{proof}

We shall assume for a moment that $\Sigma$ is actually immersed in $G$ with angles $\nu_1,\nu_2,\nu_3$ to gain some intuition. From the compatibility equations~\ref{eqn:comp-iii} and~\ref{eqn:comp-v}, along with the fact that $JS+SJ=2H\epsilon_3 J$, it is easy to check that 
\begin{equation}\label{eqn:gaussdata:eqn1}
X_\alpha = - 2 H\epsilon_3 JT_\alpha + \epsilon_1\epsilon_2\epsilon_3\zeta\, T_\alpha,\qquad\text{where }\zeta=\textstyle\sum_{\gamma=1}^3c_\gamma\nu_\gamma^2.
\end{equation}
for all $\alpha\in\{1,2,3\}$. Let us fix $\alpha$ such that $\epsilon_\alpha-\epsilon_3\nu_\alpha^2\neq 0$ (recall that such an $\alpha$ always exists). Since $\langle T_\alpha,T_\alpha\rangle=\epsilon_1\epsilon_2\langle JT_\alpha,JT_\alpha\rangle^2=\epsilon_\alpha-\epsilon_3\nu_\alpha^2\neq 0$ and $\langle T_\alpha,JT_\alpha\rangle=0$, we can work out $\psi=\frac{\langle X_\alpha,X_\alpha\rangle}{\epsilon_\alpha-\epsilon_3\nu_\alpha^2}$ by taking squared norms in~\eqref{eqn:gaussdata:eqn1}, which yields
\begin{equation}\label{eqn:gaussdata:eqn2}
\psi=4H^2\epsilon_1\epsilon_2+\zeta^2.
\end{equation}
This sheds light on item (a) of Lemma~\ref{lem:Xi} since the right-hand side of~\eqref{eqn:gaussdata:eqn2} does not depend on $\alpha$. At points where $\psi\neq 0$, we can solve for $T_\alpha$ in~\eqref{eqn:gaussdata:eqn1} to reach
\begin{equation}\label{eqn:gaussdata:eqn3}
T_\alpha=\frac{\epsilon_1\epsilon_2\epsilon_3\zeta}{\psi}X_\alpha+\frac{2H\epsilon_3}{\psi}JX_\alpha.
\end{equation}
It might seem that~\eqref{eqn:gaussdata:eqn3} uniquely determines $T_\alpha$ only in terms of $\nu_1,\nu_2,\nu_3$ and $J$, because $H^2$ can be retrieved from~\eqref{eqn:gaussdata:eqn2}. Unfortunately, the sign of $H$ is undefined and we will see that different signs potentially lead to non-congruent immersions. However, the existence of such immersions will depend on whether or not~\eqref{eqn:gaussdata:eqn3} (for $H$ or $-H$) satisfies the hypotheses of Theorem~\ref{thm:fundamental} when the shape operator $S$ is necessarily given by~\eqref{thm:fundamental:eqn1}. In the next result, we show that the algebraic relations~\ref{eqn:comp-iii} and the symmetry of $S$ turn out to be automatically satisfied so the only nontrivial restrictions are the compatibility equations~\ref{eqn:comp-iv}.

\begin{theorem}\label{thm:angles}
	Let $G$ be the semi-Riemannian unimodular metric Lie group of diagonalizable type with structure constants $c_1,c_2,c_3$ and signs $\epsilon_1,\epsilon_2,\epsilon_3$. Let $\Sigma$ be a semi-Riemannian simply connected surface oriented by $J$ with signs $\epsilon_1,\epsilon_2$ and let $\nu_1,\nu_2,\nu_3\in\mathcal C^\infty(\Sigma)$ be such that $\sum_{\alpha=1}^3\epsilon_\alpha\nu_\alpha^2=\epsilon_3$. Assume that the function $\psi$ defined by Lemma~\ref{lem:Xi} never vanishes, and let $\zeta=\sum_{\alpha=1}^3c_\alpha\nu_\alpha^2$.
	\begin{enumerate}[label=\emph{(\alph*)}]
		\item There are at most as many isometric immersions of $\Sigma$ into $G$ (up to left-translations) with prescribed angles $\nu_1,\nu_2,\nu_3$ and orientation $J$ as choices of a smooth function $H\in\mathcal C^\infty(\Sigma)$ such that $\psi=4H^2\epsilon_1\epsilon_2+\zeta^2$ (pointwise speaking, there are at most two possible values of $H$). The function $H$ is \emph{a posteriori} the mean curvature of the isometric immersion.

		\item For each such choice of $H\in\mathcal C^\infty(\Sigma)$, the existence of the immersion only depends on whether or not $T_1,T_2,T_3$ and $S$ defined by~\eqref{eqn:gaussdata:eqn3} and~\eqref{thm:fundamental:eqn1}, respectively, satisfy the compatibility equations~\ref{eqn:comp-iv}.

		\item If two such immersions $\phi$ and $\widetilde\phi$ exist (for two choices of $H$), then there is $\theta\in\mathcal C^\infty(\Sigma)$ such that their tangent projections are rotated by an angle of $\theta$, i.e., for all $\alpha\in\{1,2,3\}$ we have
		\[\widetilde T_\alpha=\Rot_\theta(T_\alpha)=\begin{cases}\cos(\theta) T_\alpha+\sin(\theta) JT_\alpha&\text{if }\epsilon_1\epsilon_2=1,\\ \cosh(\theta) T_\alpha+\sinh(\theta) JT_\alpha&\text{if }\epsilon_1\epsilon_2=-1.
		\end{cases}\]

		\item If $H\in\mathcal C^\infty(\Sigma)$ is chosen with sign (equivalently, $\epsilon_1\epsilon_2(\psi-\zeta^2)>0$) and gives rise to an immersion $\phi$, then the immersion $\widetilde\phi$ corresponding to the opposite sign (i.e., $\widetilde H=-H$) exists if and only if
		\begin{equation}\label{eqn:spm}
		\zeta\,\nabla\log|H|=\nabla\zeta+\sum_{\alpha=1}^3\epsilon_\alpha\mu_\alpha\nu_\alpha JX_\alpha,
		\end{equation}
		in which case the angle $\theta$ given in item \emph{(c)} satisfies $\tan(\frac{\theta}{2})=\frac{-2H}{\zeta}$ if $\Sigma$ is spacelike or $\tanh(\frac{\theta}{2})=\frac{-2H}{\zeta}$ if $\Sigma$ is timelike.
	\end{enumerate}
\end{theorem}

\begin{proof}
If the immersion exists, then $H$ must be a smooth function satisfying~\eqref{eqn:gaussdata:eqn2}, and conversely, for each such $H$, the immersion exists if and only if $T_1,T_2,T_3$ and $S$ satisfy Theorem~\ref{thm:fundamental}. On the one hand, from~\eqref{eqn:gaussdata:eqn3} and item (b) of Lemma~\ref{lem:Xi}, we get the algebraic relations $\langle T_\alpha,T_\beta\rangle=\psi^{-1}\langle X_\alpha,X_\beta\rangle=\epsilon_\alpha\delta^\alpha_\beta-\epsilon_3\nu_\alpha\nu_\beta$. Therefore, the three rows of the matrix $M$ in~\eqref{eqn:M} form an orthonormal frame, and hence the columns also form an orthonormal frame, that is, $\sum_{\alpha=1}^3\epsilon_\alpha\langle T_\alpha,e_i\rangle\langle T_\alpha,e_j\rangle=\epsilon_i\delta^j_i$ for all $i,j\in\{1,2\}$ on a positive orthonormal frame $\{e_1,e_2\}$ with signature $\epsilon_1,\epsilon_2$. On the other hand, the condition for the symmetry of $S$ in Theorem~\ref{thm:fundamental} can be written as $\sum_{\alpha=1}^3\epsilon_\alpha\langle\nabla\nu_\alpha,JT_\alpha\rangle=-\epsilon_1\epsilon_2\epsilon_3\zeta$, which by~\eqref{eqn:gaussdata:eqn3} can be in turn transformed into $\frac{\zeta}{\psi}\sum_{\alpha=1}^3\epsilon_\alpha\langle J\nabla\nu_\alpha,X_\alpha\rangle+\frac{2H}{\psi}\sum_{\alpha=1}^3\epsilon_\alpha\langle\nabla\nu_\alpha,X_\alpha\rangle=
\zeta$, and this last identity follows readily from item (c) of Lemma~\ref{lem:Xi}. All in all, Theorem~\ref{thm:fundamental} says that the immersion exists if and only if~\ref{eqn:comp-iv} are fulfilled, and uniqueness is guaranteed up to left-translations, that is, items (a) and (b) in the statement are proved, except for the last assertion in item (a). 

We will see next that the compatibility equations~\ref{eqn:comp-v} are also satisfied automatically (this simplifies the forthcoming arguments). Observe that~\ref{eqn:comp-v} holds true if and only if it holds true when multiplied by any $T_\alpha$. As for the first equation (the other two can be proved similarly), this means we have to show that
\begin{align*}
\langle\nabla\nu_1,T_\alpha\rangle&=-\sum_{\beta=1}^3\epsilon_\beta(\mu_\beta\langle T_\beta,T_1\rangle\langle JT_\beta,T_\alpha\rangle-\langle\nabla\nu_\beta,T_1\rangle\langle T_\beta,T_\alpha\rangle)\\
&\qquad +\epsilon_2\epsilon_3(\mu_3\nu_2\langle T_3,T_\alpha\rangle-\mu_2\nu_3\langle T_2,T_\alpha\rangle)
\end{align*}
for all $\alpha\in\{1,2,3\}$. As in the proof of Lemma~\ref{lem:Xi}, this computation is burdensome but again it follows from the algebraic relations for $T_\alpha$ we have already proved (it is very useful to check first that~\eqref{eqn:JEi} holds just by these algebraic relations) plus the spherical constraints given by~\eqref{lem:Xi:eqn1} and~\eqref{lem:Xi:eqn2}. Knowing that~\ref{eqn:comp-v} is true, we can go back to the last sentence in item (a), and compute
\begin{align*}
\trace(S)&\stackrel{(\star)}{=}\sum_{\alpha=1}^3\epsilon_\alpha\langle ST_\alpha,T_\alpha\rangle=-\sum_{\alpha=1}^3\epsilon_\alpha\langle\nabla\nu_\alpha,T_\alpha\rangle\\
&=-\tfrac{\epsilon_1\epsilon_2\epsilon_3\zeta}{\psi}\sum_{\alpha=1}^3\epsilon_\alpha\langle\nabla\nu_\alpha,X_\alpha\rangle+\tfrac{2H\epsilon}{\psi}\sum_{\alpha=1}^3\epsilon_\alpha\langle J\nabla\nu_\alpha,X_\alpha\rangle=2H\epsilon_3,
\end{align*}
where we have used~\eqref{eqn:gaussdata:eqn3} and item (c) of Lemma~\ref{lem:Xi}. We would like to remark that the equality $(\star)$ is true for any operator and follows from elementary linear algebra along with the algebraic relations.

As for item (c), again the algebraic relations imply that $\langle \widetilde T_\alpha,\widetilde T_\beta\rangle=\langle T_\alpha, T_\beta\rangle$ for all $\alpha,\beta\in\{1,2,3\}$ (both immersions share the angle functions). Since the tangent projections span the tangent plane, there must be an isometry $R$ at each tangent plane of $\Sigma$ such that $\widetilde T_\alpha=R(T_\alpha)$. Equation~\eqref{eqn:gaussdata:eqn3} implies that either $R$ is the identity or a proper rotation at each point depending on whether the mean curvatures agree or they are opposite. Pointwise speaking, this means that $\widetilde{T}_\alpha=\cos(\theta)T_\alpha + \sin(\theta)JT_\alpha$ if $\Sigma$ is spacelike or $\widetilde{T}_\alpha=\cosh(\theta)T_\alpha + \sinh(\theta)JT_\alpha$ if $\Sigma$ is timelike, with $\theta=0$ if $R$ is the identity. We deduce that $\theta$ is smooth because in a neighborhood of any point for we can always choose the index $\alpha$ such that $T_\alpha\neq 0$. Observe that, as $\Sigma$ is assumed simply connected, $\theta$ can be actually defined globally (otherwise, we should consider $\theta$ modulo $2\pi$).

Finally, we will assume that $H$ is positive in order to prove item (d). Let $T_1,T_2,T_3,S$ be the data of the associated immersion. The first observation is that the equation for $\nabla_XT_\alpha$ in~\ref{eqn:comp-iv}, when multiplied by $T_\alpha$ is automatically fulfilled. For instance, when $\alpha=1$, we get from~\ref{eqn:comp-v} and the symmetry of $S$ that
\begin{align*}
	\langle\nabla_XT_1,T_1\rangle&=\tfrac{1}{2}X\langle T_1,T_1\rangle=\tfrac{1}{2}X(\epsilon_1-\epsilon_3\nu_1^2)=-\epsilon_3\nu_1\langle \nabla\nu_1,X\rangle\\
	&=\epsilon_3\nu_1\langle ST_1,X\rangle-\epsilon_2\nu_1(\mu_3\nu_2\langle T_3,X\rangle-\mu_2\nu_3\langle T_2,X\rangle),
\end{align*}
which agrees with~\ref{eqn:comp-iv} innerly multiplied by $T_1$. Therefore, the most interesting information in $\nabla_X T_\alpha$ is its component in the direction $JT_\alpha$. Furthermore, the other immersion corresponding to the opposite sign $\widetilde H=-H$ will exist if and only if after expanding both sides of~\ref{eqn:comp-iv} multiplied by $JT_\alpha$, we end up with an even function of $H$. For simplicity, write $T_\alpha=aX_\alpha+bJX_\alpha$ with $a=\frac{\epsilon_1\epsilon_2\epsilon_3\zeta}{\psi}$ and $b=\frac{2H\epsilon_3}{\psi}$ as in~\eqref{eqn:gaussdata:eqn3}, so being even in $H$ amounts to being even in $b$. On the one hand, the left-hand side of~\ref{eqn:comp-iv} gives
\begin{align}
	\langle \nabla_X T_\alpha,JT_\alpha\rangle &=\langle \nabla_X (aX_\alpha+bJX_\alpha),aJX_\alpha-\epsilon_1\epsilon_2bX_\alpha\rangle\notag\\
	&=E(b)-\epsilon_1\epsilon_2 b(X(a)\langle X_\alpha,X_\alpha\rangle+a\langle\nabla_XX_\alpha,X_\alpha\rangle)\notag\\
	&\qquad +a(X(b)\langle JX_\alpha,JX_\alpha\rangle+b\langle\nabla_XJX_\alpha,JX_\alpha\rangle)\notag\\
	&=E(b)+\epsilon_1\epsilon_2(aX(b)-bX(a))\langle X_\alpha,X_\alpha\rangle\notag\\
	&=E(b)+\epsilon_3b(\epsilon_\alpha-\epsilon_3\nu_\alpha^2)\langle X,\zeta\nabla\log|H|-\log(\zeta)\rangle,\label{thm:angles:eqn1}
\end{align}
where $E(b)$ collects all terms which are even in $b$ (notice that the gradient of the logarithm ignores the sign of $H$). On the other hand, the right-hand side of~\ref{eqn:comp-iv} multiplied by $JT_\alpha$ can be expanded by means of~\ref{eqn:comp-iii} and~\ref{eqn:comp-v} plus the symmetry of $S$, and then simplified by means of the spherical constraints~\eqref{lem:Xi:eqn1} and~\eqref{lem:Xi:eqn2} as in the proof of Lemma~\ref{lem:Xi}. After onerous but simple computations, we reach
\begin{equation}\label{thm:angles:eqn2}
	\langle \nabla_X T_\alpha,JT_\alpha\rangle =\widehat{E}(b)+\epsilon_3b(\epsilon_\alpha-\epsilon_3\nu_\alpha^2)\bigl\langle X,\textstyle\sum_{\alpha=1}^3\epsilon_\alpha\mu_\alpha\nu_\alpha JX_\alpha\bigr\rangle,
\end{equation}
where again $\widehat{E}(b)$ contains all even terms. By comparing~\eqref{thm:angles:eqn1} and~\eqref{thm:angles:eqn2} for some $\alpha$ with $\epsilon_\alpha-\epsilon_3\nu_\alpha^2\neq 0$, the condition in the statement follows.
\end{proof}

\begin{remark}~\label{rmk:kenmotsu}
 In the Riemannian setting, the existence and local uniqueness of at most two isometric immersions with opposite mean curvature functions (around points with $\psi\neq 0$) can be actually deduced from the work of Gálvez--Mira~\cite[Thm.~3.1]{GM16}, where the authors give a more general statement for conformal immersions in both the unimodular and non-unimodular cases. In our Theorem~\ref{thm:angles}, item (a), we prescribe the metric of $\Sigma$, which amounts to prescribing the conformal factor in Gálvez and Mira's result, from where one finds the value of $H^2$, cf.~\cite[Eqns.~(2.10) and~(3.15)]{GM16}. Our approach to existence is a bit more involved, but it also sheds some light into the role of the tangent projections, and relies only on a first-order \textsc{ode} system. The rest of items in Theorem~\ref{thm:angles} are not discussed in~\cite{GM16}.
\end{remark}

Isometric immersions that share the angle functions will be called \emph{angular companions} in what follows. It seems tough in general to find explicit companions since, except for some particular cases that will be explored throughout the next examples, Equation~\eqref{eqn:spm} admits no simple resolution. In these examples, we will consider only the Riemannian case $\epsilon_1=\epsilon_2=\epsilon_3=1$, because it captures the main ideas, though some of them can be extended to the rest of signatures. In fact, we will begin by considering the case of Euclidean space $\R^3$, in which will also analyze the case $\psi=0$ not considered in Theorem~\ref{thm:angles}.

\begin{example}[Antipodal immersions]\label{ex:antipodal}
Let $\mathbb{R}^3$ endowed with the Euclidean metric, which satisfies $\mu_1=\mu_2=\mu_3=0$ and $\zeta=0$, so that~\eqref{eqn:spm} holds. We will see that all surfaces have angular companions even in the minimal case (in which $\psi\equiv0$ and Theorem~\ref{thm:angles} does not apply). Let $\phi:\Sigma\to\R^3$ be \emph{any} immersion and let $N$ be the unit normal. The antipodal map $\tau(x,y,z)=(-x,-y,-z)$ is an isometry, whence $\tau\circ\phi:\Sigma\to\R^3$ is also an isometric immersion with the same mean curvature but opposite angles when $-N=\tau_*N$ is chosen as unit normal. However, $\tau$ reverses the orientation, so the proper choice of normal for $\tau\circ\phi$ compatible with the orientation of $\Sigma$ is the opposite (i.e., $N=-\tau_*N$), which gives the same angle functions but opposite mean curvature as $\phi$ (note also that this antipodal immersion $\tau\circ\phi$ changes the signs of $T_1,T_2,T_3$). We deduce that $\phi$ and $\tau\circ\phi$ are angular companions.
\end{example}

\begin{example}[Bonnet associate family]\label{ex:bonnet} 
The case $\psi\equiv 0$ in $\R^3$ can be elaborated further. A minimal isometric immersion $\phi:\Sigma\to\R^3$ has Weierstrass data $(\mathcal{G},\eta)$, where $\mathcal G=\varphi(\nu_1,\nu_2,\nu_3):U\to\C$ is the stereographic projection of the Gauss map and $\eta$ is the height differential. If another minimal isometric immersion $\widetilde\phi:\Sigma\to\R^3$ has the same Gauss map, then its data are $(\mathcal G,\widetilde\eta)$ with possibly a different $\widetilde\eta$. In an isothermal coodinate $z$, we can express $\eta=f(z)\df z$ and $\widetilde\eta=\widetilde f(z)\df z$, and the intrinsic curvature of $\Sigma$ is well known to be $K=-16|\mathcal G'|^2/(|f|^2(1+|\mathcal G|^2)^2)=-16|\mathcal G'|^2/(|\widetilde f|^2(1+|\mathcal G|^2)^2)$, from where we deduce that $|\widetilde f|=|f|$ at non-flat points. Since $f$ and $\widetilde f$ are holomorphic, there is $\theta\in\R$ such that $\widetilde f=e^{i\theta}f$, and hence $\phi$ and $\widetilde\phi$ are in the same Bonnet family. Although this argument is local, it becomes global by the analiticity of minimal surfaces. Notice that, if all points happen to be flat, then $\phi$ and $\widetilde\phi$ are planes (with constant angle functions).
\end{example}

\begin{example}[Vertical cylinders]\label{ex:cylinders}
In $\Nil$ (with $\mu_1=\mu_2=\tau$ and $\mu_3=-\tau$), the condition $\zeta=0$ reduces to $\nu_3=0$, which leads to surfaces tangent to the unit Killing vector field $E_3$. This easily implies that the sum in~\eqref{eqn:spm} is also zero and hence all these \emph{vertical cylinders} have angular companions. Notice that such a cylinder is flat and can be parametrized in the Cartan model as 
\[(u,v)\in I\times\R\mapsto \left(x(u),y(u),v+\tau z(u)\right),\]
where $t\mapsto (x(t),y(t))$ is a unit speed curve defined in an interval $I\subseteq\R$ containing the origin, and $z(t)=\int_0^t(x(s)y'(s)-y(s)x'(s)))\,\df s$. This parametrization is chosen so that curves with constant $u$ run in the direction of $E_3$ and curves with constant $v$ are orthogonal to $E_3$, whence the induced metric is $\df u^2+\df v^2$. It is not difficult to check that the angular companion is $(u,v)\mapsto \left(-x(u),-y(u),-v+\tau z(u)\right)$.

More generally, let $\phi:\Sigma\to\E(\kappa,\tau)$ be an isometric immersion, i.e., $c_1=c_2=\frac{\kappa}{2\tau}$ and $c_3=2\tau\neq 0$. Since $\zeta=\tfrac{\kappa-(\kappa-4\tau^2)\nu_3^2}{2\tau}$, the condition~\eqref{eqn:spm} reads 
\begin{equation}\label{eqn:spm:ekt}
\begin{aligned}
		\zeta\,\nabla\log|H|&=\nabla\zeta+\tfrac{\kappa-4\tau^2}{2\tau}\nu_3 J\left(J\nabla\nu_3+\nu_2\nabla\nu_1-\nu_1\nabla\nu_2\right)\\
&=\tfrac{\kappa-4\tau^2}{2\tau}\nu_3\left(2HT_3+\zeta JT_3-2\nabla\nu_3\right),
\end{aligned}
\end{equation}
where the last equality follows by expanding the gradients by means of the compatibility equations~\ref{eqn:comp-v} and using the identity $JS+SJ=2HJ$. Equation~\eqref{eqn:spm:ekt} does not seem to have an easy resolution, even for \textsc{cmc} surfaces. It would be interesting to investigate if this condition is related to helicoidal Bonnet mates~\cite{GMM08}.

As a matter of fact, if we further restrict to vertical cylinders ($\nu_3=0$) and we assume $\kappa\neq 0$ (i.e., the space is not $\Nil$), we get that $\zeta=\frac{\kappa}{2\tau}\neq 0$, and hence~\eqref{eqn:spm:ekt} reduces to $\nabla H=0$. Therefore, the surface projects to a curve of constant curvature $2H$ in $\mathbb{M}^2(\kappa)$. If $4H^2+\kappa>0$, this can be seen as an Euclidean cylinder
\[(u,v)\in\R^2\mapsto (r\cos(u),r\sin(u),\tfrac{4r}{4+\kappa r^2}(v+\tau r u)).\]
in the Cartan model. Again, $u$ and $v$ are orthogonal coordinates with $v$ running in the direction of $E_3$, and the induced metric reads $\frac{16r^2}{(4+\kappa r^2)^2}(\df u^2+\df v^2)$. The radius $r$ satisfies $4+\kappa r^2>0$ (so the surface actually lies in the Cartan model) and the mean curvature is given by $H=\frac{-4+\kappa r^2}{8r}$. The angular companion of this cylinder is nothing but a reparametrization of the same cylinder given by
\[(u,v)\in\R^2\mapsto (-r\cos(\overline{u}),r\sin(\overline{u}),\tfrac{4r}{4+\kappa r^2}(\overline{v}-\tau r \overline{u})),\]
where $(\overline{u},\overline{v})$ is a rotation of the parameters $(u,v)$ given by
\[\overline{u}=\tfrac{\kappa^2-16H^2\tau^2}{\kappa^2+16H^2\tau^2}u-\tfrac{16H\kappa\tau}{\kappa^2+16H^2\tau^2}v,\qquad \overline{v}=\tfrac{16H\kappa\tau}{\kappa^2+16H^2\tau^2}u+\tfrac{\kappa^2-16H^2\tau^2}{\kappa^2+16H^2\tau^2}v.\]
Interestingly, this rotation is an intrinsic isometry of the cylinder which is not extrinsic. A similar reasoning works if $4H^2+\kappa\leq0$, $\kappa<0$ and $H\neq 0$ (in which case the cylinder is over a horocycle or an equidistant curve to a geodesic of $\mathbb{H}^2(\kappa)$).
\end{example}

\begin{example}[Twin $H$-immersions in $\mathbb{S}^3$]\label{ex:twin}
In the sphere $\mathbb{S}^3$, with constant curvature $1$, we have $\kappa=4$ and $\tau=1$, whence $c_1=c_2=c_3=2$ and $\zeta=2$ is constant. Equation~\eqref{eqn:spm:ekt} reduces to $\nabla H=0$ (with $H\neq 0$), so all non-zero \textsc{cmc} immersions into $\mathbb{S}^3$ do have angular companions.

Assume that $T_1,T_2,T_3,S$ and $\widetilde T_1,\widetilde T_2,\widetilde T_3,\widetilde S$ are the remaining data for these companions of \textsc{cmc} $H$ and $\widetilde H=-H$, respectively. From~\eqref{eqn:gaussdata:eqn3} we get $\widetilde T_\alpha=\Rot_\theta(T_\alpha)$ with $\theta=-2\arctan(H)$, whereas~\eqref{thm:fundamental:eqn1} yields
\begin{equation}\label{ex:twin:eqn1}
SX=\sum_{\alpha=1}^3\bigl(\langle T_\alpha,X\rangle JT_\alpha-\langle\nabla\nu_\alpha,X\rangle T_\alpha\bigr)=JX-\sum_{\alpha=1}^3\langle\nabla\nu_\alpha,X\rangle T_\alpha.
\end{equation}
The same computation for the angular companion implies that
\begin{equation}\label{ex:twin:eqn2}
\widetilde SX=JX-\sum_{\alpha=1}^3\langle\nabla\nu_\alpha,X\rangle\widetilde T_\alpha=JX-\operatorname{Rot}_\theta\left(\sum_{\alpha=1}^3\langle\nabla\nu_\alpha,X\rangle T_\alpha\right),
\end{equation}
and combining~\eqref{ex:twin:eqn1} and~\eqref{ex:twin:eqn2} we conclude that $\widetilde S-J={\Rot_\theta}\circ{(S-J)}$. Taking into account that $\Rot_\theta=\frac{1-H^2}{1+H^2}\mathrm{id}-\frac{2H}{1+H^2}J$, this last relation can be rewritten as $\widetilde S-\widetilde H\,\mathrm{id}={\Rot_\theta}\circ{(S- H\,\mathrm{id})}$, that is, the traceless shape operator rotates a constant angle. We deduce that the angular companion of a \textsc{cmc} surface in $\mathbb{S}^3$ belongs to Lawson's associate family, see~\cite[Thm.~8]{Lawson70}. Indeed, these angular companions are \emph{twin immersions} in the sense of Daniel, see~\cite[Thm.~5.14]{Daniel07}. Note that our argument in particular shows that twin immersions have the same angle functions (it is not so well known that Lawson correspondence when acting on surfaces in $\mathbb{S}^3$ and $\mathbb{R}^3$ preserves the angle functions up to ambient isometries, see~\cite{GB93}). It is also important to point out that twin $H$-surfaces in $\mathbb{S}^3$ are not related by ambient isometries (the above computations yield a nontrivial rotation of the traceless shape operators). Explicit examples of non globally congruent twin $H$-surfaces are the horizontal $H$-tubes, see~\cite[Thm.~5.5]{Man24}.

Note also that twin $H$-immersions in $\E(\kappa,\tau)$ are not angular companions in general, because all \textsc{cmc} surfaces have twin immersions but not all of them satisfy~\eqref{eqn:spm:ekt}. For instance, the parabolic helicoids $P_{\kappa,\tau,H}$ with $\kappa<0$, $\tau\neq 0$ and $4H^2+\kappa<0$ (as described in~\cite[Eqn.~(2.7)]{DomMan}) have constant angle $\nu_3^2=\frac{4H^2+\kappa}{\kappa-4\tau^2}$, whence $\zeta=\frac{-2H^2}{\tau}$ is a nonzero constant. Since $T_3$ never vanishes in $P_{\kappa,\tau,H}$, we conclude that parabolic helicoids do not satisfy~\eqref{eqn:spm:ekt}.
\end{example}

We can use this discussion to obtain a complete classification of isometric immersions $\phi,\widetilde\phi$ of a Riemannian surface $\Sigma$ into $\R^3$ or into $\mathbb{S}^3$ whose left-invariant Gauss maps (as maps $g,\widetilde{g}:\Sigma\to\mathbb{S}^2$) differ by an orientation-preserving isometry. This problem was tackled by Abe--Erbacher~\cite{AbeErb75} in $\mathbb{R}^3$, see item (D) of their main theorem. We include in the statement the case of minimal immersions (not explicitly mentioned in~\cite{AbeErb75}) and provide a proof not using elliptic operators. More importantly, we extend their result to the round sphere $\mathbb{S}^3$.

\begin{corollary}\label{coro:uniqueness-angles}
Let $\phi,\widetilde\phi:\Sigma\to G$ be two isometric immersions of an oriented Riemannian surface $\Sigma$ into $G$. Assume that their left-invariant Gauss maps $g,\widetilde g:\Sigma\to\mathbb{S}^2$ satisfy $\widetilde g=\sigma\circ g$ for some orientation preserving $\sigma\in\Iso_+(\mathbb{S}^2)$.
\begin{enumerate}[label=\emph{(\alph*)}]
	\item If $G=\mathbb{R}^3$, then $\phi$ and $\widetilde\phi$ are either congruent by any ambient isometry or they are minimal immersions lying in the same Bonnet associate family.
	\item If $G=\mathbb{S}^3$, then $\phi$ and $\widetilde\phi$ are either congruent by an orientation-preserving ambient isometry or they are twin $H$-immersions with opposite \textsc{cmc}.
\end{enumerate}
\end{corollary}

\begin{proof}
Observe that $\sigma$ can be thought of as an element of $\SO(3)$ by viewing $\mathbb{S}^2\subset\R^3$, so it acts in the Lie algebra $\mathfrak{g}$ when we identify the usual basis of $\R^3$ with the left invariant frame $\{E_1,E_2,E_3\}$ of $G$. This yields an explicit isomorphism between $\Iso_+(\mathbb{S}^2)$ and the orientation-preserving subgroup of $\Stab_p$ for any $p\in G$ (recall that isometries of $G$ preserve left-invariant vector fields). This means that we can change $\phi$ or $\widetilde\phi$ by a orientation-preserving isometry of $G$ to assume that $\widetilde g=g$, i.e., both immersions can be assumed to share their angle functions $\nu_1,\nu_2,\nu_3$. 

Item (c) of Theorem~\ref{thm:angles} gives a function $\theta\in\mathcal C^\infty(\Sigma)\pmod{2\pi}$ such that $\widetilde T_\alpha=\Rot_\theta(T_\alpha)$. To be precise, this result applies to points with $\psi\neq 0$ but extends by continuity to points with $\psi=0$ in $\R^3$ (notice that, if the set of such points have non-empty interior, then $\phi$ and $\widetilde{\phi}$ are associate in this interior by Example~\ref{ex:bonnet}, where $\theta$ must be constant). Indeed, we claim that $\theta$ is locally constant (and hence constant) in $\Sigma$. For, assume that $\theta(p)\not\in\{0,\pi\}\pmod{2\pi}$ for some $p\in\Sigma$, which implies by~\eqref{eqn:gaussdata:eqn3} that either $\psi,H\neq 0$ or $\psi\equiv0$ in a neighborhood $U$ of $p$. This last case follows by contradiction: if $\psi(p)=0$ and there is a sequence $\{p_n\}\to p$ with $\psi(p_n)=4H(p_n)^2\neq 0$, then $\theta(p_n)=\pi\pmod{2\pi}$ by~\eqref{eqn:gaussdata:eqn3} (notice that $\zeta=0$ because this situation can only happen if $G=\R^3$), and hence $\theta(p)=\pi\pmod{2\pi}$ in the limit, which is absurd. All in all, if $\theta(p)\not\in\{0,\pi\}\pmod{2\pi}$, the restrictions $\phi|_U$ and $\widetilde\phi|_U$ are either associate immersions in $\R^3$ or twin immersions in $\mathbb{S}^3$, and either way $\theta$ is constant in $U$ so the claim is true. 

The same argument also applies globally: if $\theta\not\in\{0,\pi\}\pmod{2\pi}$, then either $\psi$ has no zeros (and we get twin $H$-immersions in $\mathbb{S}^3$) or $\psi$ identically vanishes (and we get associate immersions in $\R^3$). If $\theta=0\pmod{2\pi}$, then $\phi$ and $\widetilde\phi$ coincide up to a left-translation by Theorem~\ref{thm:fundamental}. Finally, if $\theta=\pi\pmod{2\pi}$ then $\zeta\equiv 0$ again by~\eqref{eqn:gaussdata:eqn3} and we can recover the immersion from the antipodal map as in Example~\ref{ex:antipodal} (only in this case, which only occurs if $G=\mathbb{R}^3$, the ambient isometry that relates $\phi$ and $\widetilde\phi$ is orientation-reversing).
\end{proof}

Note that a similar argument also works if $\Sigma$ is spacelike into a Minkowski space $\mathbb{L}^3$ or the universal cover of the anti-de Sitter space $\mathbb{H}^3_1$, where there is also a notion of associate family and a Lawson-type correspondence given by Palmer~\cite{Palmer90}.

\subsection{Constant angles and totally geodesic surfaces}\label{sec:Hzeta0}

In our previous discussion, the case $\psi=0$ seems to be problematic since it means that either $X_1=X_2=X_3=0$ (if $H=\zeta=0$) or $X_1,X_2,X_3$ are all lightlike (if $4H^2=\zeta^2$ in the Lorentzian timelike case $\epsilon_1\epsilon_2=-1$). In this section, we will focus only on the first case $H=\zeta=0$, which hides a geometrically appealing family of examples.

\begin{proposition}\label{prop:Hzeta0}
Let $\Sigma$ be a Riemannian or Lorentzian surface isometrically immersed in a unimodular metric Lie group of diagonalizable type $G$ with not constant sectional curvature (see also Remark~\ref{rmk:Hzeta0}). The following are equivalent:
\begin{enumerate}[label=\emph{(\roman*)}]
	\item The immersion has zero mean curvature and satisfies $\zeta=0$.
	\item The immersion has constant angle functions.
	\item The immersion lies in a left coset of a $2$-dimensional Lie subgroup of $G$.
\end{enumerate}
Moreover, given $\nu_1,\nu_2,\nu_3\in\R$ such that $\sum_{\alpha=1}^3\epsilon_\alpha\nu_\alpha^2=\epsilon_3$ and $\sum_{\alpha=1}^3c_\alpha\nu_\alpha^2=0$, there is a complete surface (unique up to left-translation) with constant angles $\nu_1,\nu_2,\nu_3$.
\end{proposition}

\begin{proof}
Assume that item (i) holds true. By taking gradients in $\sum_{\alpha=1}^3\epsilon_\alpha\nu_\alpha^2=\epsilon_3$ and $\sum_{\alpha=1}^3c_\alpha\nu_\alpha^2=0$, we get to $\sum_{\alpha=1}^3\epsilon_\alpha\nu_\alpha\nabla\nu_\alpha=0$ and $\sum_{\alpha=1}^3c_\alpha\nu_\alpha\nabla\nu_\alpha=0$, respectively. The latter yields $\sum_{\alpha=1}^3c_\alpha\nu_\alpha J\nabla\nu_\alpha=0$ by just applying $J$. Since we are assuming $H=\zeta=0$, it follows from~\eqref{eqn:gaussdata:eqn1} that $X_\alpha=0$, whence~\eqref{eqn:Xi} gives
	\begin{equation}\label{prop:Hzeta0:eqn1}
	\begin{aligned}
		J\nabla\nu_1&=\epsilon_2\epsilon_3(\nu_2\nabla\nu_3-\nu_3\nabla\nu_2),\\
		J\nabla\nu_2&=\epsilon_3\epsilon_1(\nu_3\nabla\nu_1-\nu_1\nabla\nu_3),\\
		J\nabla\nu_3&=\epsilon_1\epsilon_2(\nu_1\nabla\nu_2-\nu_2\nabla\nu_1).
	\end{aligned}
	\end{equation}
 Expanding $\sum_{\alpha=1}^3c_\alpha\nu_\alpha J\nabla\nu_\alpha=0$ by means of~\eqref{prop:Hzeta0:eqn1}, we reach the linear system
	\begin{equation}\label{prop:Hzeta0:eqn2}
	\left.\begin{array}{r}
		\epsilon_1\nu_1\nabla\nu_1+\epsilon_2\nu_2\nabla\nu_2+\epsilon_3\nu_3\nabla\nu_3=0\\
		c_1\nu_1\nabla\nu_1+c_2\nu_2\nabla\nu_2+c_3\nu_3\nabla\nu_3=0\\
		(\epsilon_2c_2-\epsilon_3c_3)\nu_2\nu_3\nabla\nu_1+(\epsilon_3c_3-\epsilon_1c_1)\nu_3\nu_1\nabla\nu_2+(\epsilon_1c_1-\epsilon_2c_2)\nu_1\nu_2\nabla\nu_3=0
	\end{array}\right\}.
	\end{equation}
	Assume not all angle functions are constant, in which case we will work in an open subset of $\Sigma$ in which some of the gradients $\nabla\nu_1,\nabla\nu_2,\nabla\nu_3$ never vanish. This implies that the coefficient matrix of~\eqref{prop:Hzeta0:eqn2} is singular, which amounts to
	\begin{align}
	0&=-\epsilon_1\epsilon_2\epsilon_3\det\left(\begin{matrix}
		\epsilon_1\nu_1&\epsilon_2\nu_2&\epsilon_3\nu_3\\
		c_1\nu_1&c_2\nu_2&c_3\nu_3\\
		(\epsilon_2c_2-\epsilon_3c_3)\nu_2\nu_3&(\epsilon_3c_3-\epsilon_1c_1)\nu_3\nu_1&(\epsilon_1c_1-\epsilon_2c_2)\nu_1\nu_2
	\end{matrix}\right)\notag\\
	&=\epsilon_1(\epsilon_2c_2-\epsilon_3c_3)^2\nu_2^2\nu_3^2+\epsilon_2(\epsilon_3c_3-\epsilon_1c_1)^2\nu_3^2\nu_1^2+\epsilon_3(\epsilon_1c_1-\epsilon_2c_2)^2\nu_1^2\nu_2^2.\label{prop:Hzeta0:eqn3}
	\end{align}
	We will distinguish cases according to how many distinct $\epsilon_1c_1,\epsilon_2c_2,\epsilon_3c_3$ we have. 
	\begin{enumerate}
		\item Assume first $\epsilon_1c_1,\epsilon_2c_2,\epsilon_3c_3$ are all distinct. We can think of~\eqref{prop:Hzeta0:eqn3} as a quadric $Q\subset\R^3$ in the unknowns $x=\nu_1^2,y=\nu_2^2,z=\nu_3^2$, and we can think of the intersection of $\sum_{\alpha=1}^3\epsilon_\alpha\nu_\alpha^2=\epsilon_3$ and $\sum_{\alpha=1}^3c_\alpha\nu_\alpha^2=0$ as a line $L\subset\R^3$. Since $Q$ is the union of lines through the origin and $L$ does not contain the origin, either $L\cap Q$ consists of finitely many points or there is an affine plane through the origin contained in $Q$, i.e., the quadratic polyonomial in~\eqref{prop:Hzeta0:eqn3} factors as a product of linear polynomials. This is impossible since all coefficients of $xy,yz,zx$ are nonzero and there are no square terms $x^2,y^2,z^2$. Thus, $L\cap Q$ is finite and all angle functions are constant.

		\item Assume now $\epsilon_1c_1=\epsilon_2c_2\neq\epsilon_3c_3$ up to relabeling indexes, so $\sum_{\alpha=1}^3c_\alpha\nu_\alpha^2=0$ reads $\epsilon_1\epsilon_3c_1-\epsilon_3(\epsilon_1c_1-\epsilon_3c_3)\nu_3^2$. Since $\epsilon_1c_1\neq\epsilon_3c_3$, we have that $\nu_3$ is constant. If $\epsilon_1c_1=\epsilon_2c_2=0$, then $\nu_3=0$ and~\eqref{prop:Hzeta0:eqn1} gives $\nabla\nu_1=\nabla\nu_2=0$, whence the angles are all constant. Otherwise, we have $\epsilon_1c_1=\epsilon_2c_2\neq 0$ so that $\nu_3\neq 0$ and~\eqref{prop:Hzeta0:eqn3} now reduces to $\epsilon_1\nu_1^2+\epsilon_2\nu_2^2=0$, and hence $\nu_3^2=1$. Therefore, the above equation $\epsilon\epsilon_1c_1-\epsilon_3(\epsilon_1c_1-\epsilon_3c_3)\nu_3^2$ yields $c_3=0$, and we get a space of constant curvature according to Proposition~\ref{prop:dim-iso-4-6}.

		\item If $\epsilon_1c_1=\epsilon_2c_2=\epsilon_3c_3$, then $0=\sum_{\alpha=1}^3c_\alpha\nu_\alpha^2=\epsilon_1c_1\sum_{\alpha=1}^3\epsilon_\alpha\nu_\alpha^2=\epsilon_1\epsilon_3 c_1$, that is, $c_1=c_2=c_3=0$ and we get a zero mean curvature immersion in $\R^3$ or $\mathbb{L}^3$.
	\end{enumerate}
	All in all, we have shown that the immersion has constant angle functions unless possibly if the space has constant sectional curvature. 

	Conversely, suppose that the immersion has constant angles $\nu_1,\nu_2,\nu_3\in\mathbb{R}$, so the normal reads $N=\sum_{\alpha=1}^3\epsilon_\alpha\nu_\alpha E_\alpha$. Assume first that $\nu_3\neq 0$, we can consider $Y_1=\nu_3E_1-\nu_1E_3$ and $Y_2=\nu_3 E_2-\nu_2E_3$, which form a tangent frame in $\Sigma$ and satisfy $\langle[Y_1,Y_2],N\rangle=\nu_3\zeta$. As $\nu_3\neq 0$, we deduce that the distribution $\Delta=\operatorname{span}\{Y_1,Y_2\}$ (orthogonal to $N$) is integrable if and only if $\zeta=0$. If, on the contrary, we have $\nu_3=0$, the discussion is similar by defining $Y_1=\nu_2E_1-\nu_1E_2$ and $Y_2=E_3$ instead. Either way, since $N$ is a unit normal for the foliation consisting of the integral surfaces of $\Delta$ and $\operatorname{div}(N)=0$ (note that~\eqref{eqn:nabla-unimodular} implies that $\operatorname{div}(E_\alpha)=0$ for all $\alpha\in\{1,2,3\}$), we conclude that $\Sigma$ is minimal. This argument also proves the last assertion in the statement. Note also that (ii) and (iii) are known to be equivalent in general (e.g., see~\cite[Lem.~3.9]{MeeksPerez12}).
\end{proof}

\begin{remark}\label{rmk:Hzeta0}
If $G$ has constant sectional curvature, the classification of surfaces with $H=\zeta=0$ also follows from the proof of Proposition~\ref{prop:Hzeta0}; namely, we have:
\begin{itemize}
	\item any surface with zero mean curvature in $\mathbb{R}^3$ or $\mathbb{L}^3$;
	\item the integral surfaces of the distribution $\operatorname{span}\{E_1,E_2\}$ in $\widetilde{\mathrm{E}}(2)$ with a Riemannian or Lorentzian flat metric ($c_1=c_2\neq c_3=0$, $\epsilon_1=\epsilon_2$); and
	\item any timelike surface in $\mathrm{Sol}_3$ with a Lorentzian flat metric ($c_1=-c_2\neq c_3=0$, $\epsilon_1=-\epsilon_2$) satisfying $\nu_1^2=\nu_2^2$. 
\end{itemize}
In the case of $4$-dimensional isometry group, the classification also follows easily. For instance, in the Riemannian case $\epsilon_1=\epsilon_2=\epsilon_3=1$, we encounter vertical planes in $\mathrm{Nil}_3$ ($c_1=c_2=0\neq c_3$) and minimal parabolic helicoids $P_{\kappa,\tau,0}$ in $\mathbb{E}(\kappa,\tau)$ with $\kappa<0$ and $\tau\neq 0$ (these surfaces are described in~\cite{DomMan}, see Example~\ref{ex:twin}). This result can be also easily deduced from the classification of \textsc{cmc} surfaces in $\mathbb{E(\kappa,\tau)}$ with constant angle $\nu_3$ obtained by Espinar--Rosenberg~\cite{ER} (see also~\cite{DomMan}). 

We also remark that, in the Riemannian case, a classification of 2-dimensional Lie subgroups of $G$ can be found in~\cite{MeeksPerez12}. In the case $\dim(\Iso(G))=3$, the connected groups of isometries of $G$ that act on $G$ with orbits of codimension one (namely, giving rise to extrinsically homogeneous surfaces) are nothing but the $2$-dimensional Lie subgroups of $G$. The classification of the corresponding cohomogeneity one actions up to orbit equivalence can be found in~\cite[Tab.~1]{DFT25}, from where the classification of homogeneous surfaces up to congruence follows.
\end{remark}

The family of surfaces with constant angles turns out to contain all totally geodesic surfaces in our unimodular metric Lie groups. This was proved by Tsukada~\cite{Tsukada} in the Riemannian case (see also~\cite[Prop.~3.4]{MS}). Here, we will provide a different proof of this result that also applies to the Lorentzian case.

\begin{theorem}\label{thm:tot-geodesic}
Let $\Sigma$ be a Riemannian or Lorentzian surface isometrically immersed in a unimodular metric Lie group $G$ of diagonalizable type with structure constants $c_1,c_2,c_3$ and signs $\epsilon_1,\epsilon_2,\epsilon_3$, and assume $G$ has not constant sectional curvature. Up to relabeling indexes, the immersion is totally geodesic if and only if $\epsilon_3c_3=\epsilon_1c_1+\epsilon_2c_2$, with $c_2>0$ and $c_1<0$, and $\Sigma$ is (an open subset of) an integral surface of one of the distributions $\{\sqrt{-c_1}E_1+\sqrt{c_2}E_2,E_3\}$ or $\{\sqrt{-c_1}E_1-\sqrt{c_2}E_2,E_3\}$.
\end{theorem}

\begin{proof}
If the immersion is totally geodesic, Codazzi equation~\ref{eqn:comp-ii} implies that the vector field $T$ given by~\eqref{eqn:T} is identically zero, or equivalently 
\begin{equation}\label{thm:tot-geodesic:eqn1}
\epsilon_1\mu_2\mu_3\nu_1T_1+\epsilon_2\mu_1\mu_3\nu_2 T_2+\epsilon_3\mu_1\mu_2\nu_3T_3=0.
\end{equation}
The inner products of~\eqref{thm:tot-geodesic:eqn1} and $T_1,T_2,T_3$ can be computed by means of the algebraic relations~\ref{eqn:comp-iii} and yields the following equations:
\begin{equation}\label{thm:tot-geodesic:eqn2}
 	(\mu_2\mu_3-\varrho)\nu_1=(\mu_3\mu_1-\varrho)\nu_2=(\mu_1\mu_2-\varrho)\nu_3=0,
\end{equation}
where $\varrho=(\epsilon_1\mu_2\mu_3\nu_1^2+\epsilon_2\mu_3\mu_1\nu_2^2+\epsilon_3\mu_1\mu_2\nu_3^2)\epsilon_3$. If there are points where no angle function is zero, we get $\mu_2\mu_3=\mu_3\mu_1=\mu_1\mu_2=\varrho$, so the space has constant sectional curvature by Proposition~\ref{prop:dim-iso-4-6}, which is discarded in the statement. Moreover, if two of the angle functions vanish on (an open subset of) $\Sigma$, namely $\nu_1=\nu_2=0$ with no loss of generality, we get $\nu_3=\pm 1$, whence $\Sigma$ is an integral surface of the distribution $\operatorname{span}\{E_1,E_2\}$, which is integrable if and only if $c_3=0$. From the compatibility equations~\ref{eqn:comp-v}, we get $\mu_2\nu_3 T_2=\mu_3\nu_2 T_3$ and $\mu_1\nu_3 T_1=\mu_3\nu_1 T_3$, and this implies $\mu_1T_1=\mu_2T_2=0$. As $T_1$ and $T_2$ do not vanish, it must be $\mu_1=\mu_2=0$, i.e., $\epsilon_1c_1=\epsilon_2c_2$ and the sectional curvature is constant again by Proposition~\ref{prop:dim-iso-4-6}.

We are left with the case just one of the angle functions vanish in (an open subset of) $\Sigma$, namely $\nu_3=0$ and $\nu_1,\nu_2\neq 0$ with no loss of generality. Equation~\eqref{thm:tot-geodesic:eqn2} tells us that $\varrho=\mu_1\mu_3=\mu_3\mu_2$, whence $\mu_1=\mu_2$ or $\mu_3=0$.	Since $\nabla\nu_3=0$ and $S\equiv 0$, the compatibility equation~\ref{eqn:comp-v} guarantees that $\mu_2\nu_1 T_2=\mu_1\nu_2 T_1$, so that $\eta=\mu_2\nu_1E_2-\mu_1\nu_2E_1$ is normal to the immersion. In particular, $\eta$ must be collinear with $N=\epsilon_1\nu_1 E_1+\epsilon_2\nu_2 E_2$, which amounts to $\epsilon_1\mu_2\nu_1^2+\epsilon_2\mu_1\nu_2^2=0$. This last equation along with $\epsilon_1\nu_1^2+\epsilon_2\nu_2^2=\epsilon_3$ form a linear system in the unknowns $\{\nu_1^2,\nu_2^2\}$, which is compatible if and only if $\mu_1\neq\mu_2$, in which case has a unique solution. This rules out the case $\mu_1=\mu_2$, whence it must be $\mu_3=0$ (i.e., $\epsilon_3c_3=\epsilon_1c_1+\epsilon_2c_2$). The aforesaid unique solution gives constant angles
\begin{equation}\label{thm:tot-geodesic:eqn3}
\nu_1^2=\frac{\epsilon_1\mu_1}{\mu_1-\mu_2}=\frac{c_2}{\epsilon_1c_2-\epsilon_2c_1},\quad \nu_2^2=\frac{-\epsilon_2\mu_2}{\mu_1-\mu_2}=\frac{-c_1}{\epsilon_1c_2-\epsilon_2c_1},\quad \nu_3=0.
\end{equation}
We can swap the indexes $1$ and $2$ to additionally assume that $\epsilon_1c_2-\epsilon_2c_1>0$ (note that $\epsilon_1c_2=\epsilon_2c_1$ if and only if $\epsilon_1c_1=\epsilon_2c_2$ if and only if $\mu_1=\mu_2$). This implies $c_2>0$ and $c_1<0$ (if $c_1=0$ or $c_2=0$, then we get back to a space of constant sectional curvature), and hence the unit normal is given by
	\[N=\epsilon_1\nu_1E_1+\epsilon_2\nu_2E_2=\frac{\epsilon_1\sqrt{c_2}}{\sqrt{\epsilon_1c_2-\epsilon_2c_1}}E_1\pm\frac{\epsilon_2\sqrt{-c_1}}{\sqrt{\epsilon_1c_2-\epsilon_2c_1}}E_2.\]
In other words, $\Sigma$ is contained locally in an integral surface of $\{\sqrt{-c_1}E_1+\sqrt{c_2}E_2,E_3\}$ or $\{\sqrt{-c_1}E_1-\sqrt{c_2}E_2,E_3\}$, both of them being integrable by Proposition~\ref{prop:Hzeta0} since $\sum_{\alpha=1}^3c_\alpha\nu_\alpha^2=0$ trivially holds true in view of~\eqref{thm:tot-geodesic:eqn3}.

Although the argument so far is local, the existence of the totally geodesic surfaces implies that only one $\mu_i$ is zero, so for a given metric Lie group $G$, there are at most two distributions whose integral surfaces provide the totally geodesic examples. Since both distributions are transversal, we cannot glue smoothly two integral surfaces coming from different distributions.
\end{proof}

Notice that the foliations by totally geodesic surfaces spanned by both distributions in Theorem~\ref{thm:tot-geodesic} could be congruent to each other, see~\cite[Rmk.~4.1]{DFT25}.

\section{Bianchi-Cartan-Vr\u{a}nceanu spaces}\label{sec:bcv}

Proposition~\ref{prop:second-integration} shows how a solution of the equation $M^{-1}\df M=\Omega+L(M)$ characterizes the immersion up to left-translations. This raises the natural question of whether or not one can get such a map $M:\Sigma\to\SO^\epsilon_3(\R)$ with less information. By looking at~\eqref{eqn:M}, we can say that Theorem~\ref{thm:fundamental} prescribes the whole of $M$ and Theorem~\ref{thm:angles} prescribes only the last column, whereas Daniel~\cite{Daniel07,Daniel09} in $\E(\kappa,\tau)$-spaces prescribes the last row plus the shape operator.

Observe that, in order to compute the $\mathfrak{so}^\epsilon_3$-symmetric matrix of $1$-forms $\Theta=\Omega+L(M)$ a priori we need the intrinsic forms $\omega^i_j$, with $i,j\in\{1,2\}$, the shape operator $S$ (encoded in $\omega^1_3$ and $\omega^2_3$), and all entries of the matrix $M$. If $\Sigma$ is simply connected, then the equation $M^{-1}\df M=\Theta$ has a solution $M$ if and only if $\Theta$ has zero \emph{Darboux derivative}, i.e., $\df \Theta +\tfrac{1}{2}[\Theta,\Theta] = 0$ (this is a classical result in the theory of moving frames, see~\cite[Ch.~3]{Sharpe97}). In this section, we will prescribe $S$ as a fundamental datum as Daniel does (in contrast with our previous results in Section~\ref{sec:fundamental}), which is the same as prescribing $\Omega$, and additionally we would like to have $L(M)$ without knowing all the information contained in $M$. Next result shows that we are forced to assume $\dim(\Iso(G))\geq 4$.

\begin{proposition}\label{prop:M-LM}
The matrices $M$ and $L(M)$ determine each other algebraically if and only if $\epsilon_1c_1,\epsilon_2c_2,\epsilon_3c_3$ are all distinct, i.e., if and only if $\dim(\Iso(G))=3$.
\end{proposition}

\begin{proof}
The matrix $L(M)$ contains three different entries, namely $\sum_{\gamma=1}^3M^\gamma_\delta\eta^\gamma$, with $\delta\in\{1,2,3\}$, which result in six functions $\sum_{\gamma=1}^3\epsilon_\gamma\mu_\gamma M^\gamma_\delta M^\gamma_\zeta$, with $\delta,\zeta\in\{1,2,3\}$ and $\delta\leq \zeta$, when considering the components in $\omega^1$ and $\omega^2$. These make up the entries of $Q=M^t\mathcal{E}DM$, where $D$ and $\mathcal{E}$ are the diagonal matrices with entries $\{\mu_1,\mu_2,\mu_3\}$ and $\{\epsilon_1,\epsilon_2,\epsilon_3\}$ in the diagonal, respectively. In other words, $L(M)$ contains the same information as $Q$. Since $M\in\mathrm{SO}_3^\epsilon(\R)$, we have $M^t\mathcal{E}=\mathcal{E}M^{-1}=(M\mathcal{E})^{-1}$ and hence $Q\mathcal{E}=(M\mathcal{E})^{-1}D(M\mathcal{E})$ is diagonalizable. This means that $\mu_1,\mu_2,\mu_3$ must be the eigenvalues of $Q\mathcal{E}$ and the columns of $(M\mathcal{E})^{-1}=M^t\mathcal{E}$ (i.e., the rows of $M$ multiplied by $\epsilon_1,\epsilon_2,\epsilon_3$) are the corresponding eigenvectors of $Q\mathcal{E}$. 

If all $\epsilon_1c_1,\epsilon_2c_2,\epsilon_3c_3$ are distinct, then all the eigenspaces of $M$ are $1$-dimensional, so the rows of $M$ are determined up to a constant. Since $M\in\SO_3^\epsilon(M)$, the rows (and the columns) of $M$ form an orthonormal frame for the metric of signature $\epsilon_1,\epsilon_2,\epsilon_3$, whence $M$ is determined by $Q$ up to the signs of its rows. Indeed, it is determined up to changing the signs of two rows because it must be $\det(M)=1$, which agrees with the composition in $G$ with the an element of the stabilizer (in the case of $\mathrm{Sol}_3$ with a metric homothetic to the standard one, we can also swap two columns corresponding to non zero $c_i$, which changes all signs of the $c_i$, see Remark~\ref{rmk:stabilizer-data}). On the contrary, if two or more of the $\epsilon_ic_i$ coincide, we can change the corresponding rows of $M$ by other two orthonormal vectors in the same eigenspace and the new matrix $\widetilde M$ still satisfies $Q=\widetilde M^tBD\widetilde M$. 
\end{proof}

\subsection{Recovering the immersion from the shape operator}\label{subsec:daniel-fundamental}
We will assume in the rest of this section that $G$ belongs to one of the families $\E(\kappa,\tau)$, $\mathbb{L}(\kappa,\tau)$ or $\widehat{\mathbb{L}}(\kappa,\tau)$. The key observation is that in these ambient spaces a subset of the fundamental equations only involves one angle function and the corresponding tangent component. Since we will keep on assuming that $(\hat\epsilon_1,\hat\epsilon_2,\hat\epsilon_3)=(\epsilon_1,\epsilon_2,\epsilon_3)$ in order to apply the results in previous sections (just for the sake of simplicity), in exchange we have to deal with two possible scenarios:

\noindent\textbf{Case A:} $\epsilon_1c_1=\epsilon_2c_2\neq\epsilon_3c_3\neq 0$. It comprises all surfaces in $\mathbb{E}(\kappa,\tau)$, in which case we recover Daniel's fundamental theorem~\cite{Daniel07}, spacelike surfaces in $\mathbb{L}(\kappa,\tau)$, and timelike surfaces in $\widehat{\mathbb{L}}(\kappa,\tau)$. Since the vector field defined by~\eqref{eqn:T} reads $T=-c_3(\epsilon_1c_1-\epsilon_3c_3)\nu_3T_3$ and the last equations in~\ref{eqn:comp-iv} and~\ref{eqn:comp-v} can be rewritten by taking into account~\eqref{eqn:JEi} and the fact that $\langle X,T_1\rangle T_2-\langle X,T_2\rangle T_1=\epsilon_1\epsilon_2\epsilon_3\nu_3JX$, we end up with five conditions:
\begin{enumerate}
	\item[(\textsc{i}\textsubscript{A})] $K=\epsilon_3\det(S)+\epsilon_1\epsilon_2(\tfrac{1}{4}\epsilon_3c_3^2+(\epsilon_1c_1-\epsilon_3c_3)c_3\nu_3^2)$,
	\item[(\textsc{ii}\textsubscript{A})] $\nabla_XSY-\nabla_YSX-S[X,Y]=\epsilon_1\epsilon_2\epsilon_3(\epsilon_1c_1-\epsilon_3c_3)c_3\nu_3(\langle Y,T_3\rangle X-\langle X,T_3\rangle Y)$,
	\item[(\textsc{iii}\textsubscript{A})] $\langle T_3,T_3\rangle=\epsilon_3-\epsilon_3\nu_3^2$,
	\item[(\textsc{iv}\textsubscript{A})] $\nabla_XT_3=\epsilon_3\nu_3SX-\frac{1}{2}c_3\nu_3JX$,
	\item[(\textsc{v}\textsubscript{A})] $\nabla\nu_3=-ST_3-\frac{1}{2}\epsilon_3c_3JT_3$.
\end{enumerate}
Also, by plugging $\mu_1=\mu_2=\frac{\epsilon_3c_3}{2}$ and $\mu_3=\epsilon_1c_1-\frac{\epsilon_3c_3}{2}$ into~\eqref{eqn:L}, and since the columns of $M$ form a semi-Riemannian orthonormal basis, we find that
\begin{equation}\label{eqn:M:dim4:A}
L(M)=(\epsilon_1c_1-\epsilon_3c_3)\epsilon_1\epsilon_2\!\!\left(\begin{smallmatrix}
	0& \epsilon_2M^3_3&-\epsilon_3 M^3_2\\
	-\epsilon_1M^3_3&0&\epsilon_3 M^3_1\\
	\epsilon_1M^3_2&-\epsilon_2M^3_1&0
\end{smallmatrix}\right)\!\eta^3+\tfrac{\epsilon_3c_3}{2}\!\!\left(\begin{smallmatrix}
	0&0&-\epsilon_1\omega^2\\
	0&0&\epsilon_2\omega^1\\
	\epsilon_3\omega^2&-\epsilon_3\omega^1&0
\end{smallmatrix}\right)
\end{equation}
only depends on the last row of $M$ (as expected by Proposition~\ref{prop:M-LM}), which is in turn related to the fundamental data as $(M^3_1,M^3_2,M^3_3)=(\epsilon_3T^1_3,\epsilon_3T^2_3,\epsilon_3\nu_3)$, where we will denote $T^i_3=\langle T_3,e_i\rangle$ for $i\in\{1,2\}$. 

\noindent\textbf{Case B:} $\epsilon_1c_1=\epsilon_3c_3\neq\epsilon_2c_2\neq 0$, which consists of all surfaces in $\E(\kappa,\tau)$ once again, the timelike surfaces in $\mathbb{L}(\kappa,\tau)$, and spacelike surfaces in $\widehat{\mathbb{L}}(\kappa,\tau)$. In this case, we get $T=-(\epsilon_1c_1-\epsilon_2c_2)c_2T_2$ and $\langle X,T_3\rangle T_1-\langle X,T_1\rangle T_3=\epsilon_1\nu_2JX$, and this results in the equations:
\begin{enumerate}
	\item[(\textsc{i}\textsubscript{B})] $K=\epsilon_3\det(S)+\epsilon_1\epsilon_2(\tfrac{1}{4}\epsilon_3c_2^2+(\epsilon_1c_1-\epsilon_2c_2)c_2\nu_2^2)$,
	\item[(\textsc{ii}\textsubscript{B})] $\nabla_XSY-\nabla_YSX-S[X,Y]=\epsilon_1\epsilon_2\epsilon_3(\epsilon_1c_1-\epsilon_2c_2)c_2\nu_2(\langle Y,T_2\rangle X-\langle X,T_2\rangle Y)$,
	\item[(\textsc{iii}\textsubscript{B})] $\langle T_2,T_2\rangle=\epsilon_2-\epsilon_3\nu_2^2$,
	\item[(\textsc{iv}\textsubscript{B})] $\nabla_XT_2=\epsilon_3\nu_2SX-\frac{1}{2}\epsilon_2\epsilon_3c_2\nu_2JX$,
	\item[(\textsc{v}\textsubscript{B})] $\nabla\nu_2=-ST_2-\frac{1}{2}\epsilon_2c_2JT_2$.
\end{enumerate}
We can also simplify the matrix $L(M)$ in~\eqref{eqn:L} in this case as
\begin{equation}\label{eqn:M:dim4:B}
L(M)=(\epsilon_1c_1-\epsilon_2c_2)\epsilon_1\epsilon_3\!\!\left(\begin{smallmatrix}
	0& \epsilon_2M^2_3&-\epsilon_3 M^2_2\\
	-\epsilon_1M^2_3&0&\epsilon_3 M^2_1\\
	\epsilon_1M^2_2&-\epsilon_2M^2_1&0
\end{smallmatrix}\right)\!\eta^2+\tfrac{\epsilon_2c_2}{2}\!\!\left(\begin{smallmatrix}
	0&0&-\epsilon_1\omega^2\\
	0&0&\epsilon_2\omega^1\\
	\epsilon_3\omega^2&-\epsilon_3\omega^1&0
\end{smallmatrix}\right),
\end{equation}
which now depends on the second row of $M$ with $(M^2_1,M^2_2,M^2_3)=(\epsilon_2T^1_2,\epsilon_2T^2_2,\epsilon_2\nu_2)$.

Cases A and B can be considered at once by letting $\hat\epsilon_3=-\epsilon_3$ in the Lorentzian setting. This is not just swapping the subindexes $2$ and $3$ to go from (\textsc{i}\textsubscript{A})-(\textsc{v}\textsubscript{A}) to (\textsc{i}\textsubscript{B})-(\textsc{v}\textsubscript{B}) and viceversa, but it agrees with our original compatibility equations in Proposition~\ref{prop:compatibility}. Consequently, we can write a unified statement in all spaces with $4$-dimensional isometry group, but still the proof must be split in cases A and B (which are rather similar but none of them implies the other algebraically speaking) in order to apply Proposition~\ref{prop:second-integration}.

\begin{theorem}\label{thm:dim4}
	Let $G$ be a semi-Riemannian unimodular metric Lie group of diagonalizable type with structure constants $c_1,c_2,c_3$ and signs $\epsilon_1,\epsilon_2,\epsilon_3$ such that $\epsilon_1c_1=\epsilon_2c_2\neq\epsilon_3c_3\neq 0$. Let $\hat\epsilon_1,\hat\epsilon_2,\hat\epsilon_3$ be a permutation of $\epsilon_1,\epsilon_2,\epsilon_3$ and suppose that $\Sigma$ is a simply connected surface with signs $\hat\epsilon_1,\hat\epsilon_2$. Assume there is a smooth function $\nu_3\in\mathcal{C}^\infty(\Sigma)$, a smooth vector field $T_3\in\mathfrak{X}(\Sigma)$, and a smooth field of symmetric operators $S\in\mathfrak{T}_{1,1}(\Sigma)$, satisfying 
	\begin{enumerate}
	\item[$(\textsc{i}^*)$] $K=\hat\epsilon_3\det(S)+\epsilon_1\epsilon_2\epsilon_3\hat\epsilon_3(\tfrac{1}{4}\hat\epsilon_3c_3^2+(\epsilon_1c_1-\epsilon_3c_3)c_3\nu_3^2)$,
	\item[$(\textsc{ii}^*)$] $\nabla_XSY-\nabla_YSX-S[X,Y]=\epsilon_1\epsilon_2\epsilon_3c_3(\epsilon_1c_1-\epsilon_3c_3)\nu_3(\langle Y,T_3\rangle X-\langle X,T_3\rangle Y)$,
	\item[$(\textsc{iii}^*)$] $\langle T_3,T_3\rangle=\epsilon_3-\hat\epsilon_3\nu_3^2$,
	\item[$(\textsc{iv}^*)$] $\nabla_XT_3=\hat\epsilon_3\nu_3SX-\frac{1}{2}\epsilon_3\hat\epsilon_3c_3\nu_3JX$,
	\item[$(\textsc{v}^*)$] $\nabla\nu_3=-ST_3-\frac{1}{2}\epsilon_3c_3JT_3$.
\end{enumerate}
Then, there exists an isometric immersion $\phi:\Sigma\to G$ with third angle function $\nu_3$, third tangent projection $T_3$ and shape operator $S$, and it is unique up to orientation-preserving isometries that also preserve the vertical direction $E_3$.
\end{theorem}

\begin{proof}
We will consider Case A first. Motivated by~\eqref{eqn:M:dim4:A}, define
\begin{equation}\label{thm:dim4:eqn0}
L=(\epsilon_1c_1-\epsilon_3c_3)\epsilon_1\epsilon_2\!\left(\!\begin{smallmatrix}
	0& \epsilon_2\nu_3&-\epsilon_3T^2_3\\
	-\epsilon_1\nu_3&0&\epsilon_3T^1_3\\
	\epsilon_1T^2_3&-\epsilon_2T^1_3&0
\end{smallmatrix}\!\right)\!T_3^\flat+\tfrac{\epsilon_3c_3}{2}\!\!\left(\begin{smallmatrix}
	0&0&-\epsilon_1\omega^2\\
	0&0&\epsilon_2\omega^1\\
	\epsilon_3\omega^2&-\epsilon_3\omega^1&0
\end{smallmatrix}\right),
\end{equation}
where $T_3=\epsilon_1T^1_3e_1+\epsilon_2T^2_3e_2$ is expressed in an orthonormal frame $\{e_1,e_2\}$ in $\Sigma$ with $\langle e_i,e_j\rangle=\epsilon_j\delta^i_j$. Here, $\{\omega^1,\omega^2\}$ is the dual frame and we have also considered the flat $1$-form $T_3^\flat=T_3^1\omega^1+T_3^2\omega^2$ for simplicity. Finally, let $\Theta=\Omega+L$, where the matrix of $1$-forms $\Omega$ is defined by~\eqref{eqn:curvature-forms-as-products} using the given $S$.

Since $e_1$ and $e_2$ are unitary,~\eqref{eqn:curvature-forms} implies that $\omega_1^1\equiv\omega_2^2\equiv0$, and Equations~\eqref{eqn:domega3} and~\eqref{eqn:domega4} actually have a few non-trivial terms. By Equations (\textsc{i}\textsubscript{A}) and (\textsc{ii}\textsubscript{A}) and the fact that $\omega^1_3\wedge\omega^3_2=-\epsilon_2 \epsilon_3 \det(S)\omega^1\wedge\omega^2$ we can write these non-trivial terms as
\begin{equation}\label{thm:dim4:eqn1}
	\begin{aligned}
		\df\omega^1_2 + \omega^1_3 \wedge \omega^3_2 &=\epsilon_1(\tfrac{1}{4}\epsilon_3c_3^2+(\epsilon_1c_1-\epsilon_3c_3)c_3\nu_3^2)\,\omega^1\!\wedge \omega^2,\\
		\df\omega^3_1+\omega^3_2\wedge\omega^2_1&=\epsilon_2c_3(\epsilon_1c_1-\epsilon_3c_3)\nu_3 T^2_3\, \omega^1\!\wedge \omega^2,\\
		\df\omega^3_2+\omega^3_1\wedge\omega^1_2&=-\epsilon_1c_3(\epsilon_1c_1-\epsilon_3c_3)\nu_3 T^1_3\, \omega^1\!\wedge \omega^2.
\end{aligned}\end{equation}
On the other hand, (\textsc{iv}\textsubscript{A}) and (\textsc{v}\textsubscript{A}) can be used along with~\eqref{eqn:curvature-forms-as-products} to compute $\df\nu_3(e_k)=\langle\nabla\nu_3,e_k\rangle$ and $\df T^i_3(e_k)=\langle\nabla_{e_k}T_3,e_i\rangle+\langle T_3,\nabla_{e_k}e_i\rangle$, obtaining
\begin{equation}\label{thm:dim4:eqn2}
	\begin{aligned}
	\df\nu_3&=T_3^1\omega^1_3+T^2_3\omega^2_3+\tfrac{\epsilon_2\epsilon_3}{2}c_3T^2_3\omega^1-\tfrac{\epsilon_1\epsilon_3}{2}c_3T^1_3\omega^2,\\
	\df T^1_3&=\nu_3\omega^3_1+T^2_3\omega^2_1+\tfrac{1}{2}c_3\nu_3\omega^2,\\
	\df T^2_3&=\nu_3\omega^3_2+T^1_3\omega^1_2-\tfrac{1}{2}c_3\nu_3\omega^1.
\end{aligned}
\end{equation}
Equation (\textsc{iv}\textsubscript{A}) and the symmetry of $S$ also yield
\begin{align*}
\df T_3^\flat(e_1,e_2)&=e_1(\langle T_3,e_2\rangle)-e_2(\langle T_3,e_1\rangle)-\langle T_3,[e_1,e_2]\rangle\\
&=\langle\nabla_{e_1}T_3,e_2\rangle-\langle\nabla_{e_2}T_3,e_1\rangle=\tfrac{1}{2}c_3\nu_3(\langle Je_2,e_1\rangle-\langle Je_1,e_2\rangle)=-c_3\nu_3,
\end{align*}
so that we get $\df T_3^\flat=-c_3\nu_3\,\omega^1\wedge\omega^2$. We can then employ this differential and Equations~\eqref{thm:dim4:eqn1},~\eqref{thm:dim4:eqn2} and~\eqref{eqn:domega1} to compute $\df\Theta$. Also taking into account~\eqref{eqn:domega2}, it can be shown that $\df\Theta+\frac{1}{2}[\Theta,\Theta]=0$ after a rather long (but simple) calculation. Here, the Lie bracket is computed in the Lie algebra $\mathfrak{so}^\epsilon_3(\R)$, which is a Lie subalgebra of $\mathfrak{gl}_3(\R)$, whence $\tfrac{1}{2}[\Theta,\Theta]^\alpha_\beta=(\Theta\wedge\Theta)^\alpha_\beta=\sum_{\gamma=1}^3\Theta^\alpha_\gamma\wedge\Theta^\gamma_\beta$ (i.e., the wedge product has to be computed matrixwise). By means of~\cite[Thm.~3.7.14]{Sharpe97}, since $\Sigma$ is assumed simply connected, we can ensure that there is smooth map $M:\Sigma\to\mathrm{SO}_3^\epsilon(\R)$ such that $M^{-1}\df M=\Theta=\Omega+L$ and such an $M$ is unique up to left multiplication by a constant matrix $A\in\mathrm{SO}_3^\epsilon(\R)$.

Since $M$ takes values in $\mathrm{SO}_3^\epsilon(\R)$, its last row $(M^3_1,M^3_2,M^3_3)$ has norm $\epsilon_3$, and so does $(T^1_3,T^2_3,\nu_3)$ because of (\textsc{iii}\textsubscript{A}). Therefore, we can choose $A$ such that these two vectors coincide at some $p_0\in\Sigma$. We claim that then they coincide in all $\Sigma$. For, we observe that expanding out the last row of $\df M=M\Omega+ML$, we find out that $(f_1,f_2,f_3)=(M^3_1,M^3_2,M^3_3)$ is a solution to the following first-order system of linear differential equations:
\begin{equation}\label{thm:dim4:eqn3}
	\begin{aligned}
	\df f_1&=f_2\omega^2_1+f_3\omega^3_1+\epsilon_2(\epsilon_1c_1-\epsilon_3c_3)(T^2_3f_3-\nu_3f_2)T_3^\flat+\tfrac{1}{2}c_3f_3\omega^2,\\
	\df f_2&=f_3\omega^3_2+f_1\omega^1_2+\epsilon_1(\epsilon_1c_1-\epsilon_3c_3)(\nu_3f_1-T^1_3f_3)T_3^\flat-\tfrac{1}{2}c_3f_3\omega^1,\\
	\df f_3&=f_1\omega^1_3+f_2\omega^2_3+\epsilon_1\epsilon_2\epsilon_3(\epsilon_1c_1-\epsilon_3c_3)(T^2_3f_1-T^1_3f_2)T_3^\flat\\
	&\qquad+\tfrac{\epsilon_3}{2}c_3(\epsilon_2f_2\omega^1-\epsilon_1f_1\omega^2).
\end{aligned}
\end{equation}
As $(f_1,f_2,f_3)=(T^1_3,T^2_3,\nu_3)$ is also a solution of~\eqref{thm:dim4:eqn3} by~\eqref{thm:dim4:eqn2} and both solutions coincide at $p_0$, the claim is proved (as a matter of fact, we are using the same result in~\cite{Sharpe97} but in the Lie group $\R^3$ rather than $\mathrm{SO}^\epsilon_3(\R)$). A simple comparison between~\eqref{eqn:M:dim4:A} and~\eqref{thm:dim4:eqn0} implies that $L=L(M)$ so the immersion exists by Proposition~\ref{prop:second-integration}. Uniqueness follows from the uniqueness in Proposition~\ref{prop:second-integration} (which allows us to fix the image of $p_0$) plus the degrees of freedom we have when choosing $A$ (which amounts to choosing an element of the stabilizer of $p_0$).

Assume now that we have Case B after relabeling indexes, and define
\[
L=(\epsilon_1c_1-\epsilon_2c_2)\epsilon_1\epsilon_3\!\!\left(\begin{smallmatrix}
	0& \epsilon_2\nu_2&-\epsilon_3 T^2_2\\
	-\epsilon_1\nu_2&0&\epsilon_3 T^1_2\\
	\epsilon_1T^2_2&-\epsilon_2T^1_2&0
\end{smallmatrix}\right)\!T_2^\flat+\tfrac{\epsilon_2c_2}{2}\!\left(\begin{smallmatrix}
	0&0&-\epsilon_1\omega^2\\
	0&0&\epsilon_2\omega^1\\
	\epsilon_3\omega^2&-\epsilon_3\omega^1&0
\end{smallmatrix}\right),
\]
where $T_2=\epsilon_1T^1_2e_1+\epsilon_2T^2_2e_2$. Gauss and Codazzi equations (\textsc{i}\textsubscript{B}) and (\textsc{ii}\textsubscript{B}) give
	\begin{align*}
		\df\omega^1_2 + \omega^1_3 \wedge \omega^3_2 &=\epsilon_1(\tfrac{1}{4}\epsilon_3c_2^2+(\epsilon_1c_1-\epsilon_2c_2)c_2\nu_2^2)\,\omega^1\!\wedge \omega^2,\\
		\df\omega^3_1+\omega^3_2\wedge\omega^2_1&=\epsilon_2(\epsilon_1c_1-\epsilon_2c_2)c_2\nu_2 T^2_2\, \omega^1\!\wedge \omega^2,\\
		\df\omega^3_2+\omega^3_1\wedge\omega^1_2&=-\epsilon_1(\epsilon_1c_1-\epsilon_2c_2)c_2\nu_2 T^1_2\, \omega^1\!\wedge \omega^2,
\end{align*}
whereas (\textsc{iv}\textsubscript{B}) and (\textsc{v}\textsubscript{B}) can be rewritten as
\begin{align*}
	\df\nu_2&=T_2^1\omega^1_3+T^2_2\omega^2_3+\tfrac{1}{2}c_2T^2_2\omega^1-\tfrac{1}{2}\epsilon_1\epsilon_2c_2T^1_2\omega^2,\\
	\df T^1_2&=\nu_2\omega^3_1+T^2_2\omega^2_1+\tfrac{1}{2}\epsilon_2\epsilon_3c_2\nu_2\omega^2,\\
	\df T^2_2&=\nu_2\omega^3_2+T^1_2\omega^1_2-\tfrac{1}{2}\epsilon_2\epsilon_3c_2\nu_2\omega^1.
\end{align*}
Note also that $\df T_2^\flat=-\epsilon_2\epsilon_3c_2\nu_2\omega^1\wedge\omega^2$. The rest of the argument is completely analogous to Case A, and will be omitted.
\end{proof}

\begin{remark}\label{rmk:product-compatibility}
As $\tau\to 0$, the unimodular Lie group structure we have been considering in $\E(\kappa,\tau)$, $\mathbb{L}(\kappa,\tau)$ and $\widehat{\mathbb{L}}(\kappa,\tau)$ disappears. However, these families converge to the product manifolds $\mathbb{M}^2(\kappa)\times\R$, $\mathbb{M}^2(\kappa)\times\R_1$ and $\mathbb{M}^2_1(\kappa)\times\R$ by just taking limits of the metrics given by~\eqref{eqn:Ekt-Lkt-metric} and~\eqref{eqn:darkLkt-metric}. The fundamental equations (\textsc{i}$^*$)-(\textsc{v}$^*$) in Theorem~\ref{thm:dim4} also have nice limits as $\tau\to 0$ because they ultimately depend on $c_1c_3=\pm\kappa$ and $c_3^2=\tau^2$. Interestingly, the limit version of Theorem~\ref{thm:dim4} still holds true by the work of Daniel in $\mathbb{M}^n(\kappa)\times\R$~\cite[Thm.~3.3]{Daniel09}, which was later adapted to the Lorentzian case $\mathbb{M}^n(\kappa)\times\R_1$ by Roth~\cite[Thm.~1]{Roth11}. Finally, the case of $\mathbb{M}^2_1(\kappa)\times\R$ is a particular case of the warped products considered by Lawn--Ortega~\cite[Thm.~1]{LO15}.
\end{remark}

\subsection{Lorentzian versions of the Daniel correspondence} 
An isometric deformation of an immersion is equivalent to a deformation of its fundamental data which preserves the fundamental equations. Daniel's generalization of the Lawson correspondence relies on the fact that the traceless operator is rotated, which gives a nice control of Gauss and Codazzi equations, as we show next.

Consider the rotation of angle $\theta\in\R$ in the tangent bundle of $\Sigma$ we have already introduced in Theorem~\ref{thm:angles} given by 
\[\Rot_\theta=\begin{cases}
	\cos(\theta)\,\mathrm{id}+\sin(\theta)J&\text{if $\Sigma$ is Riemannian},\\
	\cosh(\theta)\,\mathrm{id}+\sinh(\theta)J&\text{if $\Sigma$ is Lorentzian}.
\end{cases}\]
Note that $\Rot_\theta$ is a field of isometric operators (recall that $J$ is not an isometry in the Lorentzian case), and satisfies $\det(\Rot_\theta)=1$, since, in our usual positive orthonormal frame $\{e_1,e_2\}$ with signs $\hat\epsilon_1,\hat\epsilon_2$, we can express
\[\Rot_\theta\equiv\begin{pmatrix}
	\cos\theta&\sin\theta\\
	-\sin\theta&\cos\theta
\end{pmatrix}\qquad\text{or}\qquad
\Rot_\theta\equiv\begin{pmatrix}
	\cosh\theta&-\hat\epsilon_1\sinh\theta\\
	\hat\epsilon_2\sinh\theta&\cosh\theta
\end{pmatrix}, \]
depending on whether $\hat\epsilon_1\hat\epsilon_2=1$ or $\hat\epsilon_1\hat\epsilon_2=-1$, respectively. 

\begin{lemma}\label{lemma:traceless}
Let $S$ and $\widetilde S$ be fields of symmetric operators in a Riemannian or Lorentzian surface $\Sigma$ and assume that $H=\epsilon\trace(S)$ and $\widetilde H=\epsilon\trace(\widetilde S)$ are constants for some $\epsilon\in\{-1,1\}$. If $\widetilde S-\epsilon\widetilde{H}\mathrm{id}={\Rot_\theta}\circ{(S-\epsilon H\mathrm{id})}$ for some $\theta\in\mathbb{R}$, then the following assertions hold true:
	\begin{enumerate}[label=\emph{(\alph*)}]
		\item $\det(\widetilde S)-\widetilde H^2=\det(S)-H^2$,
		\item $\nabla_X\widetilde SY-\nabla_Y\widetilde SX-\widetilde S[X,Y]=\Rot_\theta(\nabla_X SY-\nabla_Y SX- S[X,Y])$,
		\item $\widetilde S-\widetilde \tau J={\Rot_\theta}\circ{(S-\tau J)}$, where $\tau,\widetilde{\tau}\in\R$ are constants such that
		\begin{equation}\label{lemma:traceless:eqn1}
		(\epsilon\widetilde H,\widetilde{\tau})=\begin{cases}
		(\epsilon H\cos\theta+\tau\sin\theta,-\epsilon H\sin\theta+\tau\cos\theta)&\text{if $\Sigma$ is Riem.},\\
		(\epsilon H\cosh\theta-\tau\sinh\theta,-\epsilon H\sinh\theta+\tau\cosh\theta)&\text{if $\Sigma$ is Lorentz}.\end{cases}\end{equation}
	\end{enumerate}
\end{lemma}

\begin{proof}
If $\Sigma$ is Riemannian, this result is essentially proved in the work of Daniel~\cite{Daniel07,Daniel09}, so we will briefly discuss the Lorentzian case. By taking determinants in $\widetilde S-\epsilon\widetilde{H}\mathrm{id}={\Rot_\theta}\circ{(S-\epsilon H\mathrm{id})}$, we get $\det(\widetilde S-\epsilon\widetilde{H}\mathrm{id})=\det(S-\epsilon H\mathrm{id})$. Since $S$ is symmetric, it takes the matrix form $S\equiv\left(\begin{smallmatrix}
	a&\hat\epsilon_1b\\\hat\epsilon_2b&c\end{smallmatrix}\right)$, whence $H=\frac{1}{2}\epsilon(a+c)$ and
\begin{align*}
\det(S-\epsilon H\,
\mathrm{id})&=\det\left(\begin{smallmatrix}
	a-\epsilon H&\hat\epsilon_1b\\\hat\epsilon_2b&c-\epsilon H\end{smallmatrix}\right)\\
	&=(ac-\hat\epsilon_1\hat\epsilon_2b^2)-\epsilon H(a+c)+H^2=\det(S)-H^2.
\end{align*}
Similarly, we find that $\det(\widetilde S-\epsilon\widetilde H\mathrm{id})=\det(\widetilde S)-\widetilde H^2$ and we have item (a). 

Item (b) is a direct consequence of the constancy of $H$ and $\widetilde H$, along with the fact that $\nabla_Z$ and $J$ commute for any vector field $Z$. This last claim in turn follows readily by considering the skew-symmetric $2$-form $\alpha(X,Y)=\nabla_XJY-J\nabla_XY$ and checking that $\alpha(e_1,e_2)=0$. Finally, as for item (c), by subtracting $\widetilde S-\epsilon\widetilde{H}\mathrm{id}={\Rot_\theta}\circ{(S-\epsilon H\mathrm{id})}$ and $\widetilde S-\widetilde aJ={\Rot_\theta}\circ{(S-aJ)}$, this last condition will hold if and only if $\epsilon H\,\mathrm{id}-aJ={\Rot_\theta}\circ{(\epsilon H\,\mathrm{id}-aJ)}$. Expanding $\Rot_\theta$ as a linear combination of $\mathrm{id}$ and $J$, the coefficients give~\eqref{lemma:traceless:eqn1} componentwise, so we are done.
\end{proof}

The fundamental equations in Theorem~\ref{thm:dim4} can be particularized for the case of $\mathbb{L}(\kappa,\tau)$ by setting $c_1=c_2=\frac{-\kappa}{2\tau}$ and $c_3=-2\tau\neq 0$ with signs $\epsilon_1=\epsilon_2=1$ and $\epsilon_3=-1$ as explained in Remark~\ref{rmk:extension-previous-metrics}. This gives
	\begin{enumerate}
	\item[$(\textsc{i}^*_{\mathbb{L}})$] $K=\hat\epsilon_3\det(S)-\tau^2-\hat\epsilon_3(\kappa+4\tau^2)\nu_3^2$,
	\item[$(\textsc{ii}^*_{\mathbb{L}})$] $\nabla_XSY-\nabla_YSX-S[X,Y]=-(\kappa+4\tau^2)\nu_3(\langle Y,T_3\rangle X-\langle X,T_3\rangle Y)$,
	\item[$(\textsc{iii}^*_{\mathbb{L}})$] $\langle T_3,T_3\rangle=-1-\hat\epsilon_3\nu_3^2$,
	\item[$(\textsc{iv}^*_{\mathbb{L}})$] $\nabla_XT_3=\hat\epsilon_3\nu_3(SX-\tau JX)$,
	\item[$(\textsc{v}^*_{\mathbb{L}})$] $\nabla\nu_3=-ST_3-\tau JT_3$,
\end{enumerate}
Likewise, in the case of $\widehat{\mathbb{L}}(\kappa,\tau)$-spaces, we can set $c_1=\frac{-\kappa}{2\tau}$, $c_2=\frac{\kappa}{2\tau}$ and $c_3=2\tau\neq 0$ with signs $\epsilon_1=\epsilon_3=1$ and $\epsilon_2=-1$ in order to get alike equations
	\begin{enumerate}
	\item[$(\textsc{i}^*_{\widehat{\mathbb{L}}})$] $K=\hat\epsilon_3\det(S)-\tau^2+\hat\epsilon_3(\kappa+4\tau^2)\nu_3^2$,
	\item[$(\textsc{ii}^*_{\widehat{\mathbb{L}}})$] $\nabla_XSY-\nabla_YSX-S[X,Y]=(\kappa+4\tau^2)\nu_3(\langle Y,T_3\rangle X-\langle X,T_3\rangle Y)$,
	\item[$(\textsc{iii}^*_{\widehat{\mathbb{L}}})$] $\langle T_3,T_3\rangle=1-\hat\epsilon_3\nu_3^2$,
	\item[$(\textsc{iv}^*_{\widehat{\mathbb{L}}})$] $\nabla_XT_3=\hat\epsilon_3\nu_3(SX-\tau JX)$,
	\item[$(\textsc{v}^*_{\widehat{\mathbb{L}}})$] $\nabla\nu_3=-ST_3-\tau JT_3$,
\end{enumerate}
Recall that all the above equations still apply in the limit product case $\tau=0$, see Remark~\ref{rmk:product-compatibility}. Next result shows a transformation of the fundamental data of a \textsc{cmc} immersion that preserves the above equations, which gives the desired isometric deformation. The proof relies on Lemma~\ref{lemma:traceless} and the fact that
\begin{equation}\label{eqn:Rot-matrix}
\Rot_\theta(\langle X,Z\rangle Y-\langle Y,Z\rangle X)=\langle X,\Rot_\theta(Z)\rangle Y-\langle Y,\Rot_\theta(Z)\rangle X\end{equation}
for all vector fields $X,Y,Z\in\mathfrak{X}(\Sigma)$. Since the argument is similar to the case of $\mathbb{E}(\kappa,\tau)$ in~\cite{Daniel07}, we will not give the details. Notice that our correspondence applies to both the spacelike case ($\hat\epsilon_3=-1$) and the timelike case ($\hat\epsilon_3=1$).

\begin{corollary}\label{coro:fundamental-L}
Let $\phi:\Sigma\to \mathbb{L}(\kappa,\tau)$ (resp.\ $\widehat{\mathbb{L}}(\kappa,\tau)$) be an isometric inmersion of a Riemannian or Lorentzian surface $\Sigma$ with \textsc{cmc} $H\in\R$ and sign $\hat\epsilon_3$. Let $(S,J,T_3,\nu_3)$ be its fundamental data. Given $\theta\in\R$, we define $\widetilde H,\widetilde\tau\in\R$ such that~\eqref{lemma:traceless:eqn1} holds for $\epsilon=\hat{\epsilon}_3$, and take $\widetilde\kappa\in\mathbb{R}$ such that $\widetilde\kappa+4\widetilde\tau^2=\kappa+4\tau^2$. Then
\[(\widetilde S_3,\widetilde J,\widetilde T_3,\widetilde\nu_3)=(\hat\epsilon_3 H\,\mathrm{id}+{\Rot_\theta}\circ{(S-\hat\epsilon_3 H\,\mathrm{id})},J,\Rot_\theta(T_3),\nu_3)\]
are the fundamental data of another isometric immersion $\widetilde\phi:\Sigma\to\widehat{\mathbb{L}}(\widetilde\kappa,\widetilde\tau)$ (resp.\ $\widehat{\mathbb{L}}(\widetilde\kappa,\widetilde\tau)$) with \textsc{cmc} $\widetilde H$ and the same sign $\hat\epsilon_3$.
\end{corollary}

\section{On the generalized Lawson correspondence}\label{sec:lawson}

This last section will be devoted to investigate if there is a possible Lawson-type correspondence between two unimodular metric Lie groups other than the Daniel correspondence. We will restrict to the Riemannian case and give a negative answer, so it is plausible that the same non-existence result holds in the Lorentzian setting. 

Fix constants $H,\widetilde H\in\mathbb{R}$ and Riemannian unimodular metric Lie groups $G$ and $\widetilde G$. By a \emph{Lawson-type correspondence} we mean a bijection (up to ambient isometries) between the family of isometric $H$-immersions $\phi:\Sigma\to G$ and the family of isometric $\widetilde H$-immersions $\widetilde\phi:\Sigma\to\widetilde G$ satisfying the following properties:
\begin{enumerate}
	\item There is a \emph{phase angle} $\theta\in\R$ (not depending on $\phi$) such that $\widetilde S-\widetilde H\,\mathrm{id}={\Rot_\theta}\circ{(S-H\,\mathrm{id})}$, where $S$ and $\widetilde S$ denote the shape operators.
	\item Both immersions induce the same orientation in $\Sigma$ (i.e., $\widetilde J=J$) and the left-invariant Gauss maps $g=(\nu_1,\nu_2,\nu_3)$ and $\widetilde g=(\widetilde \nu_1,\widetilde \nu_2,\widetilde \nu_3)$ induce the same orientation in $\mathbb{S}^2$ (i.e., at points $p\in\Sigma$ where both $\df g_p$ and $\df\widetilde g_p$ are linear isomorphisms, $\{\df g_p(e_1),\df g_p(e_2),g(p)\}$ and $\{\df \widetilde g_p(e_1),\df \widetilde g_p(e_2),\widetilde g(p)\}$ define the same orientation as bases of $\R^3$ for any basis $\{e_1,e_2\}$ of $T_p\Sigma$).
\end{enumerate}

\begin{remark}\label{rmk:gauss-orientation}
This last requirement in item (2) for the orientations of the Gauss maps is necessary because it encodes the ambient orientation given by the frames $\{E_1,E_2,E_3\}$ and $\{\widetilde E_1,\widetilde E_2,\widetilde E_3\}$, which are crucial for the definition of $g$ and $\widetilde g$, respectively. We have already met a similar situation in Corollary~\ref{coro:uniqueness-angles}, but here we are not assuming that $g$ and $\widetilde g$ differ by an isometry of $\mathbb{S}^2$, not even pointwise speaking. Notice that, in Daniel correspondence, $\widetilde g$ and $g$ pointwise differ by an ambient rotation because the angle function $\nu_3$ is preserved.

It is also interesting to point out the not so well known fact that the Daniel correspondence generalizes the classical Lawson correspondence just partially. When we restrict to $\E(\kappa,\tau)$-spaces with $\kappa-4\tau^2=0$, the Daniel correspondence has one degree of freedom (the phase angle determines the target space); on the contrary, the Lawson correspondence has two degrees of freedom. For instance, transforming minimal surfaces in $\mathbb{S}^3$ into minimal surfaces in $\mathbb{S}^3$ by a rotation of the shape operator is considered within Lawson's correspondence (see also~\cite[\S5]{CCC}), but not explained by Daniel's. It is precisely in those cases in the intersection of Lawson and Daniel in which one can assume that $g$ and $\widetilde g$ are globally congruent.
\end{remark} 

We will say that a Lawson-type correspondence is \emph{trivial} if it is just the identity (i.e., it amounts to some ambient isometry), or consists in a relabeling of the indexes $c_1,c_2,c_3$ or a global change of their signs. Note that all these cases will appear at some moment in the discussion below.

\begin{theorem}\label{thm:lawson}
	Any non-trivial Lawson-type correspondence between unimodular metric Lie groups is an instance of the Daniel correspondence in $\mathbb{E}(\kappa,\tau)$-spaces or an instance of the Lawson correspondence in space forms.
\end{theorem}

The rest of the paper treats the proof of Theorem~\ref{thm:lawson}. We will begin by developing the implications of the rotation of the traceless operators of two corresponding immersions $\phi:\Sigma\to G$ and $\widetilde\phi:\Sigma\to\widetilde G$, some of which have been already shown in Lemma~\ref{lemma:traceless}. In what follows, all symbols with tilde will represent elements computed with respect to $\widetilde\phi$ or $\widetilde G$, in contrast with the same symbols without tilde that indicate they belong to $\phi$ or $G$, respectively. Since the correspondence must apply to all $H$-surfaces in $G$, we will further assume that the left-invariant Gauss map $g=(\nu_1,\nu_2,\nu_3):\Sigma\to\mathbb{S}^2$ has maximal rank $2$ (i.e., it is a local diffeomorphism). In other words, we can discard cases that lead necessarily to rank of $g$ being less than or equal to $1$, for this means the correspondence does not apply to all $H$-surfaces. By the same reason, we will also use the fact that $(\nu_1,\nu_2,\nu_3)$ and $(\widetilde\nu_1,\widetilde\nu_2,\widetilde\nu_3)$ can be potentially any point of $\mathbb S^2$. 

Firstly, Gauss equation~\ref{eqn:comp-i} gives $K=\det(S)-\sum_{i=1}^3 a_i \nu_i^2$, which must be equal to $\det(\widetilde{S})-\sum_{i=1}^3 \widetilde{a}_i \widetilde{\nu}_i^2$ since $K$ is the intrinsic curvature of $\Sigma$. By item (a) of Lemma~\ref{lemma:traceless}, we get
\begin{equation}\label{eqn:vandermonde:eqn1}
H^2 - \sum_{i=1}^3 a_i \nu_i^2 = \widetilde{H}^2- \sum_i \widetilde{a}_i \widetilde{\nu}_i^2,
\end{equation}
where the constants $a_i$ are defined in~\eqref{eqn:ai}. If we denote $\chi=\widetilde{H}^2 - H^2$ and use the fact that $\sum_{i=1}^3\widetilde\nu_i^2=1$, then Equation~\eqref{eqn:vandermonde:eqn1} can be rewritten as
\begin{equation}\label{eqn:vandermonde:eqn2}
 	\sum_{i=1}^3 a_i\nu_i^2 = \sum_{i=1}^3 (\widetilde{a}_i-\chi)\widetilde{\nu}_i^2.
\end{equation}

Secondly, by item (b) of Lemma~\ref{lemma:traceless} (here we are actually using that $H$ and $\widetilde{H}$ are constants) and~\eqref{eqn:Rot-matrix}, we can write Codazzi equation~\ref{eqn:comp-ii} as
\begin{equation}\label{eqn:vandermonde:eqn3}\langle X,\Rot_\theta(T)\rangle Y-\langle Y,\Rot_\theta(T)\rangle X=\langle X,\widetilde T\rangle Y-\langle Y,\widetilde T\rangle X,
\end{equation}
for all $X,Y\in\mathfrak{X}(\Sigma)$, where $T=\sum_{i=1}^3a_i\nu_iT_i$ is the vector field given by~\eqref{eqn:T}. By plugging $X=e_1$ and $Y=e_2$ into~\eqref{eqn:vandermonde:eqn3}, where $\{e_1,e_2\}$ is an orthonormal frame of $\Sigma$, we get $\widetilde T=\Rot_\theta(T)$. In particular, the weaker condition $\|\widetilde{T}\|^2 = \|T\|^2$ yields $\sum_{i,j=1}^3 a_ia_j\nu_i\nu_j\langle T_i,T_j\rangle=\sum_{i,j=1}^3 \widetilde a_i\widetilde a_j\widetilde \nu_i\widetilde \nu_j\langle \widetilde T_i,\widetilde T_j\rangle$, which can be simplified by means of the algebraic relations~\ref{eqn:comp-iii} to
\begin{equation}\label{eqn:vandermonde:eqn4}
\sum_{i=1}^3 a_i^2\nu_i^2- \left(\sum_{i=1}^3 a_i\nu_i^2\right)^2 = \sum_{i=1}^3 \widetilde{a}_i^2\widetilde{\nu}_i^2 - \left(\sum_{i=1}^3 \widetilde{a}_i \widetilde{\nu}_i^2 \right)^2.
\end{equation}
Using~\eqref{eqn:vandermonde:eqn1}, after some manipulations~\eqref{eqn:vandermonde:eqn4} can be expressed as
\begin{equation}\label{eqn:vandermonde:eqn5}
\sum_{i=1}^3 a_i^2\nu_i^2 = \sum_{i=1}^3 (\widetilde{a}_i-\chi)^2\widetilde{\nu}_i^2.
\end{equation}

Thirdly and lastly, we have $\sum_{i=1}^3\nu_i^2=\sum_{i=1}^3\widetilde\nu_i^2=1$. This, combined with~\eqref{eqn:vandermonde:eqn2} and~\eqref{eqn:vandermonde:eqn5}, implies that the corresponding angle functions must satisfy the following system of equations, where we denote $b_i=\widetilde a_i-\chi$ for simplicity:
\begin{equation}\label{eqn:vandermonde:eqn6}
	1=\sum_{i=1}^3 \nu_i^2 = \sum_{i=1}^3 \widetilde{\nu}_i^2, \quad \sum_{i=1}^3 a_i\nu_i^2 = \sum_{i=1}^3 b_i\widetilde{\nu}_i^2, \quad \sum_{i=1}^3 a_i^2\nu_i^2 = \sum_{i=1}^3 b_i^2\widetilde{\nu}_i^2.
\end{equation}
These are linear relations in the squared angles $x_i = \nu_i^2$ and $\widetilde{x}_i = \widetilde{\nu}_i^2$ such that, when expressed in matrix form, display the following Vandermonde style:
\begin{equation}\label{eqn:vandermonde:eqn7}
\begin{pmatrix}
	1 & 1 & 1\\
	a_1 & a_2 & a_3\\
	a_1^2 & a_2^2 & a_3^2
\end{pmatrix}
\begin{pmatrix}
	x_1\\
	x_2\\
	x_3
\end{pmatrix}=\begin{pmatrix}
	1 & 1 & 1\\
	b_1 & b_2 & b_3\\
	b_1^2 & b_2^2 & b_3^2
\end{pmatrix}\begin{pmatrix}
\widetilde{x}_1\\
\widetilde{x}_2\\
\widetilde{x}_3
\end{pmatrix}.
\end{equation}
Let $A$ and $B$ be the $3\times 3$ matrices in the left-hand side and the right-hand side of~\eqref{eqn:vandermonde:eqn7}. If $A$ is regular, then so is $B$, for otherwise we deduce that $(x_1,x_2,x_3)^t=A^{-1}B(\widetilde x_1,\widetilde x_2,\widetilde x_3)^t$ lies in a proper linear subspace of $\R^3$, which violates our hypothesis that $g$ has maximal rank. (This argument actually shows that $A$ and $B$ must have the same rank.) Accordingly, we shall distinguish two cases depending on whether or not $A$ and $B$ are invertible.

\subsection{The non-degenerate case.} If both matrices in~\eqref{eqn:vandermonde:eqn7} are regular, the $x_i$ determine the $\widetilde x_i$ and viceversa. More precisely, if we consider the affine plane
\[\Delta = \{(x,y,z)\in \mathbb{R}^3: x + y + z = 1\},\]
the identity~\eqref{eqn:vandermonde:eqn7} can be expressed as an linear isomorphism $F:\R^3\rightarrow\R^3$ that preserves $\Delta$ by sending $(x_1, x_2, x_3)$ to $(\widetilde{x}_1, \widetilde{x}_2, \widetilde{x}_3)$. Furthermore, $F$ must induce a bijection from the set of points of $\Delta$ in first octant onto itself. This is because the $x_i$ and the $\widetilde x_i$ are potentially arbitrary non-negative values that add up to $1$. This means that the matrices that represent $F$ and $F^{-1}$ (i.e., $A^{-1}B$ and $B^{-1}A$) must have non-negative entries. Since these matrices are inverse, it easily follows that $F$ is just a permutation of axes, that is, $A$ and $B$ can be obtained from each other by a permutation of their columns, or equivalently the $a_i$ are a permutation of the $b_i$. There is no loss of generality if we rearrange the structure constants in $\widetilde G$ in order to assume that $a_i=b_i$ for all $i\in\{1,2,3\}$, whence $F$ is the identity map and consequently $\widetilde\nu_i^2=\nu_i^2$ for all $i\in\{1,2,3\}$.

Since we can change the signs of any two of the $\nu_i$ by composition of $\phi$ with an isometry of $G$ that preserves the ambient orientation and $T$ (see Remark~\ref{rmk:stabilizer-data}), we can further assume that either $\widetilde\nu_i=\nu_i$ or $\widetilde\nu_i=-\nu_i$ for all $i\in\{1,2,3\}$. However, the latter is not admissible since it contradicts our hypothesis (2) of a Lawson-type correspondence (the left-invariant Gauss maps are related by the antipodal map, which reverses the orientation of $\mathbb{S}^2$).

On the one hand, it follows from the algebraic relations~\ref{eqn:comp-iii} that $\langle\widetilde T_i,\widetilde T_j\rangle=\langle T_i,T_j\rangle$ for all $i,j\in\{1,2,3\}$, and this means that there is an field of isometric operators $R$ such that $R(T_i)=\widetilde T_i$. On the other hand, the vector fields defined by~\eqref{eqn:Xi} only depend on the angle functions and $J$, so they satisfy $\widetilde X_i=X_i$. This also means that the functions defined by~\eqref{eqn:gaussdata:eqn2} satisfy $\widetilde\psi=\psi$. Since we are assuming that the Gauss map $g$ has rank $2$, Proposition~\ref{prop:Hzeta0} tells us that $\psi$ and $\widetilde\psi$ are nowhere zero, whence~\eqref{eqn:gaussdata:eqn3} implies that $R$ is actually a rotation at each tangent plane of $\Sigma$. (This argument is rather similar to that of the proof item (c) of Theorem~\ref{thm:angles}, but now we have to deal with two possibly different ambient spaces.) Since $\sum_{i=1}^3\widetilde \nu_i\widetilde T_i=0$ and $\widetilde\nu_i=\nu_i$, we get
\begin{equation}\label{eqn:lawson:eqn1}
\Rot_\theta(T)=\widetilde T=\sum_{i=1}^3\widetilde a_i\widetilde\nu_i\widetilde T_i=\sum_{i=1}^3(a_i+\chi)\widetilde\nu_i\widetilde T_i=\sum_{i=1}^3a_i\nu_iR(T_i)= R(T),
\end{equation}
whence we get $R=\Rot_\theta$ at points with $T\neq 0$ (which can be assumed by Lemma~\ref{lem:T0} below). Using~\eqref{eqn:gaussdata:eqn1} and the fact that $\sum_{i=1}^3\nu_iX_i=0$, we get 
	\begin{align}
		\sum_{i=1}^3a_i\nu_i X_i&=\sum_{i=1}^3a_i\nu_i\widetilde X_i=\sum_{i=1}^3(\widetilde a_i-\chi)\widetilde \nu_i\widetilde X_i=\sum_{i=1}^3\widetilde a_i\widetilde \nu_i\widetilde X_i\notag\\
		&=-2\widetilde HJ\widetilde T+\widetilde\zeta\widetilde T=\Rot_\theta(-2\widetilde HJ T+\widetilde\zeta T)\notag\\
		&=\cos\theta(-2\widetilde HJ T+\widetilde\zeta T)+\sin\theta(2\widetilde HT+\widetilde\zeta JT).\label{eqn:lawson:eqn2}
	\end{align}
	We can also compute in a different way $\sum_{i=1}^3a_i\nu_i X_i=-2 HJ T+\zeta T$ so a comparison with~\eqref{eqn:lawson:eqn2}  yields the system of equations
	\begin{equation}\label{eqn:lawson:eqn3}
	\left.\begin{array}{r}
		\zeta=\widetilde\zeta\cos\theta+2\widetilde H\sin\theta\\
		2H=2\widetilde H\cos\theta-\widetilde\zeta\sin\theta
	\end{array}\right\}.
	\end{equation}
	We deduce that $\sin\theta=0$, for otherwise the second equation in~\eqref{eqn:lawson:eqn3} implies that $\widetilde\zeta$ must be constant and this would impose a restriction on $\nu_1,\nu_2,\nu_3$ by~\eqref{eqn:vandermonde:eqn7} which contradicts our assumption that the Gauss map $g$ has maximal rank. By looking at~\eqref{eqn:lawson:eqn3} again, the condition $\sin\theta=0$ gives $\widetilde H=\pm H$ (so that $\chi=0$), whence $\widetilde a_i=a_i$ for all $i\in\{1,2,3\}$. This yields $\widetilde\mu_i\widetilde\mu_j=\mu_i\mu_j$ for all $i\neq j$. Since we are assuming that no two of the $a_i$ are equal, it follows that $\widetilde\mu_i=\pm\widetilde\mu_i$ for all $i\in\{1,2,3\}$, or equivalently $\widetilde c_i=\pm c_i$ for all $i\in\{1,2,3\}$ (the same choice of sign for all $i$). Therefore, we can assume that $\widetilde G=G$ since changing the signs of all $c_i$ gives an isometric metric Lie group. Now we have $\widetilde\zeta=\zeta$, so the first equation in~\eqref{eqn:lawson:eqn3} gives $\theta=0$, and~\eqref{eqn:lawson:eqn1} implies that $\Rot_\theta=\id$ at all points. In particular, $\widetilde T_i=T_i$ for all $i\in\{1,2,3\}$. By Theorem~\ref{thm:fundamental}, we get that $\phi$ and $\widetilde\phi$ (after all the aforesaid normalizations) differ by a left-translation. so it is a trivial correspondence.

\begin{lemma}\label{lem:T0}
	If $G$ has not constant curvature, then the rank of the left-invariant Gauss map of an isometric immersion $\phi:\Sigma\to G$ such that $T\equiv 0$ is at most $1$.
\end{lemma}

\begin{proof}
We can mimick the first part of proof of Theorem~\ref{thm:tot-geodesic} to reach~\eqref{thm:tot-geodesic:eqn2}. If no angle function vanishes, then $\mu_1\mu_2=\mu_2\mu_3=\mu_1\mu_3$ and $G$ has constant curvature; otherwise, the Gauss map clearly does not have maximal rank.
\end{proof}

\subsection{The degenerate case}

We will now assume that the matrices in~\eqref{eqn:vandermonde:eqn7} are not invertible. First, notice that their common rank equals $1$ if and only if $a_1=a_2=a_3$ and $b_1=b_2=b_3$, so that $G$ and $G'$ have constant curvature by Proposition~\ref{prop:dim-iso-4-6} and the rotation of the traceless operator characterizes the classical Lawson correspondence (the only relevant fundamental datum is $S$ whilst Gauss and Codazzi are in turn the only fundamental equations). Accordingly, we can further assume that the rank equals $2$, so that exactly two of the $a_i$ and two of the $b_i$ coincide. We can relabel indexes to assume $a_1=a_2\neq a_3$ and $b_1=b_2\neq b_3$. Let us make two observations about these conditions.

On the one hand, the left-hand side of~\eqref{eqn:vandermonde:eqn7} ranges over the segment with endpoints $(1,a_1,a_1^2)^t$ and $(1,a_3,a_3^2)^t$, whereas the right-hand side ranges over the segment with endpoints $(1,b_1,b_1^2)^t$ and $(1,b_3,b_3^2)^t$ as both $(x_1,x_2,x_3)$ and $(\widetilde x_1,\widetilde x_2,\widetilde x_3)$ run over the triples of positive numbers with sum $1$. This implies that $\{b_1,b_3\}$ must be a permutation of $\{a_1,a_3\}$. The second equation of~\eqref{eqn:vandermonde:eqn6} can be expressed as
\begin{equation}\label{eqn:segment-degenerate}
a_1+(a_3-a_1)\nu_3^2=b_1+(b_3-b_1)\widetilde\nu_3^2.
\end{equation}
We have two options in view of~\eqref{eqn:segment-degenerate}: either $\widetilde\nu_3^2=\nu_3^2$ if $(a_1,a_3)=(b_1,b_3)$ or $\widetilde\nu_3^2=1-\nu_3^2$ if $(a_1,a_3)=(b_3,b_1)$.

On the other hand, $a_1=a_2$ is equivalent to $(\mu_1-\mu_2)\mu_3=0$. If $\mu_1=\mu_2$ we get $\E(\kappa,\tau)$-spaces with $c_1=c_2=\frac{\kappa}{2\tau}$ and $c_3=2\tau\neq 0$ (we have to exclude $c_3=0$ since it leads to $a_1=a_2=a_3$), whereas $\mu_3=0$ leads to metric Lie groups with $c_3=c_1+c_2$. We will say $G$ is of \textsc{bcv}-type or \textsc{sum}-type, respectively. Both types of metric Lie groups have in common the interesting property that the vector field $T$ defined by~\eqref{eqn:T} only depends on $T_3$. More specifically, we have 
\begin{equation}\label{eqn:T-degenerate}
T =\begin{cases} c_3(c_3-c_1)\nu_3T_3&\text{if $G$ is of \textsc{bcv}-type},\\
2 c_1 c_2 \nu_3 T_3& \text{if $G$ is of \textsc{sum}-type}.
\end{cases}
\end{equation}
We will analyze each of the three possible scenarios depending on which types of metric Lie groups we would like to match by a potential Lawson-type correspondence. For each case, we also have to analyze the two options $(a_1,a_3)=(b_1,b_3)$ and $(a_1,a_3)=(b_3,b_1)$ we have discussed previously. Next lemma will simplify some of the forthcoming calculations (actually, it holds for any Riemannian unimodular metric Lie group, not only for those of \textsc{bcv}-type or \textsc{sum}-type).

\begin{lemma}\label{lemma:T-divergence}~
\begin{enumerate}[label=\emph{(\alph*)}]
	\item $\operatorname{div}(T_3)=2H\nu_3+(c_2-c_1)\nu_1\nu_2$,
	\item $\operatorname{div}(JT_3)=c_3\nu_3$.
\end{enumerate}\end{lemma}

\begin{proof}
Let $\{e_1,e_2\}$ be a orthonormal frame consisting of principal directions in $\Sigma$ such that $\{e_1,e_2,N\}$ is positively oriented. On the one hand, we have
\begin{equation}\label{lemma:T-divergence:eqn1}
	\operatorname{div}(T)=\sum_{i=1}^2\langle\nabla_{e_i}T_3,e_i\rangle=\sum_{i=1}^2\langle\overline\nabla_{e_i}E_3,e_i\rangle-\nu_3\sum_{i=1}^2\langle\nabla_{e_i}N,e_i\rangle.
\end{equation}
The last sum clearly equals $-2H$, whereas the other term in the right-hand side of~\eqref{lemma:T-divergence:eqn1} swiftly gives $(c_2-c_1)\nu_1\nu_2$ when expanding $\overline\nabla_{e_i}E_3=\sum_{k=1}^2\langle e_i,E_k\rangle\overline\nabla_{E_k}E_3$ and using the ambient connection in~\eqref{eqn:nabla-unimodular}. On the other hand, we can compute
\begin{equation}\label{lemma:T-divergence:eqn2}
\operatorname{div}(JT)=\sum_{i=1}^2\langle\overline\nabla_{e_i}(N\times T_3),e_i\rangle=\sum_{i=1}^2\langle\overline\nabla_{e_i}N\times E_3,e_i\rangle+\sum_{i=1}^2\langle N\times\overline\nabla_{e_i} E_3,e_i\rangle.
\end{equation}
The first sum in the right-hand side of~\eqref{lemma:T-divergence:eqn2} vanishes since $\overline\nabla_{e_i}N$ is proportional to $e_i$, whereas the second sum equals $c_3\nu_3$ by expanding each $\overline\nabla_{e_i}E_3$ as above.
\end{proof}

\subsubsection{Case \textsc{bcv-bcv}} Let $c_1=c_2=\frac{\kappa}{2\tau}$ and $\widetilde c_1=\widetilde c_2=\frac{\widetilde\kappa}{2\widetilde\tau}$ with $c_3=2\tau\neq 0$ and $\widetilde c_3=2\widetilde\tau\neq 0$, i.e., we have a correspondence between groups of \textsc{bcv}-type. This implies that $a_1=a_2=-\tau^2$ and $a_3=-\kappa+3\tau^2$ (and likewise $b_1=b_2=-\widetilde\tau^2-\chi$ and $b_3=-\widetilde\kappa+3\widetilde\tau^2-\chi$ with $\chi=\widetilde H^2-H^2$).

Assume first that $(a_1,a_3)=(b_3,b_1)$. We claim this does not lead to any correspondence. Since $a_3-a_1=b_1-b_3$, we infer that $\widetilde\kappa-4\widetilde\tau^2=-(\kappa-4\tau^2)\neq 0$ ($G$ and $\widetilde G$ do not have constant curvature). Consequently, we can rewrite $\widetilde T=\Rot_\theta(T)$ as $\widetilde\nu_3\widetilde T_3=-\nu_3\Rot_\theta(T_3)$. This subcase gives $\widetilde\nu_3^2=1-\nu_3^2$, whence $\widetilde\nu_3\nabla\widetilde\nu_3=-\nu_3\nabla\nu_3$. Expanding out $\widetilde\nu_3\nabla\widetilde\nu_3$ by the fundamental equation~\ref{eqn:comp-v}, we reach 
\begin{align}
	\langle\widetilde\nu_3\nabla\widetilde\nu_3,X\rangle
	&=-\widetilde\nu_3\langle\widetilde S\widetilde T_3+\widetilde\tau J\widetilde T_3,X\rangle
	=-\widetilde\nu_3\langle\widetilde S X-\widetilde\tau JX, \widetilde T_3\rangle\notag\\
	&=\nu_3\langle\Rot_\theta(SX-HX)+\widetilde HX-\widetilde\tau JX, \Rot_\theta(T_3)\rangle\notag\\
	&=\nu_3\langle SX-HX,T_3\rangle+\nu_3\langle\widetilde HX-\widetilde\tau JX,\Rot_\theta(T_3)\rangle,\notag\\
	&=\nu_3\langle ST_3-HT_3,X\rangle+\nu_3\langle\widetilde H\Rot_\theta(T_3)+\widetilde\tau J\Rot_\theta(T_3),X\rangle,\label{eqn:first-case:eqn1}
\end{align}
for all tangent vector fields $X$. However, we can also compute $\langle\widetilde\nu_3\nabla\widetilde\nu_3,X\rangle=\langle-\nu_3\nabla\nu_3,X\rangle=\nu_3\langle ST_3+\tau JT_3,X\rangle$. When comparing this and~\eqref{eqn:first-case:eqn1}, we can get rid of the term $ST_3$ and reach the following non trivial condition:
\begin{equation}\label{eqn:first-case:eqn2}
\nu_3\Rot_\theta(\widetilde HT_3+\widetilde\tau JT_3)=\nu_3(HT_3+\tau JT_3).
\end{equation}
Working around some point where $\nu_3\neq 0$ and $\nu_3\neq 1$, we can take squared norms in~\eqref{eqn:first-case:eqn2} to get $\widetilde H^2+\widetilde\tau^2=H^2+\tau^2$. An easy algebraic manipulation of the conditions $a_1=b_3$ and $a_3=b_1$ now yields $\widetilde\kappa-4\widetilde\tau^2=\kappa-4\tau^2=0$, which is a contradiction and the claim is proved.

Therefore, we have the other subcase $(a_1,a_3)=(b_1,b_3)$ with $\widetilde\nu_3^2=\nu_3^2$. This means that $\widetilde H^2+\widetilde\tau^2=H^2+\tau^2$ and $\widetilde\kappa-4\widetilde\tau^2=\kappa-4\tau^2$. We can work in an open subset where $\nu_3\neq 0$ and apply an isometry to $\widetilde\phi$ to further assume that $\widetilde\nu_3=\nu_3$ (see Remark~\ref{rmk:stabilizer-data}), so that $\widetilde T=\Rot_\theta(T)$ reads $\widetilde T_3=\Rot_\theta(T_3)$. Taking divergences in both $\widetilde T_3=\Rot_\theta(T_3)$ and $J\widetilde T_3=\Rot_\theta(JT_3)$ by means of Lemma~\ref{lemma:T-divergence}, we find the following two equations:
\begin{align*}
2\widetilde H\nu_3&=\operatorname{div}(\widetilde T_3)=\cos\theta\operatorname{div}(T_3)+\sin\theta\operatorname{div}(JT_3)=2H\nu_3\cos\theta+2\tau\nu_3\sin\theta,\\
2\widetilde\tau\nu_3&=\operatorname{div}(J\widetilde T_3)=\cos\theta\operatorname{div}(JT_3)-\sin\theta\operatorname{div}(T_3)=2\tau\nu_3\cos\theta-2H\nu_3\sin\theta.
\end{align*}
This agrees with~\eqref{lemma:traceless:eqn1} and implies that $\widetilde H+i\widetilde\tau=e^{-i\theta}(H+i\tau)$. All in all, we recover the Daniel correspondence in this case.

\subsubsection{Case \textsc{sum-sum}} Assume that $c_3=c_1+c_2$ and $\widetilde c_3=\widetilde c_1+\widetilde c_2$, i.e., we have a correspondence between two metric Lie groups of \textsc{sum}-type. We have $a_1=a_2=-c_1c_2$ and $a_3=c_1c_2$ (and likewise $b_1=b_2=-\widetilde c_1\widetilde c_2-\chi$ and $b_3=\widetilde c_1\widetilde c_2-\chi$). 

Assume first that $(a_1,a_3)=(b_1,b_3)$. This swiftly gives $\widetilde H^2=H^2$ and $\widetilde c_1\widetilde c_2=c_1c_2$. We can change $\widetilde\phi$ by an ambient isometry to further assume $\widetilde\nu_3=\nu_3$ as in the above case \textsc{bcv-bcv}, which does not alter $\widetilde T$. Since $c_1c_2\neq 0$ (otherwise, we have $\widetilde c_1\widetilde c_2=0$ and we can indeed reduce to the case \textsc{bcv-bcv}), the condition $\widetilde T=\Rot_\theta(T)$ amounts to $\widetilde T_3=\Rot_\theta(T_3)$ at points with $\nu_3\neq 0$ by~\eqref{eqn:T-degenerate}. By expanding the identity $\langle\nabla\widetilde\nu_3,X\rangle=\langle\nabla\nu_3,X\rangle$ in the spirit of~\eqref{eqn:first-case:eqn1} and using~\ref{eqn:comp-v}, we obtain
\begin{equation}\label{eqn:second-case:eqn1}
\widetilde c_2\widetilde\nu_2\widetilde T_1-\widetilde c_1\widetilde\nu_1\widetilde T_2-\widetilde H\widetilde T_3= c_2\nu_2 T_1- c_1\nu_1 T_2- H T_3.
\end{equation}
The divergence trick (similar to the case \textsc{bcv-bcv}) gives by Lemma~\ref{lemma:T-divergence}
\begin{equation}\label{eqn:second-case:eqn2}
	\begin{aligned}
2\widetilde H\nu_3+(\widetilde c_2-\widetilde c_1)\widetilde\nu_1\widetilde\nu_2&=(2H\nu_3+(c_2-c_1)\nu_1\nu_2)\cos\theta+(\mu_1+\mu_2)\nu_3\sin\theta,\\
\widetilde c_3\nu_3&=c_3\nu_3\cos\theta-(2H\nu_3+(c_2-c_1)\nu_1\nu_2)\sin\theta.
\end{aligned}
\end{equation}
The second equation in~\eqref{eqn:second-case:eqn2} implies that the rank of the Gauss map is at most $1$ unless $\sin\theta=0$. Therefore, we can assume that $\theta=0$ or $\theta=\pi$. 
\begin{itemize}
	\item If $\theta=0$, Equation~\eqref{eqn:second-case:eqn2} yields $\widetilde c_3=c_3$ and hence $\{\widetilde c_1,\widetilde c_2\}$ is a permutation of $\{c_1,c_2\}$ (both pairs have the same sum and product), so we can further assume that $\widetilde c_1=c_1$ and $\widetilde c_2=c_2$ by swapping indexes. Multiplying~\eqref{eqn:second-case:eqn1} by $\widetilde T_3=T_3$ by means of the algebraic relations~\ref{eqn:comp-iii}, we get 
\begin{equation}\label{eqn:second-case:eqn3}
(c_2-c_1)\widetilde\nu_1\widetilde\nu_2\nu_3-\widetilde H(1-\nu_3^2)=(c_2-c_1)\nu_1\nu_2\nu_3- H(1-\nu_3^2).
\end{equation}
By subtracting from~\eqref{eqn:second-case:eqn3} the first equation in~\eqref{eqn:second-case:eqn2} multiplied by $\nu_3$, we reach $\widetilde H=H$, so that $\widetilde\nu_1\widetilde\nu_2=\nu_1\nu_2$. Since we also have $\widetilde\nu_1^2+\widetilde\nu_2^2=\nu_1^2+\nu_2^2$, we can change the signs of both $\nu_1$ and $\nu_2$ (by Remark~\ref{rmk:stabilizer-data}, this does not alter $\nu_3$ or $T_3$) if necessary to assume that $\{\widetilde\nu_1,\widetilde\nu_2\}$ is a permutation of $\{\nu_1,\nu_2\}$. By the algebraic relations, $\{\widetilde T_1,\widetilde T_2\}$ must be the same permutation of $\{T_1,T_2\}$, but then~\eqref{eqn:second-case:eqn1} implies that $\widetilde T_1=T_1$ and $\widetilde T_2=T_2$ unless $c_1=-c_2$, i.e., $c_3=0$. 

All in all, if $c_3\neq 0$, we get $\widetilde G=G$ and $\widetilde T_i=T_i$, so we conclude that the Lawson-type correspondence is just the identity by Theorem~\ref{thm:fundamental}. If $c_3=0$, then $\widetilde G=G=\mathrm{Sol}_3$ equipped with the standard metric (up to an homothety), and we also have $\widetilde T_i=T_i$ after applying an isometry (this is the only space in which we have a stabilizer of eight elements, see Remark~\ref{rmk:stabilizer-data}). Either way, we get a trivial Lawson-type correspondence as expected.

\item If $\theta=\pi$, we get $\widetilde c_3=-c_3$ and hence $\{\widetilde c_1,\widetilde c_2\}$ is a permutation of $\{-c_1,-c_2\}$. Therefore, we can reason as in the case $\theta=0$ with $\widetilde T_3=-T_3$ or even reduce to the case $\theta=0$ by changing the signs of $\widetilde c_1,\widetilde c_2,\widetilde c_3$.
\end{itemize}

Assume now that $(a_1,a_3)=(b_3,b_1)$, which is equivalent to having both $\widetilde H^2=H^2$ and $\widetilde c_1\widetilde c_2=-c_1c_2\neq 0$. This case also has $\widetilde\nu_ 3^2=1-\nu_3^2$, whence $\widetilde\nu_3\nabla\widetilde\nu_3=-\nu_3\nabla\nu_3$ and $\widetilde\nu_3\widetilde T_3=-\nu_3\Rot_\theta(T_3)$ by~\eqref{eqn:T-degenerate}. The same idea that we used in~\eqref{eqn:first-case:eqn1} comes in handy to expand $\langle\widetilde\nu_3\nabla\widetilde\nu_3,X\rangle=-\langle\nu_3\nabla\nu_3,X\rangle$, resulting the identity
\begin{equation}\label{eqn:second-case:eqn4}
(\widetilde c_2\widetilde\nu_2\widetilde T_1-\widetilde c_1\widetilde\nu_1\widetilde T_2+\widetilde H\widetilde T_3)\widetilde\nu_3=-(c_2\nu_2 T_1- c_1\nu_1 T_2+H T_3)\nu_3.
\end{equation}
Using the algebraic relations~\ref{eqn:comp-iii}, we can express
\begin{equation}\label{eqn:second-case:eqn5}
c_2\nu_2T_1-c_1\nu_1T_2=\frac{(c_1-c_2)\nu_1\nu_2\nu_3}{1-\nu_3^2}T_3+\frac{c_1\nu_1^2+c_2\nu_2^2}{1-\nu_3^2}JT_3
\end{equation}
at points where $\nu_3^2\neq 1$, and the same formula~\eqref{eqn:second-case:eqn5} works with tildes if $\widetilde\nu_3^2\neq 1$, or equivalently $\nu_3\neq 0$. We can now employ~\eqref{eqn:second-case:eqn5} to write~\eqref{eqn:second-case:eqn4} as combination of $\widetilde T_3$ and $J\widetilde T_3$ by expressing $\nu_3 T_3=-\widetilde\nu_3\Rot_{-\theta}(\widetilde T_3)$ and $1-\widetilde\nu_3^2=\nu_3^2$. When matching coefficients on both sides, we can isolate two polynomials on $\widetilde\nu_1,\widetilde\nu_2,\widetilde\nu_3$:
\begin{equation}\label{eqn:second-case:eqn6}
	\begin{aligned}
	\widetilde\nu_1\widetilde\nu_2\widetilde\nu_3&=\tfrac{ H(1-\nu_3^2)+( c_1- c_2) \nu_1 \nu_2 \nu_3}{(\widetilde c_1-\widetilde c_2)(1- \nu_3^2)}\nu_3^2\cos\theta+\tfrac{( c_1\nu_1^2+ c_2\nu_2^2)\nu_3^2}{(\widetilde c_1-\widetilde c_2)(1-\nu_3^2)}\sin\theta-\tfrac{\widetilde H\nu_3^2}{\widetilde c_1-\widetilde c_2},\\
	\widetilde c_1\widetilde\nu_1^2+\widetilde c_2\widetilde\nu_2^2&=-\tfrac{ H(1-\nu_3^2)+( c_1- c_2) \nu_1 \nu_2 \nu_3}{1-\widetilde \nu_3^2}\nu_3^2\sin\theta+\tfrac{( c_1\nu_1^2+ c_2\nu_2^2)\nu_3^2}{1-\nu_3^2}\cos\theta.
\end{aligned}
\end{equation}
The second equation in~\eqref{eqn:second-case:eqn6} along with $\widetilde\nu_1^2+\widetilde\nu_2^2=\nu_3^2$ allows us to get $\widetilde\nu_1^2,\widetilde\nu_2^2,\widetilde\nu_3^2$ in terms of $\nu_1,\nu_2,\nu_3$ (note that $\widetilde c_1\neq\widetilde c_2$ since we can assume we are not in an $\E(\kappa,\tau)$-space). These expressions can be plugged into the square of the first equation in~\eqref{eqn:second-case:eqn6} to reach a degree $6$ polynomial equation $p(\nu_1,\nu_2,\nu_3)=0$ with constant coefficients depending on $c_1,c_2,H$, $\widetilde c_1,\widetilde c_2,\widetilde H$ and $\theta$. The expression of $p$ is too unwieldy to write it here, but one can check that
\begin{equation}\label{eqn:second-case:eqn7}
p(\cos t,\sin t,0)=\alpha-(c_1\cos^2t+c_2\sin^2t)(c_1\cos^2t+c_2\sin^2t+\beta),
\end{equation}
for some constants $\alpha$ and $\beta$. Since $c_1\neq c_2$, the expression~\eqref{eqn:second-case:eqn7} is never a constant function of $t$. Elementary algebraic geometry shows that $p(\nu_1,\nu_2,\nu_3)=0$ implies that the Gauss map $g$ has not maximal rank, and consequently the subcase $(a_1,a_3)=(b_3,b_1)$ does not lead to any correspondence.

\subsubsection{Case \textsc{sum-bcv}} Assume that $c_3=c_1+c_2$ and $\widetilde c_1=\widetilde c_2=\frac{\widetilde\kappa}{2\widetilde\tau}$ with $\widetilde c_3=2\widetilde\tau\neq 0$, i.e., we have a correspondence that mixes spaces of \textsc{sum}-type and \textsc{bcv}-type.

The subcase $(a_1,a_3)=(b_1,b_3)$ gives $c_1c_2+H^2=\widetilde\tau^2+\widetilde H^2$ and $-c_1 c_2+ H^2=\widetilde\kappa-3\widetilde\tau^2+\widetilde H^2$ with $\widetilde\nu_3=\nu_3$ (after a isometry as in previous cases). In particular, we get $\widetilde\kappa-4\widetilde\tau^2=-2c_1c_2\neq 0$, whence $\widetilde T=\Rot_\theta(T)$ reads $\widetilde T_3=\Rot_\theta(T_3)$ at points where $\nu_3\neq 0$. Taking the divergence by means of Lemma~\ref{lemma:T-divergence}, we get
\begin{equation}\label{eqn:third-case:eqn1}
	2 H\nu_3+(c_2-c_1)\nu_1\nu_2=2\widetilde H \nu_3\cos\theta+2\widetilde\tau \nu_3\sin\theta.
\end{equation}
Since we can suppose that $c_2\neq c_1$, we conclude from~\eqref{eqn:third-case:eqn1} that the Gauss map $g$ has rank at most $1$, so there is no correspondence in this subcase.

Finally, assume $(a_1,a_3)=(b_3,b_1)$, so that $-c_1c_2+H^2=\widetilde\tau^2+\widetilde H^2$ and $c_1 c_2+ H^2=\widetilde\kappa-3\widetilde\tau^2+\widetilde H^2$ with $\widetilde\nu_ 3^2=1-\nu_3^2$. Since this implies that $\widetilde\kappa-4\widetilde\tau^2=2c_1c_2\neq 0$, the equality $\widetilde T=\Rot_\theta(T)$ is written as $\widetilde\nu_3\widetilde T_3=-\nu_3\Rot_\theta(T)$ and we also have $\widetilde\nu_3\nabla\widetilde\nu_3=-\nu_3\nabla\nu_3$. Proceeding as in the above cases, we can expand $\langle\widetilde\nu_3\nabla\widetilde\nu_3,X\rangle+\langle\nu_3\nabla\nu_3,X\rangle=0$ by means of~\ref{eqn:comp-v} to obtain
\begin{equation}\label{eqn:third-case:eqn2}
(HT_3+c_2\nu_2 T_1-c_1\nu_1T_2)\nu_3=-(\widetilde H\widetilde T_3-\widetilde\tau J\widetilde T_3)\widetilde\nu_3.
\end{equation}
As in the case \textsc{sum-sum}, we can use~\eqref{eqn:second-case:eqn5} and the fact that $\widetilde\nu_3\widetilde T_3=-\nu_3\Rot_{-\theta}(T_3)$ to rewrite~\eqref{eqn:third-case:eqn2} as a combination of $T_3$ and $JT_3$, which gives the following coefficient identifications, which are valid at points where $\nu_3\neq 1$:
\begin{equation}
	\begin{aligned}
		\tfrac{H(1-\nu_3^2)+(c_1-c_2)\nu_1\nu_2\nu_3}{1-\nu_3^2}&=\widetilde H\cos\theta+\widetilde\tau\sin\theta,\\
		\tfrac{c_1\nu_1^2+c_2\nu_2^2}{1-\nu_3^2}&=\widetilde H\sin\theta-\widetilde\tau\cos\theta.
	\end{aligned}
\end{equation}
Any of these two equations implies once again that the Gauss map $g$ has not maximal rank, so we conclude that this subcase does not yield any Lawson-type correspondence either. This finishes the proof of Theorem~\ref{thm:lawson}.

\end{document}